\documentclass{birkjour}

\usepackage{amsfonts,amsmath,amssymb,amstext}
\usepackage{amsthm}
\usepackage{graphicx}
\usepackage{subcaption}
\usepackage{appendix}

\usepackage{fullpage}

\usepackage[ruled,vlined]{algorithm2e}

\usepackage{tikz}

\usepackage[pdftex,bookmarks,colorlinks,breaklinks]{hyperref}  
\usepackage{color}

\definecolor{dullmagenta}{rgb}{0.4,0,0.4}   
\definecolor{darkblue}{rgb}{0,0,0.4}
\hypersetup{linkcolor=red,citecolor=blue,filecolor=dullmagenta,urlcolor=darkblue} 

\graphicspath{{figures/}}

\newtheorem{theorem}{Theorem}[section]
\newtheorem{lemma}[theorem]{Lemma}
\newtheorem{definition}[theorem]{Definition}
\newtheorem{corol}[theorem]{Corollary}
\newtheorem{prop}[theorem]{Proposition}

\newenvironment{remark}%
  {\par\medbreak\refstepcounter{theorem}%
    \noindent\textbf{Remark~\thetheorem. }}%
  {\par\medskip}

\newtheorem{question}[theorem]{Question}

\newcommand{\vz}[1]{\ensuremath{\mathbb{#1}}}

\newcommand{\R}{{\vz R}}
\newcommand{\N}{{\vz N}}

\newcommand{\dvg}{\text{div}\,}

\newcommand{\TV}{\text{TV}}
\newcommand{\TVa}{\text{TV}_{\text{a}}}
\newcommand{\sgn}{\text{sgn}}

\let\e\varepsilon

\newcounter{hdps}
\newenvironment{asm}[1]
  {\refstepcounter{hdps}%
    \begin{algorithm}\renewcommand\thealgocf{#1}}
  {\end{algorithm}\addtocounter{algocf}{-1}}

\title{Mean curvature, threshold dynamics, \\ and phase field theory on finite graphs}

\date{\today}

\author[Y. van Gennip]{Yves van Gennip}
\address{%
School of Mathematical Sciences \\
The University of Nottingham \\
University Park \\
Nottingham, NG7 2RD \\
UK
}
\email{y.vangennip@nottingham.ac.uk}

\author[N. Guillen]{Nestor Guillen}
\address{%
Department of Mathematics \\
University of California, Los Angeles \\
Los Angeles, CA 90095 \\
USA}
\email{nestor@math.ucla.edu}

\author[B. Osting]{Braxton Osting}
\address{%
Department of Mathematics \\
University of California, Los Angeles \\
Los Angeles, CA 90095 \\
USA}
\email{braxton@math.ucla.edu}

\author[A. L. Bertozzi]{Andrea L. Bertozzi}
\address{%
Department of Mathematics \\
University of California, Los Angeles \\
Los Angeles, CA 90095 \\
USA}
\email{bertozzi@math.ucla.edu}

\begin{document}

\maketitle

\begin{abstract}
In the continuum, close connections exist between mean curvature flow, the Allen-Cahn (AC) partial differential equation, and the Merriman-Bence-Osher (MBO) threshold dynamics scheme. Graph analogues of these processes have recently seen a rise in popularity as relaxations of NP-complete combinatorial problems, which demands deeper theoretical underpinnings of the graph processes. The aim of this paper is to introduce these graph processes in the light of their continuum counterparts, provide some background, prove the first results connecting them, illustrate these processes with examples and identify open questions for future study.

We derive a graph curvature from the graph cut function, the natural graph counterpart of total variation (perimeter). This derivation and the resulting curvature definition differ from those in earlier literature, where the continuum mean curvature is simply discretized, and bears many similarities to the continuum nonlocal curvature or nonlocal means formulation. This new graph curvature is not only relevant for graph MBO dynamics, but also appears in the variational formulation of a discrete time graph mean curvature flow.

We prove estimates showing that the dynamics are trivial for both MBO and AC evolutions if the  parameters (the time-step and diffuse interface scale, respectively) are sufficiently small (a phenomenon known as ``freezing'' or ``pinning'')  and also that the dynamics for MBO are nontrivial if the time step is large enough. These bounds are in terms of graph quantities such as the spectrum of the graph Laplacian and the graph curvature. Adapting a Lyapunov functional for the continuum MBO scheme to graphs, we prove that the graph MBO scheme converges to a stationary state in a finite number of iterations. Variations on this scheme have recently become popular in the literature as ways to minimize  (continuum) nonlocal total variation.
\end{abstract}



\clearpage
\tableofcontents

\clearpage
\section{Introduction}
Motion by mean curvature and flows involving mean curvature in general appear in many important \emph{continuum} models, including models coming from materials science \cite{Mullins64,Taylor92}, fluid dynamics \cite{HeleShaw1898}, and combustion \cite{Xin11,Peters00}. All such models involve a front propagating with a velocity depending on the mean curvature of the front. Recently, there has been an increasing interest in using ideas from continuum PDEs (related to  mean curvature) in \emph{discrete} applications such as image analysis, machine learning and data clustering \cite{BertozziFlenner12, vanGennipBertozzi12,MerkurjevKosticBertozzi2012,Garcia-CardonaMerkurjevBertozziFlennerPercus2013,HuLaurentPorterBertozzi13}. 

This paper  initiates a systematic study of the definition of mean curvature for vertex sets of an arbitrary graph $G=(V,E)$. We examine the effectiveness of the algorithms in the recent papers mentioned above
and on how they may be improved. The graphs considered are arbitrary graphs and are not necessarily obtained as the discretization of a continuum problem, so our perspective is only parallel to one that is purely motivated by numerical analysis. In particular, we do not assume an embedding of the graph in a  low dimensional space.

Of course, the various definitions of curvature in the (usual) continuum setting (see Appendix~\ref{sec:continuum} for a brief review and some references) motivate and inform our approach to defining the curvature of a vertex set $S\subset V$ using the discrete total variation norm and the discrete divergence of a ``normal'' edge flow. Since they are closely related to questions of mean curvature, the Allen-Cahn equation and the MBO scheme for arbitrary graphs $G=(V,E)$ arise naturally in the present investigation. Theoretical and numerical examples are used to highlight possible connections between all these concepts, leading to a number of open questions given in Section~\ref{sec:conjectures}.

\paragraph{Graphs Laplacians, Allen-Cahn, and MBO} Graph Laplacians are the central objects of study in spectral graph theory \cite{Chung:1997}. These graph operators share many properties with their continuum counterparts. The Allen-Cahn equation on the graph $V$ is defined in terms of the graph Laplacian, $\Delta$, and any  (typically bistable quartic) potential, $W$. One considers a phase field, $u\colon V \times \mathbb{R}_+\to\mathbb{R}$, solving the differential equation,
\[
\dot u = -\Delta u - \frac1\e W'(u).
\]
This nonlinear equation has received greater attention recently, spurred by some of the applications mentioned above. 
The graph Allen-Cahn equation was introduced in the context of data classification in \cite{BertozziFlenner12} and, in a number of examples, was shown to be both accurate and efficient.
As is well known, the continuum Allen-Cahn equation is the $L^2$ gradient flow associated to the Ginzburg-Landau functional. This is also true in the graph setting.
In \cite{vanGennipBertozzi12} it was shown that the graph Ginzburg-Landau functional
$\Gamma$-converges to the graph cut objective functional on graphs, if the characteristic length scale $\e$ goes to zero. Moreover, a relationship between the graph cut functional and the continuum total variation functional was given.
At the same time, the continuum Allen-Cahn solution is known to converge to mean curvature flow, when $\e \to 0$ \cite{BarlesSonerSouganidis93}. Furthermore, the mean curvature is directly related to the first variation of the total variation functional. In this paper we therefore study the graph Allen-Cahn equation and make connections between it and a graph cut derived `graph curvature'.

The third ingredient in this paper is the threshold dynamics or Merriman-Bence-Osher (MBO) algorithm on graphs. Its continuum counterpart was introduced in \cite{MBO1992,MBO1993} and consists of iteratively solving the heat equation 
for a short time, $\tau$, and  thresholding the result to an indicator function. It is known that, for short diffusion times $\tau$, this approximates mean curvature flow \cite{Evans93, BarlesGeorgelin95}. In this paper we therefore also study the connections between the graph MBO scheme, the graph AC equation, and graph curvature.

In a recent series of papers \cite{ElmoatazLezorayBougleux2008,TaElmoatazLezoray2008,DesquesnesElmoatazLezorayTa2010,DesquesnesElmoatazLezoray2011,TaElmoatazLezoray2011,DesquesnesElmoatazLezoray2012,ElmoatazDesquesnesLezoray2012,ElmoatazDesquenesLezoray_desquesnes-isbi2012,DesquesnesElmoatazLezoray2012_38Desquesness} Elmoataz et al. study partial differential equations and front propagation on graphs, mainly from a numerical point of view. In these papers the $1$-Laplacian on a graph is used as curvature, which differs from our approach. We  use the anisotropic graph total variation instead of the isotropic total variation (see Section~\ref{sec:setup} for definitions), since \cite{vanGennipBertozzi12} suggests that is the natural total variation on graphs.

A common obstacle, when transferring results and intuitions from the continuum to graphs, is the (implicit) lower bound on the accessible length scales on a graph. We show in Theorem~\ref{thm:pinningAC} and Theorem~\ref{thm:pinningMBO}, that if $\e$ or $\tau$ is too small, then the Allen-Cahn equation exhibits ``freezing'' (or ``pinning''), or the MBO scheme is stationary, respectively.
Hence the interesting dynamics happen at small but positive, $\e$ or $\tau$, rather than in the limits as $\e\to 0$ or $\tau\to 0$.
 Related is the lack of a chain rule for derivatives on graphs\footnote{That is, if $u\in \mathcal{V}$ and $f\colon \mathcal{V} \to \mathcal{V}$, then $\nabla f(u) \neq L_f \nabla u$ for any linear operator $L_f$.}, which can be traced back to the absence of `infinitesimal' length scales on a graph. As a consequence, the level set approach, which has proven very successful in describing continuum mean curvature flow, is not independent of the level set function on a graph.

\paragraph{New results.} The finite spectral radius of the Laplacian is used to derive explicit bounds on the parameters for both threshold dynamics (MBO) and the graph Allen-Cahn equation that guarantee ``freezing'' or ``pinning'' of the initial phase, a phenomenon that has been observed in numerical simulations and is well known for discretizations of the continuum processes \cite{MBO1992}\cite[Section 4.4]{MerrimanBenceOsher94}.

In the opposite direction, an argument based on the comparison principle is used to obtain a lower bound for the MBO time step that guarantees that a specific node of the phase changes in a single MBO iteration. This bound is given in terms of a new notion of mean curvature for general graphs, and as such, it is a ``local'' quantity (as opposed to one coming from spectral data). Such local bounds may be of use in developing adaptive time stepping schemes that complement the (spectral) adaptive schemes, such as those developed for discretizations of the continuum mean curvature flow \cite{Ruuth96,ZhangDu2009}. In this sense, introducing the graph mean curvature and highlighting its connection with subjects in continuum PDE (MBO, Ginzburg-Landau, and nonlocal mean curvature) and graph theory (graph cuts, connectivity) are the main contributions of this work.

The  results in Sections~\ref{sec:MBO}  and~\ref{sec:AC} and the numerical evidence and explicit examples in Section~\ref{sec:examples} suggest several open questions about the graph MBO scheme, the graph Allen-Cahn equation, and graph mean curvature, which are discussed in Section~\ref{sec:conjectures}. These are interesting questions for future work.

\paragraph{Outline.} In Section~\ref{sec:setup} the  relevant graph based calculus is introduced, setting the notation for the rest of the paper. In particular, the graph Laplacian and its basic properties are discussed.
Section~\ref{sec:curvature} discusses curvature and mean curvature flow on a graph.
Sections~\ref{sec:MBO} and~\ref{sec:AC} discuss the MBO scheme and Allen-Cahn equation on graphs, respectively, and sufficient conditions are given on the parameters to guarantee freezing or pinning of the initial conditions.
Section~\ref{sec:examples} explores the 
graph processes and concepts introduced in these previous three sections through theoretical  and computational  examples.
Finally, we conclude in Section~\ref{sec:conjectures}, with a discussion and a few open questions based on the new estimates and examples from previous sections.
In Appendix~\ref{sec:continuum}, we make some remarks on the continuum mean curvature flow.
Appendix~\ref{sec:formalfirstvar} discusses some similarities between the graph Laplacian, the graph 1-Laplacian, and the graph curvature.

\section{Setup}\label{sec:setup}

We work on a finite\footnote{In this paper, we are working with a fixed graph $G$ with a finite number of nodes. In no sense are we considering a sequence of graphs or taking a ``continuum limit''.},
undirected graph $G=(V,E)$ with vertex set\footnote{We will use the terms ``vertex'' and ``node'' interchangeably.} $V=\{i\}_{i=1}^n$ and edge set $E \subset V^2$.
The graph is weighted; each edge $(i,j)\in E$, incident on nodes $i$ and $j$, is assigned a weight $\omega_{ij}>0$.
Since the graph is undirected, $(i,j)$ is identified with $(j,i)$ and $\omega_{ij}=\omega_{ji}$.
To simplify notation we extend $\omega_{ij}$ to be zero  for all $(i, j)\in V^2$ which do not correspond to an edge. The degree of node $i$ is $d_i := \sum_{j\in V} \omega_{ij}$. Denote the maximal and minimal degrees by $ d_+ := \max_{i \in V} d_i$ and $d_- := \min_{i \in V} d_i$. We assume that $G$ has no isolated nodes, {\it i.e.}, $d_->0$. For each $i\in V$ we then have a non-empty set of neighbors $\mathcal{N}_i := \{j\in V: \omega_{ij}>0\}$. We also assume $G$ has no self-loops, ${\it i.e.}, \omega_{ii}=0$. In particular $i\not \in \mathcal{N}_i$.

Let $\mathcal{V}$ be the space of all functions $V\to \R$ and $\mathcal{E}$ the space of all skew-symmetric\footnote{The necessity of skew-symmetry may not be obvious at this point, but it is a common requirement for consistency of certain concepts in discrete calculus, see {\it e.g.}, \cite{FriedmanTillich04,GilboaOsher2009,GradyPolimeni10, CandoganMenacheOzdaglarParrilo11}. See also Remark~\ref{rem:isotropicTV}.} functions $E \to \R$. Again to simplify notation, we extend each $\varphi\in\mathcal{E}$ to a function $\varphi\colon V^2 \to \R$ by setting $\varphi_{ij} = 0$ if $(i,j)\not\in E$. As justified in earlier work \cite{vanGennipBertozzi12} we introduce the following inner products and operators for parameters $q\in[1/2,1]$ and $r\in[0,1]$:
\begin{align*}
&\langle u, v \rangle_{\mathcal{V}} := \sum_{i\in V} u_i v_i d_i^r, \quad
\langle \varphi, \phi \rangle_{\mathcal{E}} := \frac12 \sum_{i,j\in V} \varphi_{ij} \phi_{ij} \omega_{ij}^{2q-1},\\
&(\varphi\cdot\phi)_i := \frac12 \sum_{j\in V} \varphi_{ij} \phi_{ij} \omega_{ij}^{2q-1}, \quad
(\nabla u)_{ij} := \omega_{ij}^{1-q} (u_j-u_i),\\
&(\dvg \varphi)_i := \frac1{d_i^r} \sum_{j\in V} \omega_{ij}^q \varphi_{ji}.
\end{align*}
Note that in the sum in $\langle \varphi, \phi \rangle_{\mathcal{E}}$ the indices $i$ and $j$ both run over all nodes. The edges $(i,j)$ and $(j,i)$ are counted separately, hence the correction factor $\frac12$. Note that the powers $2q-1$ and $1-q$ in the $\mathcal{E}$ inner product and `dot product' and in the gradient, are zero for the admissible choices $q=\frac12$ and $q=1$ respectively. In these cases we define $\omega_{ij}^0=0$ whenever $\omega_{ij}=0$, so as not to make the inner product, `dot product', or gradient, nonlocal on the graph. The inner products on $\mathcal{V}$ and $\mathcal{E}$\footnote{Note that $\langle \cdot, \cdot \rangle_{\mathcal{E}}$ is indeed an inner product on the space of (skew-symmetric) functions $E\to \R$, but not for the space of functions $V^2 \to \R$, because for those functions the `inner product' can be zero for nontrivial functions.} are analogous to a weighted $L^2$ inner products in the continuum case, while the `dot product' inner product $(\varphi \cdot \phi)_i$ is analogous to a weighted dot product between vector(field)s (without the integration of the $L^2$ inner product). A direct computation shows that $\dvg$ and $\nabla$ are adjoints with respect to $\langle \cdot,\cdot\rangle_{\mathcal{V}}$ and $\langle\cdot,\cdot\rangle_{\mathcal{E}}$, namely for $u \in \mathcal{V}$ and $\phi \in \mathcal{E}$ we have

    \begin{equation*}
         \langle \nabla u,\phi\rangle_{\mathcal{E}} = \langle u, \dvg\phi\rangle_{\mathcal{V}}.	
    \end{equation*}
The characteristic function of a node set $S\subset V$ is $\chi_S\in\mathcal{V}$, defined via $(\chi_S)_i := \begin{cases} 1 &\text{ if } i\in S,\\ 0 &\text{ if } i\not\in S.\end{cases}$

This leads to the following associated norms, Laplacians, set measures, and total variation functionals:
\begin{itemize}
\item Inner product norms, $\|u\|_{\mathcal{V}} := \sqrt{\langle u,u\rangle_{\mathcal{V}}}$ and $\|\varphi\|_{\mathcal{E}} := \sqrt{\langle \varphi, \varphi\rangle_{\mathcal{E}}}$.

\item Maximum norms\footnote{To justify these definitions and
convince ourselves that there should be no $\omega_{ij}$ or $d_i$
included in the maximum norms we define $\|\varphi\|_{\mathcal{E}, p}^p :=
\frac12 \sum_{i,j\in V} \varphi_{ij}^p \omega_{ij}^{2q-1}$.
Adapting the proofs in the continuum case, \textit{e.g.},
\cite[Theorems 2.3 and 2.8]{Adams75}, to the graph situation
we can prove a H\"older inequality $\displaystyle
\|\varphi \phi\|_{\mathcal{E}, 1} \leq \|\varphi\|_{\mathcal{E},p}
\|\phi\|_{\mathcal{E},p'}$ for $1<p, p'< \infty$ such that
$\frac1p+\frac1{p'}=1$, an embedding theorem of the form
$\displaystyle \|\varphi\|_{\mathcal{E},p} \leq \left( \frac12 \sum_{i,j\in V}
 \omega_{ij}^{2q-1}\right)^{\frac1p-\frac1s} \|\varphi\|_{\mathcal{E},s}$
for $1 \leq p \leq s \leq \infty$, and the limit $\displaystyle
\underset{p\to\infty}\lim\, \|\varphi\|_{\mathcal{E},p} =
\|\varphi\|_{\mathcal{E}, \infty}$. A similar result holds for the norms
on $\mathcal{V}$.}
, $\|u\|_{\mathcal{V},\infty} := \max\{|u_i|\colon i\in V\}$ and $\|\varphi\|_{\mathcal{E},\infty} := \max\{|\varphi_{ij}|\colon i,j \in V\}$.

\item The norm corresponding to the dot product $|\varphi|_i := \sqrt{( \varphi \cdot \varphi)_i}$. Note that $|\phi| \in \mathcal{V}$.

\item The Dirichlet energy
$$
\frac12 \|\nabla u\|_{\mathcal{E}}^2 = \frac{1}{4} \sum_{i,j\in V} \omega_{ij}(u_i - u_j)^2.
$$

\item The graph Laplacian $\Delta := \dvg\circ\nabla \colon \mathcal{V} \to \mathcal{V}$. So
\[
(\Delta u)_i := d_i^{-r} \sum_{j\in V} \omega_{ij} (u_i-u_j).
\]
It is worth noting that the sign convention for the graph Laplacian is opposite to that used for the continuum Laplacian in most of the 
PDE literature (in particular, the graph Laplacian $\Delta$ is a positive semidefinite  operator).

When $r=0$, $\Delta$ is referred to as the unnormalized weighted graph Laplacian. When $r=1$, $\Delta$ is referred to as the asymmetric normalized graph Laplacian or random walk graph Laplacian. Another Laplacian, often encountered in the spectral graph theory literature, is the symmetric normalized graph Laplacian. This one falls outside the scope of the current setup and will not be considered in this paper. For general references on the graph Laplacian, consult \cite{Mohar:1991,Chung:1997,Luxburg:2007,Biyikoglu:2007}.

\item For $S\subset V$, the set measures
\begin{align*}
\mathrm{vol} \ S & = \sum_{i\in S} d_i^r = \| \chi_S \|_{\mathcal{V}}^2, \\
|S| &= \text{number of vertices in $S$}.
\end{align*}
Note that $|S|$ is just a special case of $\mathrm{vol}\ S$, for $r=0$ (recall we assume $d_->0$).

\item The
anisotropic total variation
$\TV^q_a\colon \mathcal{V}\to \R$:
\begin{align*}
\TVa^q(u) &:= \max\{ \langle \dvg \varphi, u\rangle_{\mathcal{V}} \colon \varphi\in \mathcal{E},\,\, \|\varphi\|_{\mathcal{E},\infty}\leq 1\}\\
&=  \langle \nabla u, \sgn(\nabla u)\rangle_{\mathcal{E}} = \frac12 \sum_{i,j\in V} \omega_{ij}^q |u_i-u_j|.
\end{align*}
Here, the signum function is understood to act element-wise on the elements of $\nabla u$. For the isotropic total variation on graphs, see Remark~\ref{rem:isotropicTV}.
The maximum in the definition of $\TVa^q$ is achieved by
\begin{equation}\label{eq:varphiaTV}
\varphi = \varphi^a := \sgn(\nabla u).
\end{equation}
Note that the values
$\varphi^a$ takes on the set $\{\nabla u = 0\}$ are irrelevant for achieving the maximum, hence this function is not uniquely determined.
\end{itemize}

The anisotropic total variation of the indicator function\footnote{For $\chi_V$, the indicator function of the full node set, we also write the constant function $1$.} for the set $S\subset V$, denoted $\chi_S$, is given by
\begin{equation}
\label{eq:TVofS}
\TVa^q(\chi_S) = \sum_{i\in S, j\in S^c} \omega_{ij}^q.
\end{equation}
Thus, the total variation of a set $S$ is equivalent to the graph cut between $S$ and $S^c:= V\setminus S$ which is used in graph theory and spectral clustering \cite{ShiMalik00}. For future reference it is useful to note that
\begin{equation}\label{eq:TVusingDelta}
\TVa^1(\chi_S) = \langle \nabla \chi_S, \nabla \chi_S \rangle_{\mathcal E} =  \langle \chi_S,  \Delta \chi_S\rangle_{\mathcal V}.
\end{equation}

\begin{remark}\label{rem:isotropicTV}
One can also define an isotropic total variation on graphs $\TV\colon \mathcal{V}\to \R$:
\[
\TV(u) := \sum_{i\in V}  | \nabla u |_i = \frac{\sqrt{2}}{2} \sum_{i\in V} \sqrt{\sum_{j\in V} \omega_{ij} (u_i-u_j)^2}.
\]
The isotropic total variation also has a `maximum formulation', if we are willing to let go of the skew-symmetry condition for functions in $\mathcal{E}$. To be precise, define the extended set of edge functions $\mathcal{E}_e$ as the set of all functions $E
\to \R$, and extend the definition of divergence (compatible with the earlier definition) to functions $\varphi \in \mathcal{E}_e$:
\[
(\dvg \varphi)_i := \frac1{2d_i^r} \sum_{j\in V} \omega_{ij}^q (\varphi_{ji}-\varphi_{ij}).
\]
Using the same inner product on $\mathcal{E}_e$ as on $\mathcal{E}$, this divergence is again the adjoint of the gradient $\nabla$.
Then we can write
\[
\TV(u) =  \max\{ \langle \dvg \varphi, u\rangle_{\mathcal{V}} \colon
\varphi\in \mathcal{E},\,\, \underset{i\in V}\max\,\, |\varphi |_i \leq 1\}.
\]
The maximum is achieved by taking
\[
\varphi_{ij} = \varphi_{ij}^{TV} := \begin{cases}
\frac{(\nabla u)_{ij}}{ | \nabla u |_i} & |\nabla u |_i \neq 0 \\
0 & |\nabla u |_i=0.
\end{cases}
\]
As in the anisotropic case, the values $\varphi^{TV}$ takes on the set $\{\nabla u = 0\}$ are irrelevant for achieving the maximum, hence this function is not uniquely determined. The quantity $\text{div} \frac{\nabla u }{| \nabla u |}$ is often referred to as the 1-Laplacian of $u$.
\end{remark}

\begin{lemma} \label{lem:normEquiv}
The norms $\|\cdot \|_{\mathcal V}$ and $\|\cdot \|_{\mathcal V,\infty}$ are equivalent, with optimal constants given by
$$
d_-^\frac{r}2 \| u \|_{\mathcal V,\infty} \leq \| u \|_{\mathcal V} \leq \sqrt{ \mathrm{vol} \ V}\|u \|_{\mathcal V,\infty}.
$$
\end{lemma}
\begin{proof}
We compute
$\| u \|_{\mathcal V}^2 = \sum_{i\in V} d_i^r u_i^2 \leq \max_{i\in V} u_i^2 \sum_{i\in V} d_i^r
= (\mathrm{vol} \ V) \ \|u \|_{\mathcal V,\infty}^2,
$
which is saturated if $u=\chi_V$.

Also,
$ \| u \|_{\mathcal V}^2 = \sum_{i\in V} d_i^r u_i^2  \geq d_-^r \max_{i \in V} u_i^2
= d_-^r \|u \|_{\mathcal V,\infty}^2 $. If $j\in V$ is such that $d_j = d_-$, this bound is attained for $u = \chi_{\{j\}}$.
\end{proof}

Next we  recall the definitions of node set boundaries and (signed) graph distance.

\begin{definition}\label{def:distance}
	For $j\in \mathcal{N}_i$, we define $d^G_{ij} := \omega_{ij}^{q-1}$, and we set $d^G(i,i) := 0$. A path on $V$ is a sequence $\gamma=\{i_1,i_2,...,i_m\}$ for some $m\in\mathbb{N}$ such that $i_{k+1} \in \mathcal{N}_{i_k}$ for each $k\in \{1, \ldots, m-1\}$. Given a path $\gamma=\{i_1,...,i_m\}$, its length is defined as
	\begin{equation*}
	     |\gamma|:=\sum\limits_{i=1}^{m-1}d^G_{i_ki_{k+1}}.	
	\end{equation*}	
	Then, the graph distance between arbitrary $i, j \in V$ is given by
\[
d^G_{ij} := \min \limits_{\gamma} |\gamma|
\]
where the minimum is taken over all paths $\gamma$ with $i_1=i,i_N=j$. In other words, $d^G_{ij}$ is the minimal distance to go from node $i$ to node $j$, traveling only via existing edges, where each edge represents a distance $\omega_{ij}^{q-1}$.
For a given set $S\subset V$, we define the graph distance to $S$ at each node as the minimal graph distance to a node in $S$:
\[
d^S_i := \min_{j\in S} d^G_{ij}.
\]
\end{definition}
As argued in, for example, \cite[Section 3.1, Example 2]{ManfrediObermanSviridov12}, $d^S$ is the solution $u\in\mathcal{V}$ to the equation
\begin{equation}\label{eq:eikonal}
\begin{cases}
\min_{j\in \mathcal{N}_i} (\nabla u)_{ij} + 1 = 0 & \text{if } i\in V\setminus S,\\
u_i = 0 & \text{if } i\in S.
\end{cases}
\end{equation}

\begin{definition}\label{def:boundary}
The boundary of $S\subset V$ is\footnote{Similarly, by changing the ``strictly less than'' inequalities into ``strictly larger than'' inequalities the boundary $\partial (S^c)$ of the set $S^c$ is defined. The reduced boundary of $S$ can be defined as the following subset of $\partial S$ (compare with the continuum case in \cite[Definition 3.54]{AmbrosioFuscoPallara00}):
$$
\partial^*S := \{i \in S\colon \exists ! j\in V\colon (\nabla \chi_S)_{ij} < 0\},
$$
and again similarly for $\partial^* S^c$.
}
\[
\partial S := \{i\in S\colon \exists j \in V \text{ s.t. } (\nabla \chi_S)_{ij} < 0 \}.
\]
\end{definition}
Note that $\partial S \subset S$. Alternative definitions appear in the literature in which $\partial S \subset S^c$.

\subsection{Basic spectral properties of the graph Laplacian, $\Delta$}\label{sec:basicprops}
In this section, we collect a number of spectral properties of the graph Laplacian $\Delta\colon \mathcal{V}\to \mathcal{V}$. Further discussion and details for the special cases $r=0$ and $r=1$ can be found in \cite{Mohar:1991,Chung:1997,Luxburg:2007,Biyikoglu:2007}, from which our presentation follows.

Note that $\Delta\colon \mathcal{V} \to \mathcal{V}$ is a self-adjoint operator in the $\mathcal{V}$ norm. For $u \in \mathcal{V} \setminus \{ 0 \}$,  the \emph{Rayleigh quotient} $R \colon \mathcal V \to \mathbb R$ is defined as
$$
R(u) := \frac{ \langle u, \Delta u\rangle_\mathcal V }{ \| u \|_\mathcal{V}^2} = \frac{ \| \nabla u \|_\mathcal{E}^2 }{ \| u \|_\mathcal{V}^2}.
$$
The eigenvalues of the graph Laplacian, $\Delta$, are then defined via the variational formulation,

\begin{align}
\label{eq:eigDef}
\lambda_k =
\min_{\substack{ F_k \subset \mathcal V \\ \text{subspace of dim $k$}}}
\max_{ u \in F_k \setminus \{ 0\} }
R(u) .
\end{align}
The minimum in \eqref{eq:eigDef}, is attained when $F_k$ is spanned by the first $k$ eigenfunctions, {\it i.e.}, the eigenfunctions corresponding to the $k$ smallest eigenvalues, counting multiplicities.
In particular, there are $n$ non-negative real eigenvalues (counted with multiplicity), denoted $\{ \lambda_k \}_{k=1}^n$.
If we denote the span of the first $k-1$ eigenfunctions by $\hat{F}_{k-1} = span(\{ u_i\}_{i=1}^{k-1})$, then \eqref{eq:eigDef} can be rewritten
\begin{align}
\label{eq:eigDef2}
\lambda_k =
\min_{\substack{ u \in \mathcal V \setminus \{ 0 \} \\ u \perp_\mathcal{V} \hat F_{k-1}  }}
R(u) .
\end{align}
where $u \perp_\mathcal{V} \hat F_{k-1}$ indicates that $u$ is  orthogonal (in the sense of  $\langle\cdot , \cdot \rangle_\mathcal V$) to $u_i$, for $i\in\{1,\ldots, k-1\}$.
Taking variations of the Rayleigh quotient with respect to $u$, we find that $(\lambda_k,u_k)$ satisfies \eqref{eq:eigDef2} if and only if $u \perp_\mathcal{V} \hat F_{k-1}$ and, for all $v \in \mathcal V$,
\begin{align}
\label{eq:eigDef3}
\langle \Delta u_k , v \rangle_\mathcal V = \lambda_k \langle u_k, v \rangle_\mathcal V.
\end{align}
Finally, unwinding the definitions, we find that \eqref{eq:eigDef3} is equivalent to the matrix eigenvalue problem
\begin{align}
\label{eq:eigDef4}
L x = \lambda x \qquad \text{where } L = D^{-r} [ D-A], \ \ x\in \mathbb R^n, \  \ x^t D^r x = 1,
\end{align}
where $A_{ij} = \omega_{ij}$ and $D_{ii} = d_i$ is a diagonal matrix\footnote{Note that here we use the fact that there are no isolated vertices, {\it i.e.}, $d_i >0$ for all $i\in V$.}.
We remark that the change of variables, $y = D^{r/2}x$, in \eqref{eq:eigDef4} gives the standard eigenvalue problem
$$
 [D^{1-r} - D^{-r/2} A D^{-r/2}]y = \lambda y  \qquad y\in \mathbb R^n, \  \ y^t  y = 1.
$$
Recall that the \emph{spectral radius} $\rho$ of $\Delta$ is defined as the maximum of the absolute values of the eigenvalues of $\Delta$,
$$\rho(\Delta) := \max_{1\leq i\leq n} \lambda_i = \lambda_n = \sup_{u \in \mathcal{V}\setminus \{0\} } R(u).$$
    \begin{lemma}[Spectral properties of the graph Laplacian, $\Delta$] \label{lem:Spectral}
    The following properties are satisfied:
    \begin{enumerate}
\item[(a)] The smallest eigenvalue is $\lambda_1=0$. The multiplicity of $\lambda_1=0$ is the number of connected components of the graph and the associated eigenspace is spanned by set of indicator vectors for each connected component. If there is only one connected component, $\lambda_1$ is simple and the first (unnormalized) eigenfunction is  $u_1 = 1 = \chi_V$.

\item[(b)] The operator norm of $\Delta$, $\| \Delta \|_{\mathcal V} := \sup_{u \neq 0} \frac{\| \Delta u \|_{\mathcal V} }{\| u \|_{\mathcal V}}$, and the spectral radius are equal, $\| \Delta \|_{\mathcal V} = \rho(\Delta)$. This implies that, for all $u\in \mathcal{V}$,
\begin{align*}
\| \Delta u\|_{\mathcal V}  \leq \rho(\Delta) \ \| u\|_{\mathcal V}.
\end{align*}

\item[(c)] The trace satisfies $ \mathrm{tr} \ \Delta = \sum_{k=1}^n \lambda_k = \sum_{i \in V} d_i^{1-r}$. Consequently,
\begin{align}
\label{eq:mean bounds lambda}
&\lambda_2 \leq \frac{1}{n-1}\sum_{i \in V} d_i^{1-r} \leq \frac{n  \ d_+^{1-r}}{n-1}
\qquad \mathrm{and} \notag \\
&\lambda_n \geq \frac{1}{n-1}\sum_{i \in V} d_i^{1-r} \geq \frac{n  \ d_-^{1-r}}{n-1}.
\end{align}

\item[(d)] If $G$ is not a complete graph then
$$
\lambda_2 \leq \min_{(i,j) \notin E} \ \frac{d_i d_j^{2r} + d_i^{2r} d_j}{d_i^r d_j^{2r} + d_i^{2r} d_j^r}.
$$
\item[(e)] The second eigenvalue satisfies
\[
\lambda_2 \leq \min_{S\subset V} \ \frac{( \mathrm{vol} \ V) TV_a^1(\chi_S) }{ (\mathrm{vol} \ S) (\mathrm{vol} \ S^c)} \leq 2 \min_{S\subset V} \  \frac{TV_a^1(\chi_S) }{\min(\mathrm{vol} \ S,\mathrm{vol} \ S^c)}.
\]

    \item[(f)]          The spectral radius of $\Delta$ satisfies $\rho\leq 2 \ d_+^{1-r}$.
    \end{enumerate}
    \end{lemma}

\bigskip

\begin{proof} (a) These properties  follow  directly from \eqref{eq:eigDef2}.

(b) Noting that $\Delta$ is a self-adjoint operator, a proof can be found in, for example, \cite[Thm. VI.6]{RS1}.

(c) Because the trace of the operator $\Delta$ is equal to the trace of its matrix representation, we have $ \mathrm{tr} \ \Delta = \mathrm{tr} \ L = \sum_{k=1}^n \lambda_k$. Since we assume there are no self-loops in the graph, $\mathrm{tr}\ D^{-r} A = 0$, hence $\mathrm{tr} \ L = \mathrm{tr} \ D^{1-r}$. Equation \eqref{eq:mean bounds lambda} follows from the fact that $ \lambda_1 = 0$ and the maximum (minimum) of a set is greater (less) than or equal to the mean of the set.

(d) If $G$ is not a complete graph, then there exists an $(a,b) \notin E$. We define the test function $v\in \mathcal V$,
$$
v_i = \begin{cases}
d_b^r & i = a\\
-d_a^r & i = b\\
0 & \text{otherwise.}
\end{cases}
$$
Note that $\langle v, 1\rangle_\mathcal V = 0$. The desired upper bound then follows from \eqref{eq:eigDef2}.

(e) For $S\in V$, define the test function $v\in \mathcal V$,
$$
v_i = \begin{cases}
 \mathrm{vol} \ S^c & i \in S\\
- \mathrm{vol} \ S & i \in S^c.
\end{cases}
$$
Then $\langle v, 1\rangle_\mathcal V = 0$ and $\| v\|_{\mathcal V}^2 = (\mathrm{vol} \ S^c) (\mathrm{vol} \ S) \mathrm{vol} \ V$. Using \eqref{eq:TVofS}, we compute $\| \nabla v \|_{\mathcal E}^2 = (\mathrm{vol} \ V)^2 \TVa^1(\chi_S) $. The first inequality then follows from \eqref{eq:eigDef2}.
For the second inequality,
\[
\frac{( \mathrm{vol} \ V) TV_a^1(\chi_S) }{ (\mathrm{vol} \ S) (\mathrm{vol} \ S^c)} = \frac{TV_a^1(\chi_S) }{ \mathrm{vol} \ S} + \frac{TV_a^1(\chi_S) }{ \mathrm{vol} \ S^c} \leq 2 \frac{TV_a^1(\chi_S) }{\min(\mathrm{vol} \ S,\mathrm{vol} \ S^c)}.
\]

(f) Using the identity $(a-b)^2 \leq 2(a^2 + b^2)$, we compute
\begin{align*}
\rho(\Delta) &=  \sup_{u \in \mathcal{V}\setminus \{0\} } R(u)
= \sup_{u \in \mathcal{V}\setminus \{0\} } \frac{1}{2} \frac{\sum_{ij} w_{ij} (u_i - u_j)^2}{\sum_i d_i^r u_i^2} \\
&\leq  \sup_{u \in \mathcal{V}\setminus \{0\} } \frac{\sum_{ij} w_{ij} (u_i^2 + u_j^2)}{\sum_i d_i^r u_i^2}\\
&=  \sup_{u \in \mathcal{V}\setminus \{0\} } 2 \frac{ \sum_i d_i u_i^2 }{ \sum_i d_i^r u_i^2}
\end{align*}
If $j\in V$ is such that $d_j = d_+$, then the supremum is attained by the vector $u = \chi_{ \{j\} }$ and the result follows.
\end{proof}

The following lemma states properties of the diffusion operator $e^{-t\Delta}\colon \mathcal V \to \mathcal V$.
\begin{lemma}[Diffusion on a graph] \label{lem:DiffusionGraph}
Let $u(t) := e^{-t\Delta} u_0$ for $t\geq 0$ denote the evolution of $u_0 \in \mathcal{V}$ by the diffusion operator. The following properties hold.
\begin{enumerate}
\item[(a)] The mass,
\begin{equation}
\label{eq:mass}
M(u) :=  \langle u, \chi_V \rangle_{\mathcal{V}} = \sum_{i\in V} u_i d_i^r,
\end{equation}
is conserved, {\it i.e.},
$M(u(t)) = M(u_0)$ for all $t\geq 0$.
\item[(b)] $\frac{d}{dt} \| u \|_{\mathcal V}^2 = - 2 \| \nabla u \|_{\mathcal E}^2 \leq 0$. In particular, $ \|e^{-\Delta t}  u_0 \|_{\mathcal V} \leq  \|  u_0 \|_{\mathcal V} $.
 \item[(c)] Let the mass, $M$, be defined as in \eqref{eq:mass}, $\lambda_2$ be the second eigenvalue of the graph Laplacian, and $\e > 0$. Assume the graph is connected.
If $\tau > \frac{1}{\lambda_2} \log \left(\e^{-1} \ d_-^{-\frac{r}{2}}  \  \| u_0 - (\mathrm{vol} \ V)^{-1} M \|_{\mathcal V}  \right)$, then
$$
\| u(t) - (\mathrm{vol} \ V)^{-1} M \|_{\mathcal{V},\infty} \leq \e, \qquad \text{for all } t>\tau.
$$
 \item[(d)] (Comparison Principle) If, for all $j\in V$, $(u_0)_j \leq (v_0)_j$, then $(e^{-t\Delta} u_0 )_j  \leq (e^{-t \Delta} v_0 )_j$, for all $j\in V$ and $t\geq0$.
In particular,  $\|e^{-t \Delta } u_0 \|_{\mathcal{V},\infty} \leq \|u_0\|_{\mathcal{V},\infty}$.

If $V$ is connected, the strong comparison principle holds: If, for all $j\in V$, $(u_0)_j \leq (v_0)_j$, and for some $j_0\in V$, $(u_0)_{j_0}<(v_0)_{j_0}$, then, for all $k\in V$ and $t>0$, $(e^{-t\Delta} u_0 )_k  < (e^{-t \Delta} v_0 )_k$.
  \end{enumerate}
\end{lemma}
\begin{proof}
(a) We compute
\[
\frac{d}{dt} M(u) = \langle \dot u, \chi_V \rangle_{\mathcal{V}}
= - \langle \Delta u, \chi_V \rangle_{\mathcal{V}}
= - \langle \nabla u, \nabla \chi_V \rangle_{\mathcal{E}} = 0.
\]

(b) We compute
\[
\frac{1}{2} \frac{d}{dt} \| u \|_{\mathcal V}^2 = \langle u, \dot u \rangle_{\mathcal V}
= - \langle u, \Delta u \rangle_{\mathcal V}
= - \langle \nabla u, \nabla u \rangle_{\mathcal E}
= - \| \nabla u \|_{\mathcal E}^2.
\]

(c) If $\{(\lambda_j, v_j) \}_{j=1}^n$ denote the eigenpairs of the graph Laplacian with $\mathcal{V}$-normalized eigenvectors, then the spectral decomposition of $u$ is given by
\begin{equation}
\label{eq:specDecomp}
u(t) = \sum_{j=1}^n e^{-\lambda_j t} \ \langle u_0, v_j \rangle_{\mathcal V} \ v_j.
 \end{equation}
Recalling from Lemma~\ref{lem:Spectral} that $\lambda_1=0$ and $v_1 = (\mathrm{vol} \ V)^{-\frac{1}{2}} \chi_V$ and using \eqref{eq:specDecomp}, we compute
\begin{align*}
\| u - (\mathrm{vol} \ V)^{-1} M   \|_{\mathcal{V}} &= \| \sum_{j>1} e^{-\lambda_j t} \langle u_0, v_j \rangle_{\mathcal V} \ v_j  \|_{\mathcal{V}}\\
&\leq e^{-\lambda_2 t}  \|u_0 - (\mathrm{vol} \ V)^{-1}  M  \|_{\mathcal{V}} .
\end{align*}
But by Lemma~\ref{lem:normEquiv}, this implies
$$
\| u - (\mathrm{vol} \ V)^{-1}  M   \|_{\mathcal{V}, \infty} \leq d_-^{-\frac{r}{2}} e^{-\lambda_2 t}  \|u_0 - (\mathrm{vol} \ V)^{-1}  M  \|_{\mathcal{V}}.
$$
The  result now follows, since by Lemma~\ref{lem:Spectral}, $\lambda_2>0$ for a connected graph.

         (d) If $u_0\equiv v_0$, then $u(t)\equiv v(t)$, for all $t>0$, by the uniqueness theorem for ordinary differential equations. In this case there is nothing more to prove. Moreover, by repeating the argument on each connected component we may assume without loss of generality that the entire graph is connected.

         Let $u_0,v_0$ be such that, for all $j\in V$, we have $(u_0)_j \leq (v_0)_j$, and for some $j_0\in V$, $(u_0)_{j_0}<(v_0)_{j_0}$. We will show that in this case $u_j(t)<v_j(t)$, for every $j\in V$ and all $t>0$, which proves the strong comparison principle and in particular the comparison principle.

         Arguing by contradiction, suppose that $u_j(t)\geq v_j(t)$ for some $t$ and some $j$. Let $t_0$ be the last time $v(t)$  lies everywhere above $u(t)$, that is
         \begin{equation*}
              t_0 := \sup \{ \;t\geq 0: \forall\; s\in (0,t) \text{ we have } \forall j\ u_j(s)< v_j(s)\}.         	
         \end{equation*}
         By our assumption we have that $0\leq t_0<\infty$. Then, from the definition of $t_0$ there is some $k\in V$ such that $u_k(t_0) = v_k(t_0)$ and $\dot u_k(t_0) \geq \dot v_k(t_0)$. Moreover, again due to the definition of $t_0$, $u_j(t_0) \leq v_j(t_0)$, for all $j\in V$. This shows that if $u_i(t_0) < v_i(t_0)$ for some neighbor $i$ of $k$, then $(\Delta u(t_0))_k > (\Delta v(t_0))_k$ and
         \begin{equation*}
              0 = \dot v_k(t_0) +(\Delta v(t_0))_k < \dot u_k(t_0) +(\Delta u(t_0))_k = 0,
         \end{equation*}
         which is a contradiction. We conclude that $u(t_0)=v(t_0)$ at all neighbors of $k$, and by iterating the above argument we get in fact that $u(t_0) = v(t_0)$ at all nodes of $V$, since we are dealing with the case where $V$ is connected. By the uniqueness theorem for ordinary differential equations we conclude that $u_0 = v_0$ at all nodes, which gives a contradiction, and the strong comparison principle is proved.
    \end{proof}

\begin{remark} The dependence of the convergence of $u(t)$ to the steady state $ (\mathrm{vol} \ V)^{-1} M \chi_V$ on the second Laplacian eigenvalue, $\lambda_2$, in Lemma~\ref{lem:DiffusionGraph}(c), is related to the rate of convergence of a Markov process on a graph to the uniform distribution \cite{Sun2004}. Due to this property, $\lambda_2^{-1}$ is sometimes referred to as the \emph{mixing-time} for a graph. The second eigenvalue $\lambda_2$ is also referred to as the \emph{algebraic connectivity} or \emph{Fiedler value} for a graph \cite{Fiedler:1973}, and plays an important role in many applications. The robustness of a network to node/edge failures is highly dependent on the algebraic connectivity of the graph. In the ``chip-firing game'' of Bj\"orner, Lov\'asz, and Shor, the algebraic connectivity dictates the length of a terminating game  \cite{Bjorner:1991}. The algebraic connectivity is also related to the informativeness of a least-squares ranking on a graph  \cite{Osting:2012b}. Consequently, algebraic connectivity is a measure of performance for the convergence rate in sensor networks, data fusion, load balancing, and consensus problems \cite{Olfati-Saber:2007}.
\end{remark}

\subsection{Relation between graph Laplacians and balanced graph cuts} \label{sec:BalancedGraphCut}

In spectral graph theory there are some well known connections between the various graph Laplacians and different normalizations of the graph cut $\TVa^1(\chi_S) = \sum_{i\in S, j\in S^c} \omega_{ij}$ from \eqref{eq:TVofS}. For example, when $r=0$ (and hence $\mathrm{vol}\ S = |S|$), the quantity $\displaystyle \frac{TV_a^1(\chi_S) }{\min(\mathrm{vol} \ S,\mathrm{vol} \ S^c)}$ in Lemma~\ref{lem:Spectral}(e) is the {\it Cheeger cut} and its minimum over $S\subset V$ is the {\it Cheeger constant}, {\it e.g.},  \cite{ShiMalik00,SzlamBresson10}. Let $S_1, \ldots, S_k$ be a partition of all the nodes $V$, for a given integer $k$ and define the balanced graph cut
\[
C_r(S_1, \ldots, S_k) := \sum_{i=1}^k \frac{\TVa^1(\chi_{S_i})}{\mathrm{vol}\ S_i}.
\]
We use the subscript $r$ to remind us that $\mathrm{vol} \ S_i$ depends on $r$. For $r=0$ this quantity is known in the literature as the {\it ratio cut}\footnote{Confusingly, the ratio cut is sometimes also called {\it average cut}, and the Cheeger cut $\frac{\TVa^1(\chi_S)}{\min(|S|, |S^c|)}$ is sometimes called ratio cut.}, for $r=1$ as the {\it normalized cut} \cite{ShiMalik00,Luxburg:2007}. These quantities are introduced, because minimization of the graph cut, without a balancing term in the denominator, often leads to a partition with many singleton sets, which is typically unwanted in the application at hand. Minimization of this balanced cut over all partitions of $V$ is an NP-complete problem \cite{WagnerWagner93,ShiMalik00}, but a relaxation of this problem can be defined using the graph Laplacian. For example, in \cite[Section 5]{Luxburg:2007} it is shown that
\[
C_r(S_1, \ldots, S_k) = \mathrm{Tr}(H^T L H),
\]
where $H$ is an $n$ by $k$ matrix with elements
\begin{equation}\label{eq:matrixH}
h_{ij} := \begin{cases}
\big(\mathrm{vol} \ S_j\big)^{-\frac12} & \text{if } i\in S_j,\\
0 & \text{else}.
\end{cases}
\end{equation}
Note that $H$ is orthonormal in the $\mathcal{V}$ inner product, {\it i.e.}, $H^T D^r H = I$, where $I$ is the $k$ by $k$ identity matrix. Hence the minimization of $C_r$ over all partitions, is equivalent to minimizing $\mathrm{Tr}(H^T L H)$ over all $n$ by $k$ $\mathcal{V}$-orthonormal matrices of the form as in \eqref{eq:matrixH}. The relaxation of this NP-complete minimization problem is now formulated by dropping the condition \eqref{eq:matrixH} and minimizing over all $n$ by $k$ $\mathcal{V}$-orthonormal matrices. The problem then becomes an eigenvalue problem and the first $k$ eigenvectors of $L$ are expected to be approximations to the indicator functions of the optimal partition $S_1, \ldots, S_k$. There is often no guarantee of the quality of this approximation though \cite{SpielmanTeng96,GuatteryMiller98,KannanVempalaVetta00,KannanVempalaVetta04,SpielmanTeng07,Luxburg:2007}.

In order to turn the approximations of the indicator functions into true indicator functions a method like thresholding or $k$-means clustering is often used \cite{NgJordanWeiss02}. For partitioning the nodes into two subsets ($k=2$), the potential term in the Allen-Cahn equation (discussed  in Section~\ref{sec:AC}) can be interpreted as a nonlinear extension of the graph Laplacian eigenvalue problem which forces the approximate solutions to be close to indicator functions. In this light, it is interesting that the graph Ginzburg-Landau functional (with possibly a mass constraint), of which the Allen-Cahn equation is a gradient flow, $\Gamma$-converges to the graph cut functional \cite{vanGennipBertozzi12}.

\section{Curvature and mean curvature flow on graphs}\label{sec:curvature}

In this section we will derive a graph curvature, analogous to mean curvature in the continuum, and study some of its properties. We end the section with a description of mean curvature flow on graphs.

\subsection{Graph curvature}\label{subsec: divergence and curvature}
In the continuum case, the mean curvature is given by (minus) the divergence of the normal vector field  on the boundary of the set (see Appendix~\ref{subsec: continuum mcf} and \eqref{eq:MCF}). This normal vector field achieves the supremum in the definition of total variation of a characteristic function (under sufficient smoothness conditions).
Hence, we define the  normal of a vertex set by using $\varphi^a$  from \eqref{eq:varphiaTV} which achieves the supremum of the anisotropic total variation.
\begin{definition}
The normal of the vertex set $S\subset V$ is
\begin{equation}
\label{eq:graphNormal}
\nu^S_{ij} := \sgn\big((\nabla \chi_S)_{ij}\big) =\begin{cases} 1 & \text{if } \omega_{ij}>0,  i\in S^c, \textrm{ and } j\in S,\\ -1 &\text{if } \omega_{ij}>0, j\in S^c, \textrm{ and }i\in S  ,\\ 0 &\text{else}.\end{cases}
\end{equation}
\end{definition}
As in the continuum case, we define the curvature of a set as the divergence of the normal.
\begin{definition}
The curvature of the vertex set $S\subset V$ at node $i\in V$ is
\begin{equation}\label{eq:dvgnu}
(\kappa_S^{q,r})_i: = (\dvg \nu^S)_i = d_i^{-r} \begin{cases} \sum_{j\in S^c} \omega_{ij}^q &\text{if } i\in S,\\ -\sum_{j\in S} \omega_{ij}^q &\text{if } i\in S^c.\end{cases}
\end{equation}
\end{definition}
Recall from \eqref{eq:varphiaTV}, that $\nu^S_{ij}$ is not uniquely determined on $\{(i,j)\in E\colon (\nabla \chi_S)_{ij}=0\}$ (\textit{i.e.,} away from the boundary $\partial S \cup \partial (S^c)$, in the sense of Definition~\ref{def:boundary}) and hence the value $0$ in \eqref{eq:graphNormal} is a choice corresponding to the extension of the normal field away from the boundary. This ambiguity is irrelevant when $\dvg \nu^S$ is coupled to the characteristic function $\chi_S$ via the $\mathcal{V}$-inner product, as in
\begin{equation}\label{eq:recallTVdef}
\TVa^q(\chi_S) = \langle \kappa_S^{q,r}, \chi_S\rangle_{\mathcal{V}},
\end{equation}
but care should be taken when trying to interpret the normal or the curvature outside this setting.

Note that for $q=1$ and $S\subset V$, $| (\kappa_S^{1,r})_i | \leq d_i^{1-r} $ for all $i\in V$. Also, $\langle \kappa_S^{q,r},  \chi_V \rangle_{\mathcal V} = \langle \nu, \mathrm{grad} \ \chi_V \rangle_{\mathcal V} = 0$.
The curvature, $\kappa_S^{q,r}$, has the property that it vanishes away from the boundary $\partial S \cup \partial(S^c)$.
In particular, we see that the above mentioned ambiguity is also irrelevant for pairings with $\chi_{S^c}$  since
\begin{equation}\label{eq:negativeTV}
\langle \kappa_S^{q,r}, \chi_{S^c} \rangle_{\mathcal V}
= \langle \kappa_S^{q,r}, \chi_V \rangle_{\mathcal V} - \langle \kappa_S^{q,r}, \chi_S \rangle_{\mathcal V}
= - \TVa^q( \chi_{S} ).
\end{equation}

Let $S, \hat S \subset V$ be two node sets, then \eqref{eq:recallTVdef} implies
\begin{align*}
\TVa^q(\chi_{\hat S}) - \TVa^q(\chi_S) &= \langle \kappa_{\hat S}^{q,r}, \chi_{\hat S} \rangle_{\mathcal{V}} - \langle \kappa_S^{q,r}, \chi_S \rangle_{\mathcal{V}} \\
&= \langle \kappa_{\hat S}^{q,r} + \kappa_S^{q,r}, \chi_{\hat S} - \chi_S \rangle_{\mathcal{V}} + \langle \kappa_{\hat S}^{q,r}, \chi_S \rangle_{\mathcal{V}} - \langle \kappa_S^{q,r}, \chi_{\hat S} \rangle_{\mathcal{V}}.
\end{align*}
For the last two terms, we compute
\begin{align*}
&\langle \kappa_{\hat S}^{q,r}, \chi_S \rangle_{\mathcal{V}} - \langle \kappa_S^{q,r}, \chi_{\hat S} \rangle_{\mathcal{V}} \\
&\hspace{0.2cm} = \left[ \sum_{i\in S\cap \hat S} \sum_{j\in \hat S^c} - \sum_{i\in S\setminus \hat S} \sum_{j\in \hat S} - \sum_{i\in \hat S \cap S} \sum_{j\in S^c} + \sum_{i\in \hat S\setminus S} \sum_{j\in S}\right] \omega_{ij}^q\\
&\hspace{0.2cm} = \left[-\sum_{i\in S\setminus \hat S} \sum_{j\in \hat S} + \sum_{i\in S}\sum_{j\in \hat S} - \sum_{i\in S}\sum_{j\in\hat S\cap S} - \sum_{i\in \hat S\cap S} \left(\sum_{j\in S^c} - \sum_{j\in \hat S^c}\right)\right] \omega_{ij}^q\\
&\hspace{0.2cm} =\sum_{i\in \hat S\cap S} \left[\sum_{j\in \hat S} - \sum_{j\in S} - \sum_{j\in S^c} + \sum_{j\in \hat S^c}\right] \omega_{ij}^q\\
 &= \sum_{i\in \hat S \cap S} \left[\sum_{j\in V} - \sum_{j\in V}\right] \omega_{ij}^q = 0,
\end{align*}
and thus
\begin{equation}\label{eq:TVdifference}
\TVa^q(\chi_{\hat S}) - \TVa^q(\chi_S) =  \langle \kappa_{\hat S}^{q,r} + \kappa_S^{q,r}, \chi_{\hat S} - \chi_S\rangle_{\mathcal{V}}.
\end{equation}

In particular, if $\hat S = S \setminus\{n\}$ for a node $n\in S$, then
    \begin{align*}
         \TVa^q(\chi_{S\setminus\{n\}}) - \TVa^q(\chi_S) &=  -\langle \kappa_{S\setminus \{n\}}^{q,r} + \kappa_S^{q,r}, \chi_{\{n\}}\rangle_{\mathcal{V}}\\
         &= \sum_{i\in S} \omega_{in}^q - \omega_{nn}^q - \sum_{j\in S^c} \omega_{nj}^q\\
         &= \left(\sum_{j \in S}-\sum_{j \in S^c}\right) \omega_{nj}^q.
    \end{align*}
Because we assume there are no self-loops, in the final equality $\omega_{nn}=0$. A similar computation shows for $n\in S^c$
    \begin{align*}
         \TVa^q(\chi_{S\cup\{n\}}) - \TVa^q(\chi_S) &= \left(\sum_{j\in S^c}-\sum_{j\in S}\right) \omega_{nj}^q - \omega_{nn}^q\\
         &= \left(\sum_{j\in S^c}-\sum_{j\in S}\right) \omega_{nj}^q.
    \end{align*}

The preceding discussion implies the following: if $\Omega \subset V$ is such that $S$ minimizes $\TVa^q(\chi_S)$ among all sets $S' \subset V$ such that $S\Delta S' \subset \Omega$, then we have that
    \begin{equation}\label{eq: zero graph curvature inequalities}
    	 \left \{ \begin{array}{rl}
              \left(\sum_{j\in S}-\sum_{j\in S^c}\right) \omega_{nj}^q \leq 0 & \text{ if } n \in S^c \cap \Omega,\\
    	 	  \left(\sum_{j \in S}-\sum_{j \in S^c}\right) \omega_{nj}^q \geq 0 & \text{ if } n \in S \cap \Omega,
    	 \end{array} \right.
    \end{equation}
    (compare with the nonlocal mean curvature, \eqref{eq: zero curvature inequalities}).
Here, $\Omega$ is a set where $S$ and $S'$ are forced to agree (similar to enforcing a boundary condition in the continuum case). The two inequality conditions in \eqref{eq: zero graph curvature inequalities} are opposite for the two sides of the interface between $S$ and $S^c$. This strengthens the heuristic idea that the `real' interface, where there would be an equality condition, is lost due to the lower bound on the accessible length scales on a graph.

\begin{remark}
It is interesting to make some connections between the graph total variation $\TVa$ and graph curvature $\kappa_S^{1,r}$ on the one hand, and the local clustering coefficient \cite{WattsStrogatz98} on the other. Assume $G$ is unweighted (and undirected, as per usual in this paper), then the clustering coefficient $C_i$ of node $i$, is the number of triangles node $i$ is part of, divided by the number of possible triangles in the neighborhood of $i$, {\it i.e.}, $C_i=\frac{2|T_i|}{d_i (d_i-1)}$, where
\[
T_i := \{ \{i, j, k\}: (i,j), (i,k), (k,i) \in E\}.
\]
(A version for weighted graphs was introduced in \cite[Formula 5]{Barratetal04}.) Using \eqref{eq:TVofS} and \eqref{eq:negativeTV}, we can rewrite, for $r=1$,
\begin{align*}
C_i &= \frac1{d_i (d_i-1)} \sum_{j,h \in \mathcal{N}_i} \omega_{jh}\\
 &= \frac1{d_i (d_i-1)} \left(\sum_{j\in V} \sum_{h\in \mathcal{N}_i} \omega_{jh} - \sum_{j\in \mathcal{N}_i^c} \sum_{h\in \mathcal{N}_i} \omega_{jh}\right)\\
&= \frac1{d_i (d_i-1)} \left( \sum_{h\in \mathcal{N}_i} d_h - \TVa(\chi_{\mathcal{N}_i})\right)\\
&= \frac1{d_i (d_i-1)} \left( \mathrm{vol} \ \mathcal{N}_i - \TVa(\chi_{\mathcal{N}_i})\right)\\
&= \frac1{d_i (d_i-1)} \langle \chi_{\mathcal{N}_i}, \chi_{\mathcal{N}_i} + \kappa_{\mathcal{N}_i^c}^{1,1}\rangle_{\mathcal{V}}.
\end{align*}
\end{remark}

\smallskip
\smallskip

As for the continuum case, we can arrive at the graph curvature, $\kappa_S^{q,r}$, in several ways.  We discuss some of them below. However, the analogy with the continuum curvature becomes even more apparent if instead of the standard mean curvature, we consider the nonlocal mean curvature (see Section~\ref{subsec: nonlocal mean curvature}).

If we (formally) compute the first variation of the continuum $\TV(u)$ over all functions $u\in BV$ of {\it bounded variation} (and not restricting ourselves to characteristic functions only), we find $\dvg \frac{\nabla u}{|\nabla u|}$, the curvature of the level sets of $u$.
In \eqref{eq:variationofTV} and Appendix~\ref{sec:formalfirstvar}
 we follow a similar procedure on the graph and again find $\dvg \varphi^a$, with $\varphi^a$ from \eqref{eq:varphiaTV}. 


Similarly, in the continuum case, an alternative definition of the mean curvature is $\dvg \left( \frac{\nabla \chi_{S}}{| \nabla \chi_{S}| }\right) $ which is a Radon measure defined on the boundary.
However, $|\nabla \chi_{S}| = 1$ on the boundary of $S$, so that the mean curvature  is simply $ \Delta \chi_{S}$.
As long as $S$ is a rectifiable set, this computation can be made rigorous in the context of BV functions (see \cite[Chapter 5]{EvansGariepy92}).
Computing the analogous quantity on a graph, we find
\begin{equation}\label{eq:DeltaChi}
(\Delta \chi_S)_i = d_i^{-r} \left \{ \begin{array}{rl}
\sum \limits_{j \in S^c} \omega_{ij} & \text{ if }i \in S,\\
- \sum \limits_{j \in S} \omega_{ij} & \text{ if } i \in S^c.
\end{array} \right.
\end{equation}
This is equal to $\kappa_S^{1,r}$ (compare with \eqref{eq:TVusingDelta}). The choice $q=1$ is a natural one, because it corresponds to the $\Gamma$-limit of the graph Ginzburg-Landau functional \eqref{eq:GL functional} \cite{vanGennipBertozzi12}, whose definition is given in Section~\ref{sec:AC}.

In the continuum case the mean curvature $\kappa(x)$ at the point $x\in \partial S$ in the boundary of a set $S\in \R^d$ satisfies the property that, if $S$ is smooth enough, then for any given ball $B_\delta(x)$ of radius $\delta$ and center $x \in \partial S$,
$$ |B_\delta(x) \cap S|-\tfrac12 |B_\delta(x)| = \kappa \delta^2|B_\delta(x)|  + o(\delta^2|B_\delta(x)|).$$
Note that if $S$ were a half space, then the expression on the right is zero for all $\delta$, since $\partial S$ would separate each $B_\delta(x)$ in sets of equal volume. Thus $\kappa$ measures how much $\partial S$ deviates from cutting $B_\delta(x)$ in sets of equal volume. The analogous computation on a graph, replacing the ball by the set $\overline{\mathcal{N}_i} := \mathcal{N}_i \cup \{i\}$ of neighbors of node $i$ together with node $\{i\}$, gives, for $S\subset V$,
\begin{align*}
\mathrm{vol}\ (\overline{\mathcal{N}_i} \cap S) - \tfrac12 \mathrm{vol}\ \overline{\mathcal{N}_i} &= \sum_{j\in \overline{\mathcal{N}_i}\cap S}  d_j^r - \tfrac12 \sum_{j\in \overline{\mathcal{N}_i}} d_j^r\\
 &= \tfrac12 \left(\sum_{j\in \overline{\mathcal{N}_i}\cap S} - \sum_{j\in \overline{\mathcal{N}_i} \cap S^c}\right) d_j^r.
\end{align*}
Note that if $r=0$ and $G$ is an unweighted graph, such that  $\omega_{ij}\in\{0,1\}$, then
\begin{align*}
\mathrm{vol}\ (\overline{\mathcal{N}_i} \cap S) - \tfrac12 \mathrm{vol}\ \overline{\mathcal{N}_i} &= \tfrac12 \left(\sum_{j\in \overline{\mathcal{N}_i}\cap S} - \sum_{j\in \overline{\mathcal{N}_i} \cap S^c}\right) \omega_{ij}^q\\
&= \tfrac12 \left(\sum_{j\in S} - \sum_{j\in S^c}\right) \omega_{ij}^q,
\end{align*}
in whose right hand side we recognize \eqref{eq: zero graph curvature inequalities}. In the last equality we used that $\omega_{ij}=0$, if $j\not \in \mathcal{N}_i$.

The equality in \eqref{eq:DeltaChi} has an interesting consequence on the level of $\Gamma$-convergence of functionals. For more information about the theory of $\Gamma$-convergence, we refer to \cite{DalMaso93,Braides02}. The first result of $\Gamma$-convergence for the Ginzburg-Landau functional goes back to work of Modica and Mortola \cite{ModicaMortola77,Modica87}.

\begin{remark}
We note that graph curvature is related to the process of \emph{bootstrap percolation} \cite{ChalupaLeathReich79,JansonLuczakTurovaVallier12,BradonjicSaniee14}, in which nodes on an unweighted graph switch from `inactive' to `active', if their number of active neighbors exceeds a given threshold value. This number corresponds exactly to the graph curvature $(\kappa_S^{q,0})_i$, if node $i$ is inactive and $S$ is the set of active nodes.
\end{remark}

\begin{theorem}\label{thm:curvGamma}
Let $g\in C(\R^n)$, $\e>0$, $\tilde W\in C^2(\R)$ a nonnegative double well potential with wells at $0$ and $1$, and consider the functionals $f_\e, f_0: \mathcal{V} \to \R$, defined by
\begin{align*}
f_\e(u) &:= \sum_{i\in V} g\big((\Delta u)_i\big) + \frac1\e \sum_{i\in V} \tilde W(u_i),\\
f_0(u) &:= \begin{cases} \sum_{i\in V} g\big((\kappa_S^{1,r})_i\big) & \text{if  } u=\chi_S \text{ for some } S\subset V,\\ +\infty & \text{else}.\end{cases}
\end{align*}
Then $f_\e \overset{\Gamma}\to f_0$ as $\e \to 0$ (using any of the equivalent metrics on $\R^n$).

\smallskip

Furthermore, if the double well potential $\tilde W$ satisfies a coercivity condition ---{\it i.e.}, there exists a $c > 0$ such that for large $|u|$, $\tilde W(u) \leq c(u^2-1)$--- then compactness holds, in the following sense: Let $\{\e_n\}_{n=1}^\infty \subset \R_+$ be such that $\e_n \to 0$ as $n\to \infty$, and let $\{u_n\}_{n=1}^\infty$ be a sequence such that there exists a $C>0$ such that for all $n\in \N\,$ $f_{\e_n}(u_n) < C$. Then there exists a subsequence $\{u_{n'}\}_{n'=1}^\infty \subset \{u_n\}_{n=1}^\infty$ and a $u_\infty$ of the form $u_\infty = \chi_S$, for some $S\subset V$, such that $u_{n'} \to u_\infty$ as $n\to\infty$.
\end{theorem}
\begin{proof}
The key point in the proof of the $\Gamma$-convergence is to note that $f_\e$ is a continuous perturbation of the functional $w_\e: \mathcal{V} \to \R$,
\[
w_\e(u) :=  \frac1\e \sum_{i\in V} \tilde W(u_i).
\]
By \cite[Lemma 3.3]{vanGennipBertozzi12}\footnote{Note that in the statement and proof of Lemma 3.3 in \cite{vanGennipBertozzi12}, it says $g_\e$ twice where $w_\e$ is meant.} $w_\e \overset{\Gamma}\to w_0$ as $\e\to 0$, where
\[
w_0(u) := \begin{cases} 0 & \text{if } u=\chi_S \text{ for some } S\subset V,\\ +\infty & \text{else}.\end{cases}
\]
By a well known property of $\Gamma$-convergence \cite[Proposition 6.21]{DalMaso93}, the $\Gamma$-limit is preserved under continuous perturbations. Then using the fact, shown above in \eqref{eq:DeltaChi}, that $\Delta u = \kappa_S^{1,r}$ if $u=\chi_S$, completes the proof of $\Gamma$-convergence.

\smallskip

The compactness result is a direct adaptation of the proof of \cite[Theorem 3.2]{vanGennipBertozzi12} to the current functionals $f_\e$.
\end{proof}

\begin{remark}
Note that in Theorem~\ref{thm:curvGamma} above, we can also use the double well potential $W$ with wells at $\pm 1$, instead of $\tilde W$. In that case, the limit functional $w_0$ in the proof takes finite values only for functions of the form $u=\chi_S - \chi_{S^c}$, for some $S\subset V$. Because $\Delta (\chi_S + \chi_{S^c}) = \Delta \chi_V = 0$, we have $\Delta (\chi_S - \chi_{S^c}) = 2 \kappa_S^{1,r}$ and hence the limit functional $f_0$ takes the form
\[
f_0(u) := \begin{cases} \sum_{i\in V} g\big(2(\kappa_S^{1,r})_i\big) & \text{if  } u=\chi_S \text{ for some } S\subset V,\\ +\infty & \text{else}.\end{cases}
\]

\end{remark}

\bigskip

We end this subsection with another similarity between the graph based objects we introduced and their continuum counterparts. The gradient of the graph distance $d^{\partial S}$, from Definition~\ref{def:distance}, agrees with the normal $\nu$, from \eqref{eq:graphNormal}, on the boundary of $S$ induced by the graph distance, in the sense of the following lemma. This again corresponds to what we expect based on the continuum case. We define the signed distance to $\partial S$ as
\begin{equation}\label{eq:signeddisttopartialS}
sd^{\partial S} := (\chi_{S^c}-\chi_S) d^{\partial S}.
\end{equation}

\begin{lemma}
Let $S\subset V$. Define the exterior boundary of $S$ induced by the graph distance $d^{\partial S}$ as
\[
\partial_{ext}S := \{i\in S^c : \exists j\in \partial S \text{ such that } d^{\partial S}_i = \omega_{ij}^{q-1}\}.
\]
Let $i\in \partial_{ext}S$, then there is a $j\in \partial S$ such that $(\nabla sd^{\partial S})_{ij} = -\nu_{ij}$.

Similarly, let the interior boundary of $S$ induced by the graph distance be
\[
\partial_{int}S := \{i\in S : \exists j\in \partial S \text{ such that } d^{\partial S}_i = \omega_{ij}^{q-1}\}.
\]
If $i\in \partial_{int}S$, then there is a $j\in \partial S$ such that $(\nabla sd^{\partial S})_{ij} = \nu_{ij}$.
\end{lemma}
\begin{proof}
First we note that, for $i\in \partial_{ext}S \subset S^c$, we have $sd^{\partial S}_i = d^{\partial S}_i$. Because $d^{\partial S}$ satisfies equation \eqref{eq:eikonal} (with $S$ replaced by $\partial S$), we have
\[
\min_{k\in \mathcal{N}_i} (\nabla d^{\partial S})_{ik} = \min_{k\in \mathcal{N}_i} \omega_{ik}^{1-q}(d^{\partial S}_k - d^{\partial S}_i) = -1.
\]
Note that $\partial_{ext}S \subset \partial (S^c)$. Hence, $i\in \partial (S^c)$ and thus there is a $k\in \mathcal{N}_j$ such that $k\in \partial S$ and therefore $d^{\partial S}_k = 0$. Because $d^{\partial S}$ is nonnegative and $i\in \partial_{ext}S$, we deduce
\[
1 = \max_{k\in \mathcal{N}_i} \omega_{ik}^{1-q} d^{\partial S}_i = \max_{k\in \mathcal{N}_i} \omega_{ik}^{1-q} \omega_{ij}^{q-1}.
\]
Thus the maximum is achieved for $k=j$, and hence so is the minimum in \eqref{eq:eikonal}, which shows that
\[
(\nabla d^{\partial S})_{ij} = -1 = -\nu_{ij}.
\]
The proof for $i\in \partial_{int}S$ follows from similar arguments, noting that $sd^{\partial S}_i = -d^{\partial S}$.
\end{proof}

Note that $\partial_{ext}S \subset \partial (S^c)$, but equality does not necessarily hold. If a shortest path from $i\in \partial (S^c)$ to $\partial S$ does not equal $\{i,j\}$, for some $j\in \partial S$, then $i\not\in \partial_{ext}S$. This situation does not occur if the graph distances are consistent, in the following sense: if, for all $i, j, k\in V$, $\omega_{ij}^{q-1} \leq \omega_{ik}^{q-1} + \omega_{kj}^{q-1}$. In that case, $\partial_{ext}S = \partial (S^c)$.

\subsection{Relation with the continuum nonlocal mean curvature}\label{subsec: nonlocal mean curvature}
    There is a clear analogy between the expressions in \eqref{eq: zero graph curvature inequalities}
    and the (continuum) {\it nonlocal mean curvature} \cite{CaffarelliSouganidis10,CaffarelliRoquejoffreSavin10}, as well as between $\TVa^q$ and continuum nonlocal energy functionals.

    Consider an interaction kernel $K\colon \mathbb{R}^n \times \mathbb{R}^n \to [0,+\infty)$ with $K(x,y) = K(y,x)$ and
    \begin{align*}
         \sup \limits_{x \in \mathbb{R}^n} \int_{\mathbb{R}^n} \min \{1,|x-y|^2\}K(x,y)\;dy	<+\infty.
    \end{align*}
    This kernel $K$ can be thought of as the energy given by a long-range interaction between a particle placed at $x$ with a particle at $y$. It defines a functional on subsets $S \subset \mathbb{R}^n$, sometimes called ``nonlocal perimeter'' or ``nonlocal energy'' and it is given by
    \begin{equation*}
         J_K(S) = L_K(S,S^c),
    \end{equation*}
    where for any pair $A,B \subset \mathbb{R}^n$ we write
    \begin{equation}\label{eq: interaction form}
    	 L_K(A,B) = \int_{A}\int_B K(x,y)\;dxdy.
    \end{equation}
    Compare this with the graph case, if, for $A,B \subset V$, we write
    \begin{equation}\label{eq: graph interaction form}
         L_G(A,B) = \sum \limits_{i \in A}\sum \limits_{j \in B} \omega_{ij}^q.
    \end{equation}
    In terms of this bilinear functional, the anisotropic total variation, defined in \eqref{eq:TVofS}, can be rewritten
      \begin{equation*}
    	 \TVa^q(\chi_S) = L_{G}(S,S^c).
    \end{equation*}
       We see that \eqref{eq: interaction form} is nothing but the continuum version of \eqref{eq: graph interaction form}, and one may rightfully interpret  the weight matrix $\omega_{ij}^q$ as an interaction kernel between pairs of nodes in $G$ and $L_G(S,S^c)$ as measuring the total ``interaction energy'' between $S$ and $S^c$.

    Now suppose that $S \subset \mathbb{R}^n$ minimizes $J_K(S)$ in some domain $\Omega$, meaning that if $S'$ is such that $S \Delta S' \subset \subset \Omega$ then $J_K(S)\leq J_K(S')$. In this case one can see that the following two conditions must hold
    \begin{equation}\label{eq: zero curvature inequalities}
	    \left \{ \begin{array}{rl}
              L(A,S)-L(A,S^c \setminus A) &\leq 0, \;\;\;\forall\; A \subset S^c \cap \Omega,  \\
              L(A,S \setminus A)-L(A,S^c) &\geq 0, \;\;\;\forall\; A \subset S \cap \Omega.
		\end{array} \right.
    \end{equation}
 If, arguing heuristically, we let $A$ shrink down to any $x \in (\partial S) \cap \Omega$, we find that
    \begin{equation*}
         \int_{\mathbb{R}^n}\left (\chi_S(y)-\chi_{S^c}(y)\right )K(x,y)\;dx = 0.
    \end{equation*}
    The integral on the left, which is well defined in the principal value sense\footnote{Note that $K(x,y)$ could have a very strong singularity at $x=y$ making the integral diverge, if taken as a Lebesgue integral. The boundedness of the principal value of this singular integral tells us that $\partial S$ must have some smoothness near $x$.} when $x \in \partial S$ and $\partial S$ is smooth enough ($C^2$ suffices), is known as the nonlocal mean curvature of $S$ at $x$ with respect to $K$, or just nonlocal mean curvature of $S$ at $x$, when $K$ is clear from the context. As with $J_K$ and $\TVa^q$, we see that
    \begin{equation*}
    	 \kappa_{\text{nonlocal}}(x) := \int_{\mathbb{R}^n}\left (\chi_S(y)-\chi_{S^c}(y)\right )K(x,y)\;dy
    \end{equation*}
    is exactly a continuum analogue of the quantity $\left(\sum_{j\in S}-\sum_{j\in S^c}\right) \omega_{nj}^q$ in \eqref{eq: zero graph curvature inequalities},
    moreover, the inequalities in  \eqref{eq: zero curvature inequalities} are a continuum analogue  of those in \eqref{eq: zero graph curvature inequalities}. Note however, that
    $\left(\sum_{j\in S}-\sum_{j\in S^c}\right) \omega_{nj}^q$ is defined for all of $n\in V$, whereas $\kappa_{\text{nonlocal}}(x)$ above is only defined when $x \in \partial S$ and $\partial S$ is smooth enough.

    Known regularity results deal mostly with the case $K_s(x,y):= c_{n,s}|x-y|^{-n-s}$ for $s \in (0,1)$ \cite{CaffarelliRoquejoffreSavin10,Caputo2010}. It is worth noting that  $J_{K_s}$ is a fractional Sobolev norm of the characteristic function of $S$
    \begin{equation*}
    	 J_{K_s}(S) = \tfrac{1}{2}\|\chi_S\|_{\dot{H}^{s/2}}^2,
    \end{equation*}
    where $\|.\|_{\dot{H}^s/2}$ is defined in terms of  the Fourier transform of $f$ by
    \begin{equation*}
         \|f\|_{\dot{H}^{s/2}} = \|\;|\xi|^s\hat f(\xi)\;\|_{L^2(\mathbb{R}^n)}.	
    \end{equation*}
    Moreover, as $s \to 1^-$ the quantity above gives the perimeter of $S$, and the corresponding nonlocal mean curvature converges pointwise to the standard mean curvature.
    In \cite{CaffarelliSouganidis10} it is shown that if we consider the MBO scheme,  where instead of the heat equation we use the fractional heat equation, then in the limit we get a set $S_t$ evolving over time with a normal velocity at $x \in \partial S_t$ given by
    \begin{equation}\label{eq: alpha order curvature}
         V(x) = c_{n,s}\int_{\mathbb{R}^n} \frac{\chi_S(y)-\chi_{S^c}(y)}{|x-y|^{n+s}}\;dy. 	
    \end{equation}

\subsection{Mean curvature flow}\label{sec:MCF}

In this section, we define a mean curvature flow on graphs and connect it with the curvature $\kappa_S^{q,r}$ in \eqref{eq:dvgnu}. It is not clear what is the most natural notion for the evolution of a phase in a graph. Do we want to consider a sequence of subsets $\{ S_n \}_{n \in \mathbb{N}}$, or a continuous family $\{S_t\}_{t>0}$ which, although piecewise constant in $t$, may change in arbitrarily small time intervals? How we connect solutions of the graph mean curvature flow to solutions of the graph \ref{alg:MBO} scheme from Section~\ref {sec:MBO}, or to solutions of the graph Allen-Cahn equation \eqref{eq:ACEe} from Section~\ref{sec:AC}, will depend on the answer to this question. For now, we shall be content with considering a phase evolution comprised of a discrete sequence of sets $S_n = S(n \eth t)$, $n \in \mathbb{N}$, that correspond to the state of the system at discrete time steps.

Our construction follows the well-known variational formulations for classical mean curvature flow \cite{AlmgrenTaylorWang93,LuckhausSturzenhecker95}.
Appendix~\ref{subsec: continuum mcf} has a brief overview of mean curvature and the associated flow in the continuum case. An obstacle  one encounters when trying to emulate the continuum level set method to express  mean curvature flow on graphs is, that, due to the lack of a discrete chain rule, the resulting equation is not independent on the choice of level set function.

Recall the notions of graph distance and boundary of a node set from Section~\ref{sec:setup}.

\begin{definition}\label{def:graphMCF}
The mean curvature flow, $S_n = S(n \eth t)$, with discrete time step $\eth t$ for an initial set $S_0 \subset V$, is defined
\[
\label{eq:MCFh}\tag{MCF$_{\eth t}$}
S_{n+1} \in \arg \min_{\hat S\subset V} \mathcal{F}(\hat S, S_n),
\]
where
\begin{equation}\label{eq:MCFfunctional}
\mathcal{F}(\hat S, S_n) :=  \TVa^q(\chi_{\hat S}) - \TVa^q(\chi_{S_n}) + \frac1{\eth t} \langle \chi_{\hat S} - \chi_{S_n}, (\chi_{\hat S}-\chi_{S_n}) d^{\Sigma_n}\rangle_{\mathcal{V}}
\end{equation}
and
\begin{align*}
\Sigma_n &:= \partial S_n \cup \partial (S_n^c)\\
 &= \{i\in V: \exists \, (i,j)\in E\ \\
 &\hspace{0.7cm} \text{such that } (i\in S_n \wedge j\in S_n^c) \vee (i\in S_n^c \wedge j\in S_n)\}.
\end{align*}
\end{definition}
Note that, for a given graph $G$, minimizers of $\mathcal{F}$ may not be unique. In this case different mean curvature flows can be defined, depending on the choice of $S_{n+1}$. An example of this non-uniqueness on the 4-regular graph is given in Section~\ref{sec:tori}.

We choose to use the distance to $\Sigma_n$, instead of the distance to either $\partial S_n$ or $\partial (S_n^c)$, so that the mean curvature flow is not {\it a priori} (independent of curvature) biased to either adding nodes to or removing nodes from $S$.

Since nodes in $\Sigma_n$ can be added or removed from $S$ without increasing the last term of \eqref{eq:MCFfunctional}, every stationary state $\chi_S$ of \eqref{eq:MCFh} is a minimal surface in the sense that
$\TVa^q(\chi_S) \leq \TVa^q(\chi_{\{S \cup \{n\}})$ for $n\in \partial (S^c)$ and $\TVa^q(\chi_S) \leq \TVa^q(\chi_{\{S \setminus \{n\}})$ for $n\in \partial S$ (in the case where the minimizer of $\mathcal{F}$ is unique, the inequalities are strict). In particular, the sequence defined by \eqref{eq:MCFh} is not ``frozen'' for $\eth t$ arbitrarily small.


\begin{remark}
In the continuum, stationary points of mean curvature flow are minimal surfaces, {\it i.e.}, surfaces of zero mean curvature \cite{Giusti84,Calle07,MeeksIIIPerez11,ColdingMinicozziII11}. On the graph, the zero mean curvature condition is too restrictive, since it would only allow for disconnected sets $S$ and $S^c$, but, given Definition~\ref{def:graphMCF} above, we can still define an \emph{$\eth t$-minimal set} on a graph as a node set $S\subset V$, such that $S$ is a stationary point of the mean curvature flow \eqref{eq:MCFh}. Note that, if $\eth t_1 > \eth t_2$ and $S$ is an $\eth t_1$-minimal set, then $S$ is also an $\eth t_2$-minimal set.
\end{remark}

\begin{remark} In the last term of $\mathcal F(\hat S, S_n)$, we use a symmetrized distance to the boundary, $\langle \chi_{\hat S} - \chi_{S_n}, (\chi_{\hat S}-\chi_{S_n}) d^{\Sigma_n}\rangle_{\mathcal{V}}$. Other choices are possible here, including $\|  \chi_{S^c} d^S - \chi_S d^{S^c} \|^2 $. For this alternative choice, \eqref{eq:MCFh} exhibits ``freezing''.
\end{remark}

We can rewrite the last term in $\mathcal{F}(\hat S, S_n)$ in terms of the signed graph distance
\[
sd^{\Sigma_n} := (\chi_{S_n^c}-\chi_{S_n}) d^{\Sigma_n},
\]
(compare with \eqref{eq:signeddisttopartialS}), which takes nonnegative values in $S_n^c$ and nonpositive values in $S_n$.  We state the precise result in the following lemma.

\begin{lemma}\label{lem:fromFtoF'}
$\displaystyle
\underset{\hat S\subset V}{\mathrm{argmin}} \, \mathcal{F}(\hat S, S_n) =  \underset{\hat S\subset V}{\mathrm{argmin}} \, \mathcal{F}'(\hat S, S_n),
$
where
\begin{align}
\mathcal{F}'(\hat S, S_n) &:=  \TVa^q(\chi_{\hat S}) - \TVa^q(\chi_{S_n}) + \frac1{\eth t} \langle \chi_{\hat S}, sd^{\Sigma_n}\rangle_{\mathcal{V}}\label{eq:F'}\\
&= \langle \kappa_{\hat S}^{q,r} + \kappa_{S_n}^{q,r}, \chi_{\hat S} - \chi_{S_n}\rangle_{\mathcal{V}} + \frac1{\eth t} \langle \chi_{\hat S}, sd^{\Sigma_n}\rangle_{\mathcal{V}}\notag.
\end{align}
\end{lemma}
\begin{proof}
The rewriting of $\TVa^q(\chi_{\hat S}) - \TVa^q(\chi_{S_n})$ in terms of the curvatures, follows directly from \eqref{eq:TVdifference}. For the distance term, we compute
\begin{align*}
\langle (\chi_{\hat S} - \chi_{S_n})^2, d^{\Sigma_n}\rangle_{\mathcal{V}} &= \langle \chi_{\hat S} (1-2\chi_{S_n}) + \chi_{S_n}, d^{\Sigma_n}\rangle_{\mathcal{V}}\\
&= \langle \chi_{\hat S}, (\chi_{S_n^c} - \chi_{S_n})d^{\Sigma_n}\rangle_{\mathcal{V}} + \langle \chi_{S_n}, d^{\Sigma_n}\rangle_{\mathcal{V}},
\end{align*}
where in the last line we used that $1-\chi_{S_n} = \chi_{S_n^c}$. The proof is completed by noting that the last term above does not depend on $\hat S$.
\end{proof}

Lemma~\ref{lem:fromFtoF'} allows us to find a convex functional, the superlevel sets of whose minimizers are themselves minimizers of $\mathcal{F}$ from \eqref{eq:MCFfunctional} (similar to the continuum results in, for example, \cite{ChanEsedogluNikolova06} and \cite{ChoksivanGennipOberman11}). Before we state and prove that result in Theorem~\ref{thm:convexify}, we prove a layer cake or coarea formula for the discrete total variation, $\TVa^q$.

\begin{lemma}\label{eq:coarea}
Let $u\in \mathcal{V}$ and define, for $t\in \R$,
\[
E(t) := \{i\in V\colon u_i > t \}.
\]
Then
\[
\TVa^q(u)
= \int_\R \TVa^q(\chi_{E(s)})\, ds.
\]
Let $u_-\leq \min_{i\in V} u_i$ and $\max_{i\in V} u_i \leq u_+$, then also
\begin{align*}
\TVa^q(u)
&= \int_{u_-}^\infty \TVa^q(\chi_{E(s)})\, ds = \int_{-\infty}^{u_+} \TVa^q(\chi_{E(s)})\, ds\\ &= \int_{u_-}^{u_+} \TVa^q(\chi_{E(s)})\, ds.
\end{align*}
\end{lemma}
\begin{proof}
As also noted in \cite{Zalesky02} and \cite{Chambolle05}, we have, for $i, j\in V$,
\[
|u_i-u_j| = \int_\R \left|\left(\chi_{E(s)}\right)_i - \left(\chi_{E(s)}\right)_j\right|\,ds,
\]
hence
\[
\frac12 \sum_{i, j \in V} \omega_{ij}^q |u_i-u_j| =  \int_\R \frac12 \sum_{i, j \in V} \omega_{ij}^q\left|\left(\chi_{E(s)}\right)_i - \left(\chi_{E(s)}\right)_j\right|\,ds.
\]
The second result follows, since, for $i, j \in V$,
\[
\int_{-\infty}^{u_-} \left|\left(\chi_{E(s)}\right)_i - \left(\chi_{E(s)}\right)_j\right|\,ds = \int_{u_+}^{\infty} \left|\left(\chi_{E(s)}\right)_i - \left(\chi_{E(s)}\right)_j\right|\,ds = 0.
\]
\end{proof}

\begin{theorem}\label{thm:convexify}
Let $m\in \R$, then the convex minimization problem
\begin{equation}\label{eq:Fconvex}
\min_{u\in \mathcal{V}_m} F(u),
\end{equation}
where $\mathcal{V}_m := \{u\in \mathcal{V}: -m\leq u \leq m \}$ and
\[
F(u) := \TVa^q(u) + \frac1{\eth t} \langle u, sd^{\Sigma_n}\rangle_{\mathcal{V}},
\]
has a minimizer $u\in \mathcal{V}_m$. Furthermore, for all $s\in (-m, m)$, the superlevel set
\[
E(s) := \{i\in V: u_i > s\}
\]
is a minimizer $\hat S$ of $\mathcal{F}(\cdot, S_n)$ from \eqref{eq:MCFfunctional}.
\end{theorem}
\begin{proof}
We can identify $\mathcal{V}_m$ with a compact subset of $\R^n$ and, in that setting, $F$ with a real-valued continuous function on this compact subset. Hence, a minimizer $u\in \mathcal{V}_m$ exists. 

Since $-m \leq \min_{i\in V} u_i$ and $m\leq \max_{i\in V} u_i$, we know from Lemma~\ref{eq:coarea} that
\[
\TVa^q(u) = \int_{-m}^m \TVa^q(\chi_{E(s)})\, ds.
\]
Writing $u_i-m = \int_{-m}^m \left(\chi_{E(s)}\right)_i \,ds$, 
we also see that
\[
\langle u-m, sd^{\Sigma_n}\rangle_{\mathcal{V}} = \int_{-m}^m \langle \chi_{E(s)}, sd^{\Sigma_n}\rangle_{\mathcal{V}} \,ds.
\]
This gives
\begin{align*}
F(u) - \frac1{\eth t} \langle m, sd^{\Sigma_n}\rangle_{\mathcal{V}} &= \int_{-m}^m \left[\TVa^q(\chi_{E(s)}) + \frac1{\eth t} \langle \chi_{E(s)}, sd^{\Sigma_n}\rangle_{\mathcal{V}}\right] \,ds\\
&= \int_{-m}^m \mathcal{F}''(E(s), S_n)\,ds,
\end{align*}
where, for $\hat S \subset V$,
\[
\mathcal{F}''(\hat S, S_n) := \mathcal{F}'(\hat S,S_n) + \TVa^q(S_n) = \TVa^q(\chi_{\hat S}) + \frac1{\eth t} \langle \chi_{\hat S}, sd^{\Sigma_n}\rangle_{\mathcal{V}},
\]
with $\mathcal{F}'$ as in \eqref{eq:F'}. Hence, if $u$ minimizes $F$, then for a.e. $s\in (-m, m)$, the superlevel set $E(s)$ minimizes $\mathcal{F}''(\cdot, S_n)$ and thus also $\mathcal{F}'(\cdot, S_n)$. Because $u$ takes only finitely many values, the result holds for all $s\in (-m,m)$. Lemma~\ref{lem:fromFtoF'} now completes the proof.
\end{proof}

Note that Lemma~\ref{eq:coarea} and Theorem~\ref{thm:convexify} above also hold if we define $E(t)$ in terms of a non-strict inequality instead: $E(t) := \{i\in V\colon u_i \geq t \}$.

\begin{remark}\label{rem: first variation for TVa q}
The function $\TVa^q(u)$ is a convex function, thus, taking $\partial \TVa^q(u)$ as the (possibly multivalued) subdifferential of $\TVa^q(u)$ \cite{EkelandTemam76}, any minimizer of
\begin{equation*}
    u \mapsto \TVa^q(u)+\langle u,g\rangle_{\mathcal{V}},	
\end{equation*}
for $g\in \mathcal{V}$, will solve the differential inclusion
\begin{equation*}
    \partial \TVa^q(u) \ni -g.
\end{equation*}
From the definition of $\TVa^q$ we see that its subdifferential is only multivalued at $u$ for which $\nabla u$ vanishes between two nodes, at all other $u$, $\TVa^q$ is pointwise differentiable. In particular, if $u,v\in \mathcal{V}$ and $\nabla u$ is never zero, we may differentiate\footnote{For further discussion and a generalization of this computation, see Appendix~\ref{sec:formalfirstvar}.}
\begin{align}
\left.\frac{d}{dt}\right|_{t=0}\TVa^q(u+tv) & =\left.\frac{d}{dt}\right|_{t=0}\left (\tfrac{1}{2}\sum\limits_{i,j}\omega_{ij}^q|u_i-u_j+t(v_i-v_j)| \right ),\notag \\
& = \tfrac{1}{2}\sum\limits_{ij}\omega_{ij}^q\sgn(u_i-u_j)(v_i-v_j),\notag \\
& = \langle \sgn(\nabla u),\nabla v\rangle_{\mathcal{E}}.\label{eq:variationofTV}
\end{align}
Since $\dvg$ is the adjoint of $\nabla$, it follows that, for all $v\in\mathcal{V}$,
\begin{equation*}
\frac{d}{dt}\left (\TVa^q(u+tv)+\langle u+tv,g\rangle_{\mathcal{V}}\right) = \langle \dvg(\sgn(\nabla  u))+g,v\rangle_{\mathcal{V}}.	
\end{equation*}
Therefore, the Euler-Lagrange equation for solutions of the minimization problem \eqref{eq:Fconvex} is
\[
\dvg \sgn(\nabla u) + \frac1{\eth t} sd^{\Sigma_n} + \mu_u - \mu_l= 0,
\]
provided $\nabla u$ is never zero. Here $\mu_u, \mu_l \in \mathcal{V}$ are the nonnegative Lagrange multipliers associated to the upper and lower bounds on $u\in \mathcal{V}_m$ respectively. Whenever $(\nabla u)_{ij}=0$ for some $i,j$, the above equation is replaced by a differential inclusion in terms of the subd\nobreak ifferential of the absolute value function.

Concretely and to recap, since the subdifferential of the absolute value function at $0$ is the interval $[-1,1]$, there exists $\phi \in \mathcal{E}$ such that for all $i,j\in V$, $|\phi_{ij}|\leq 1$,
\begin{equation*}
\dvg \phi+\frac{1}{\eth t}sd^{\Sigma_n} + + \mu_u - \mu_l = 0,	
\end{equation*}
and if $(\nabla u)_{ij} \neq 0$, then $\phi_{ij} = \sgn((\nabla u)_{ij})$.
\end{remark}

Fast computational methods for the solution of \eqref{eq:MCFh} based on max flow/min cut algorithms are developed in \cite{Chambolle05,ChambolleDarbon2009}. These methods exploit the homogeneity and submodularity of the total variational functional, $\TVa^q$. Those same papers also show it is possible to rewrite \eqref{eq:MCFh} as an ROF problem \cite{RudinOsherFatemi92}.

It would also be interesting to compare \eqref{eq:MCFh} to the finite difference scheme for motion of level sets by mean curvature on a lattice in \cite{Oberman04}.

\section{Threshold dynamics on graphs}\label{sec:MBO}
In this section we study the threshold dynamics or Merriman-Bence-Osher algorithm on a graph $G$. For a short overview of the continuum case we refer to Section~\ref{sec:continuumMBO} in Appendix~\ref{sec:continuum}.

\subsection{The graph MBO algorithm}
The MBO scheme on a graph, describing the evolution of a node subset $S\subset V$,  is given as follows.

\begin{asm}{(MBO$_\tau$)}
\KwData{An initial node subset $S_0 \subset V$,
a time step $\tau > 0$,
and the number of time steps $N>0$. }
\KwOut{A sequence of node sets $\{S_k\}_{k=1}^N$, which is the (MBO$_\tau$) evolution of $S_0$. }
\For{$ k = 1 \ \KwTo \  N$,}{
{\bf Diffusion step.} Let $v = e^{-\Delta \tau} \chi_{S_{k-1}} $ denote the solution at time $\tau$ of the  initial value problem
\begin{equation}
\label{eq:MBOHeat}
\dot{v} =  - \Delta v, \quad  v(0)  =  \chi_{S_{k-1}}.
\end{equation}
Here $\chi_{S}$ denotes the characteristic function of the set $S$. \\ \medskip

{\bf Threshold step.} Define the set $S_k \subset V$ to be
\[
S_k = \{ i\in V \colon v_i \geq \frac{1}{2} \}.
\]
}
\caption{\label{alg:MBO} The Merriman-Bence-Osher algorithm on a graph.}
\end{asm}

By the comparison principle, Lemma~\ref{lem:DiffusionGraph}(d), we note that the solution to
\eqref{eq:MBOHeat} satisfies $v(t) \in  [0,1]^n$ for all $t \in [0,\tau] $.

\begin{remark} \label{Rem:0or1} In the thresholding step of the \ref{alg:MBO} scheme, we have arbitrarily chosen to include the level set $\{i\in V: v_i=\frac12\}$ in the new set $S_k$, {\it i.e.}, the value at nodes $i$ for which $v_i = \frac{1}{2}$ is set to 1.
\end{remark}

An alternative description of the algorithm is as follows. Let $u_k$ be the indicator of the set $S_k$ as defined by the  \ref{alg:MBO} algorithm. If we define the thresholding function $P: \R \to \{0,1\}$,
which acts by thresholding
\[
P(x):= \begin{cases} 1 &\text{if } x\geq \frac12 \\ 0 &\text{if } x<\frac12 \end{cases},
\]
then the iterates can be succinctly written $u_k = (P e^{-\Delta \tau})^k u_0$.

Several papers use the MBO algorithm on a graph to approximate motion by mean curvature. For example, in \cite{MerkurjevKosticBertozzi2012,Garcia-CardonaMerkurjevBertozziFlennerPercus2013,HuLaurentPorterBertozzi13}, the MBO algorithm on graphs was implemented and used to study data clustering, community detection, segmentation, object recognition, and inpainting. This is accomplished by simply reinterpreting the Laplacian in \eqref{eq:MBO}, or in appropriate extensions of \eqref{eq:MBO}, as the graph Laplacian.

\subsection{The ``step-size'' $\tau$ in the \ref{alg:MBO} algorithm}
As discussed at the end of Section~\ref{sec:continuumMBO}, in a finite difference discretization of the continuum MBO algorithm, the time step $\tau$ must be chosen carefully to avoid trivial dynamics. For $\tau$ too small, there is not enough diffusion to change the value of $u$ at neighboring grid points beyond the threshold value. In this case, the solution is stationary under an \ref{alg:MBO} iteration and we say that the solution is \emph{frozen} or \emph{pinned}.  For $\tau$ too large there is so much diffusion that a stationary state is reached after one iteration in the MBO scheme. It is not surprising that these finite difference effects also appear for the MBO algorithm on graphs. From the form of the heat solution operator, $e^{\Delta \tau}$, we expect that $\tau$ should be roughly chosen in the interval $(\lambda_n^{-1}, \lambda_2^{-1})$. Theorems~\ref{thm:pinningMBO} and \ref{cor:trivialDynamics} strengthen this intuition.

The following theorem gives a lower bound  on the choice of $\tau$ to avoid freezing in the \ref{alg:MBO} algorithm  on general graphs.
\begin{theorem}\label{thm:pinningMBO}
Let $\rho$ be the spectral radius of the graph Laplacian, $\Delta$.
 Then the \ref{alg:MBO}  iterations on the graph
with initial set $S$ are stationary  if either of the two conditions are satisfied:
\begin{align}
\label{eq:taup}
& \tau < \tau_\rho(S) := \rho^{-1} \log \left( 1+\frac{1}{2} \ d_-^{\frac{r}{2}} \ ( \mathrm{vol} \ S)^{-\frac12}\right) \\
\nonumber
&\textrm{or}\\
\label{eq:tauk}
& \tau \leq \tau_\kappa(S) := \frac{1}{2 \| \Delta \chi_S \|_{\mathcal V, \infty}}.
 \end{align}
In particular, since $ \mathrm{vol} \ S > d_-^{\frac{r}{2}}$, if $\tau <  \log \frac{3}{2} \cdot \rho^{-1} \approx 0.4 \cdot \rho^{-1}$, then \eqref{eq:taup} implies  the MBO iterates  are pinned for any initial $S \subset V$.
\end{theorem}

\begin{proof} To prove \eqref{eq:taup}, let $\chi_S$ be the characteristic function on a set $S\subset V$. For a node to be added or removed from $S$ by one iteration of \ref{alg:MBO}, it is necessary that $\|e^{-\tau \Delta} \chi_S - \chi_S \|_{\mathcal V, \infty} \geq \frac{1}{2}$.
For the linear operator $A \colon \mathcal{V} \to \mathcal{V}$, let $\| A \|_{\mathcal V}$ be the operator norm induced by $\| \cdot \|_{\mathcal V}$, {\it i.e.},
$$
\| A \|_{\mathcal V} = \max_{u \in \mathcal V \setminus \{ 0\}} \frac{\| A u \|_{\mathcal V}}{\| u \|_{\mathcal V} }
$$
(see also Lemma~\ref{lem:Spectral}(b)).
Using Lemma~\ref{lem:normEquiv}, we compute
\begin{align*}
\|e^{-\tau \Delta}\chi_S  - \chi_S \|_{\mathcal V, \infty} &\leq
d_-^{-\frac{r}2} \ \| e^{ -\tau \Delta} \chi_S  - \chi_S \|_{\mathcal V}\\
&\leq d_-^{-\frac{r}{2}} \ \| e^{ -\tau \Delta}  -\mathrm{Id} \|_{\mathcal V}  \  \sqrt{\mathrm{vol} \ S }.
\end{align*}
Using the triangle inequality and the submultiplicative property of $\| \cdot \|_{\mathcal V}$ (see, {\it e.g.}, \cite{Horn1990}), we compute
$$
\| e^{-\tau \Delta} - \mathrm{Id} \|_{\mathcal V} \leq \sum_{k=1}^\infty \frac{1}{k!} (\tau \| \Delta \|_\mathcal V)^k = e^{\rho \tau}-1.
$$
 Thus, if $\tau < \rho^{-1} \log \left( 1+ \frac{1}{2} \ d_-^{\frac{r}{2}} \ ( \mathrm{vol} \ S)^{-\frac12}\right)$, all nodes are stationary under an \ref{alg:MBO} iteration.

 To prove \eqref{eq:tauk}, we write the solution to the heat equation at time $\tau$,
 $$
u(\tau) = e^{-\tau \Delta} \chi_S = \chi_S - \int_0^\tau \Delta u(t) dt.
$$
This implies
\begin{align*}
\| u(\tau) - \chi_S \|_{\mathcal V, \infty}
&= \left\| \int_0^\tau \Delta u(t) dt \right\|_{\mathcal V, \infty}
\leq \int_0^\tau \left\|  e^{-t \Delta} \Delta \chi_S \right\|_{\mathcal V, \infty} \,dt\\
&\leq \tau \left\|  \Delta \chi_S \right\|_{\mathcal V, \infty}.
\end{align*}
Here, we used the comparison principle, Lemma~\ref{lem:DiffusionGraph}(d). Thus, if $\tau \leq \frac{1}{2 \| \Delta \chi_S \|_{\mathcal V, \infty}}$, then $\| u(\tau) - \chi_S \|_{\mathcal V, \infty}  \leq \frac{1}{2}$,  implying  the \ref{alg:MBO} solution is stationary.
\end{proof}

\bigskip

The following corollary of Lemma~\ref{lem:DiffusionGraph}(c) shows that an upper bound on $\tau$ is necessary to avoid trivial dynamics.
\begin{theorem}
\label{cor:trivialDynamics}
Let the graph be connected. Consider the \ref{alg:MBO} algorithm with initial condition $\chi_S$, for a node set $S\subset V$. Assume $R_S := \frac{\mathrm{vol} S}{ \mathrm{vol} V} \neq \frac{1}{2}$.  If
\begin{equation}
\label{eq:taut}
\tau > \tau_t := \frac{1}{\lambda_2} \log \left( \frac{(\mathrm{vol} S)^\frac{1}{2} \ (\mathrm{vol} S^c)^\frac{1}{2}  } { (\mathrm{vol} V)^\frac{1}{2} \ | R_S -\frac{1}{2}| \ d_-^\frac{r}{2}}\right),
\end{equation}
 where the mass $M(u_0)$ is defined in \eqref{eq:mass}, then
$$
P e^{-\tau \Delta} u_0 = \begin{cases}
\chi_V & R_S >\frac{1}{2}, \\
0 & R_S <\frac{1}{2}. \\
\end{cases}
$$
\end{theorem}
\begin{proof}
In Lemma~\ref{lem:DiffusionGraph}(c), set $\e = | (\mathrm{vol} \ V)^{-1} M(u_0)-\frac{1}{2}| = | R_S-\frac{1}{2}|$. This implies that
$\| u(\tau ) - R_S \|_{\mathcal{V},\infty} \leq | R_S -\frac{1}{2}|$, as desired.
With $\e = | R_S-\frac{1}{2}|$, the condition on $\tau$ in Lemma~\ref{lem:DiffusionGraph}(c) is
\[
\tau > \frac{1}{\lambda_2} \log \left(| R_S-\frac{1}{2}|^{-1} \ d_-^{-\frac{r}{2}}  \  \| u_0 - R_S  \|_{\mathcal V}  \right).
\]
For $u_0 = \chi_S$, $\| u_0 - R_S  \|^2_{\mathcal V} = \frac{\mathrm{vol} S \ \mathrm{vol} S^c  } { \mathrm{vol} V }$.
\end{proof}
This corollary shows that, if $\tau$ is chosen too large, one iteration of \ref{alg:MBO} leads to a trivial state $u=\chi_V$ or $u=0$, which is stationary under the algorithm~\ref{alg:MBO}.

The following theorem gives a condition for which there is a gap between the lower and upper bound for $\tau$.

\begin{theorem}\label{thm:MBOgap}
Consider the \ref{alg:MBO} iterations on a graph with $n\geq 2$. Let $\tau_\rho$ and $\tau_t$ be defined as in \eqref{eq:taup} and \eqref{eq:taut}. If $\frac{\lambda_2}{\lambda_n} < \frac{ \log \sqrt{2}}{ \log \frac{3}{2}}  \approx 0.85$, then $\tau_\rho < \tau_t$.
\end{theorem}
\begin{proof}
Since $\left| R_S - \frac{1}{2} \right| \leq \frac{1}{2}$, $d_-^r \leq \mathrm{vol} S $, and $(\mathrm{vol} S) \ (\mathrm{vol} S^c) > d_-^r ( \mathrm{vol} V - d_-^r)$ we have
\begin{align*}
\tau_\rho &= \lambda_2^{-1} \log \left( 1+\frac{1}{2} \ d_-^{\frac{r}{2}} \ ( \mathrm{vol} \ S)^{-\frac12}\right) \geq\lambda_2^{-1} \  \log \left(2 \sqrt{ 1 - \frac{d_-^r}{ \mathrm{vol} V} } \right)\\
&\geq \lambda_2^{-1} \ \log \sqrt{2}.
\end{align*}
Since $d_-^r \leq \mathrm{vol} S $, we have
$$
\tau_t = \lambda_n^{-1} \ \log \left( 1 + \frac{1}{2} \sqrt{\frac{d_-^r}{ \mathrm{vol} S} }  \right)
\leq \lambda_n^{-1} \ \log \frac{3}{2}.
$$
The result follows. \end{proof}

\bigskip

Theorem~\ref{thm:MBOgap} further  reenforces our intuition that $\tau$ should be chosen in the interval
$(\lambda_n^{-1}, \lambda_2^{-1})$. If, for a particular graph, this interval is very small, then Theorems
\ref{thm:pinningMBO} and~\ref{cor:trivialDynamics} cannot provide an interval for which the \ref{alg:MBO} iterations has a chance of being non-stationary after the first iteration. Note however that the interval $[\tau_\rho, \tau_t]$ given by these theorems, is not necessarily a sharp interval for interesting dynamics.

\subsection{A Lyapunov functional for the graph~\ref{alg:MBO} algorithm} \label{sec:Lyapunov}
In this section we introduce a functional which is  decreasing on iterations of the \ref{alg:MBO} algorithm.
The analogous functional for the continuum setting was recently found in \cite{esedoglu2013}.
The functional is then used to show that the \ref{alg:MBO} algorithm with any initial condition converges to a stationary state in a finite number of iterations.

Let $\tau >0$ and consider the functional $J\colon \mathcal V \to \mathbb R$ defined by
\begin{equation}
\label{eq:Lyapunov}
J(u) = \langle 1-u, e^{-\tau \Delta} u \rangle_{\mathcal V}.
\end{equation}
Note that by, Lemma~\ref{lem:DiffusionGraph}(a), $J(u) = M(u) - \langle u, e^{-\tau \Delta} u \rangle_{\mathcal V}$, where $M$ is the mass from \eqref{eq:mass}.

\begin{lemma}
The functional $J$ defined in \eqref{eq:Lyapunov} has the following elementary properties.
\begin{enumerate}
\item $J$ is a strictly concave functional on $\mathcal V$.
\item $J$ is Fr\'echet differentiable with derivative in the direction $v$ given by
$$
L_u(v) := \Big\langle \frac{\delta J}{\delta u}\bigg|_u, v \Big\rangle_{\mathcal V}, \quad \mathrm{where} \quad  \frac{\delta J}{\delta u}\bigg|_u = 1-2 e^{-\tau \Delta} u.
$$
\end{enumerate}
\end{lemma}
\begin{proof}  We compute, for all $v \neq 0$,
$$
\frac{d^2}{d \alpha^2} J(u+\alpha v) = - 2 \langle v, e^{-\tau \Delta} v \rangle_{\mathcal V} < 0.
$$
Taking the first variation of $J(u) = \langle 1-u, e^{-\tau \Delta} u \rangle_{\mathcal V}$, we find that
\begin{align*}
\Big\langle \frac{\delta J}{\delta u}, \delta u \Big\rangle_{\mathcal V}
&:= \Big\langle 1-u , e^{-\tau \Delta} \delta u  \Big\rangle_{\mathcal V} - \Big\langle \delta u, e^{-\tau \Delta} u   \Big\rangle_{\mathcal V}\\
&=  \Big\langle 1-2 e^{-\tau \Delta} u, \delta v \Big\rangle_{\mathcal V},
\end{align*}
as desired.
\end{proof}

\medskip

Define the convex set $\mathcal K := \{ \phi \in \mathcal V\colon \forall j \in V \ \phi_j \in [0,1] \}$. Is it instructive to consider the optimization problem,
\begin{equation}
\label{eq:Jopt}
\min_{u \in \mathcal K} \ J(u).
\end{equation}
Since the objective function in \eqref{eq:Jopt} is concave and the admissible set is a compact and convex set,  it follows that the solution to \eqref{eq:Jopt} is attained by a vertex function $u \in  \mathcal B :=  \{ v \in \mathcal V\colon \forall j \in V\ v_j \in \{0,1\}\}$. Here $\mathcal B$ is the set of binary vertex functions, taking the value $0$ or $1$ on each vertex.
The sequential linear programming approach to solving the system
\eqref{eq:Jopt} is to consider
a  sequence of vertex functions $\{u_k\}_{k=0}^\infty$ which satisfies
\begin{equation}
\label{eq:SLP}
u_{k+1} = \arg \min_{v\in  K} \ L_{u_k} (v), \qquad u_0 = \chi_S, \text{ for a node set } S \subset V.
\end{equation}
The optimization problem in \eqref{eq:SLP} may not have  a unique solution, so the iterates are not well-defined. The following proposition shows that the iterations of the \ref{alg:MBO} algorithm define a unique sequence satisfying \eqref{eq:SLP}. Note that the optimization problem in \eqref{eq:SLP} is the minimization of a linear objective function over a compact and convex set, implying that for any sequence $\{u_k\}_{k=0}^\infty$ satisfying  \eqref{eq:SLP},  $u_k \in \mathcal B$ for all $k\geq 0$.

\begin{prop} \label{prop:Lyapunov}
The iterations defined by the \ref{alg:MBO} algorithm satisfy \eqref{eq:SLP}. The functional $J$, defined in \eqref{eq:Lyapunov}, is non-increasing on the iterates $\{ u_k \}_{k=1}^\infty$, {\it i.e.}, $J(u_{k+1}) \leq J(u_k)$, with equality only obtained if $u_{k+1} = u_k$. Consequently, the \ref{alg:MBO} algorithm with any initial condition converges to a stationary state in a finite number of iterations.
\end{prop}
\begin{proof} At each iteration $k$, the objective functional $L_{u_k}$ is linear and thus the minimum is attained by a function
$$
u_{k+1} = \begin{cases}
1 & \text{ if } 1- 2 e^{-\tau \Delta} u_k \leq 0, \\
0 & \text{ if }  1- 2 e^{-\tau \Delta} u_k > 0
\end{cases}
 \quad = \quad  \chi_{\{e^{-\tau \Delta}u_k \geq \frac{1}{2} \} }.
$$
These are precisely the \ref{alg:MBO} iterations. By the strict concavity of $J$ and linearity of $L_{u_k}$, for $u_{k+1} \neq u_k$,
$$
J(u_{k+1})- J(u_k) <  L_{u_k} (u_{k+1} - u_k) =  L_{u_k} (u_{k+1}) -  L_{u_k} (u_k).
$$
Since $u_k \in \mathcal K$, $L_{u_k} (u_{k+1}) \leq  L_{u_k} (u_k)$ which implies $J(u_{k+1}) < J(u_k)$. The convergence of the algorithm in a finite number of iterations then follows from the fact that $\mathcal B$ contains only a finite number of points, the vertices of the unit $n$-cube.
\end{proof}

Proposition~\ref{prop:Lyapunov} shows that $J$ is a  \emph{Lyapunov function} for the \ref{alg:MBO} iterates. From the proof of Proposition~\ref{prop:Lyapunov}, we also note that the non-uniqueness of the iterates in \eqref{eq:SLP} corresponds to the choice in the \ref{alg:MBO} algorithm of thresholding vertices $\{ j\in V \colon e^{-\tau \Delta}u_k = \frac{1}{2} \}$ to either 0 or 1 (see Remark~\ref{Rem:0or1}).

\begin{remark}
The framework of \cite{esedoglu2013} easily allows for the extension of the MBO algorithm to more phases, however we do not pursue these ideas here.
\end{remark}

\subsection{A local guarantee for a `nonfrozen' \ref{alg:MBO} iteration}\label{sec:connections}
We begin by observing that the constant $\tau_\kappa$ in Theorem~\ref{thm:pinningMBO} depends on the maximum curvature $\kappa_S^{1,r}$ of the indicator set in the graph, $\| \Delta \chi_S \|_{\mathcal V,\infty}$. 
In this section, we prove a theorem which gives a condition on $\tau$ in terms of the local curvature $(\kappa_S^{1,r})_i$ at a node $i$, which guarantees that the value of $u$ on that node will change in one iteration of the graph~\ref{alg:MBO} scheme. 

We  first introduce some notation which is needed to state the theorem. Recall that the set of neighbors of a node $i\in V$ is $\mathcal{N}_i = \{j\in V: \omega_{ij}>0\}$. Let $1\in V$ be an arbitrary node in the graph $G$ and let $S\subset V$. We define the sets
\[
S_1 := \begin{cases} \mathcal{N}_1 \cap S^c & \text{if } 1 \in S,\\ \mathcal{N}_1 \cap S & \text{if } 1 \not\in S,\end{cases} \quad \text{and} \quad \overline{S_1} := S_1 \cup \{1\}.
\]
The set $S_1$ contains neighbors of node $1$ which are also in either the boundary $\partial(S^c)$ or $\partial S$ (depending on whether or not $1\in S$).
For $u\in \mathcal{V}$, define $\Delta'$ as
\[
(\Delta' u)_i := \begin{cases} d_i^{-r} \sum_{j\in \overline{S_1}} \omega_{ij} (u_i-u_j) & \text{if } i\in \overline{S_1},\\ 0 & \text{if } i\not\in \overline{S_1}.\end{cases}
\]
We see that $\Delta'$ on $\overline{S_1}$ is similar to the Laplacian on the subgraph induced by $\overline{S_1}$, with the important distinction that, for each $i\in V$, the degree $d_i$ is the degree of $i$ in the full graph $G$, not the degree in the subgraph induced by $\overline{S_1}$.
In \cite[Section 8.4]{Chung:1997}, $\Delta'$ is referred to as the Laplacian with Dirichlet conditions on $\partial (\overline{S_1}^c)$.
 If $v \in \mathcal{V}_1:= \{v\in \mathcal{V}: v=0 \text{ on } \overline{S_1}^c\}$, then
\[
(\Delta' v)_i = \begin{cases} (\Delta v)_i &\text{if } i\in \overline{S_1},\\ 0 & \text{if } i\not\in \overline{S_1}.\end{cases}
\]
Note in particular that, if $v\in \mathcal{V}_1$, then $e^{-t\Delta'} v \in \mathcal{V}_1$ for all $t\geq0$.

\begin{theorem}\label{thm:localtau}
Let $1\in V$ be an arbitrary node and $S\subset V$ be such that $|(\kappa_S^{1,r})_1|^2 > \|(\Delta')^2 \chi_{S_1}\|_{\mathcal{V},\infty}$. If $\tau \in (\tau_1, \tau_2)$, where
\begin{align}\label{eq:tau12}
\tau_{1,2} &:= \frac1{\|(\Delta')^2 \chi_{S_1}\|_{\mathcal{V},\infty}} \left(|(\kappa_S^{1,r})_1| \pm \sqrt{|(\kappa_S^{1,r})_1|^2 - \|(\Delta')^2 \chi_{S_1}\|_{\mathcal{V},\infty}}\right)\notag \\ &> 0,
\end{align}
then
\[
|(P e^{-\tau \Delta} \chi_S)_1 - (\chi_S)_1| = 1.
\]
That is, the phase at node $1$ changes after one \ref{alg:MBO} iteration.
\end{theorem}
It is important to note that both $|(\kappa_S^{1,r})_1|^2$ and $\|(\Delta')^2 \chi_{S_1}\|_{\mathcal{V},\infty}$ are local quantities, in the sense that they only depend on the structure of $G$ at node 1 and its  neighbors.
This is in contrast with Theorem~\ref{thm:pinningMBO}, which depends on the spectrum of the Laplacian on $G$. The existence of a lower bound $\tau_1$ on $\tau$ is unsurprising in the light of this earlier freezing result. The necessity for an upper bound $\tau_2$ can be understood from our wish to only use local quantities in this theorem.

\begin{proof}[Proof of Theorem~\ref{thm:localtau}]
First we assume that $1\not\in S$, so $S_1 = \mathcal{N}_1 \cap S$. By the comparison principle in Lemma~\ref{lem:DiffusionGraph}(d), $\chi_{S_1} \leq \chi_S$ on $V$ implies
$(e^{-\tau \Delta} \chi_{S_1})_1 \leq  (e^{-\tau \Delta} \chi_S)_1$. In particular, since $(\chi_{S_1})_1=(\chi_S)_1 = 0$, we have
\[
(e^{-\Delta \tau} \chi_S - \chi_S)_1 \geq (e^{-\tau \Delta} \chi_{S_1} - \chi_{S_1})_1.
\]

Let $v$ satisfy the heat equation with Dirichlet boundary data,
\[
\left\{ \begin{array}{ll}
\dot v = -(\Delta' v)_i,\\
v(0) = \chi_{S_1}.
\end{array}\right.
\]
As noted above the theorem, $v(t)\in \mathcal{V}_1$ for all $t\geq 0$.

It is easily checked that $v$ is subcaloric, {\it i.e.}, $\dot v_i \leq -(\Delta v)_i$ for all $i\in V$, and $v(0) \leq \chi_S$. In addition, the Laplacian satisfies $-(\Delta u)_i \leq -(\Delta \tilde u)_i$ if $u_i = \tilde u_i$ and $u_j \leq \tilde u_j$, for $j\neq i$. Hence, by the theory of differential inequalities (see for example \cite[Theorem 8.1(3)]{Szarski65}),
\[
v_i(t) \leq \left(e^{-t\Delta} v(0)\right)_i = \left(e^{-t\Delta} \chi_{S_1}\right)_i, \quad \text{for all } i\in V.
\]
In particular,
\begin{align*}
\left(e^{-\tau \Delta} \chi_{S_1} - \chi_{S_1}\right)_1 &\geq v_1(\tau) - v_1(0) = \left(e^{-\tau \Delta'}\chi_{S_1} - \chi_{S_1}\right)_1\\
&= -\tau \left(\Delta' \chi_{S_1}\right)_1 + \tau^2 r(\tau),
\end{align*}
where
\[
|r(\tau)| \leq \frac12 \underset{t\in[0,\tau]}\sup\, \left(e^{-t\Delta'} (\Delta')^2 \chi_{S_1}\right)_1 \leq \frac12 \|(\Delta')^2 \chi_{S_1}\|_{\mathcal{V},\infty}.
\]
Note that $-\left(\Delta' \chi_{S_1}\right)_1 = -(\kappa_S^{1,r})_1 = |(\kappa_S^{1,r})_1|$, where the last equality follows because $1\not\in S$.

We conclude that
\begin{align*}
(e^{-\tau \Delta} \chi_S - \chi_S)_1 &\geq \left(e^{-\tau \Delta} \chi_{S_1} - \chi_{S_1}\right)_1 \\
&\geq  |(\kappa_S^{1,r})_1| \tau - \frac12 \|(\Delta')^2 \chi_{S_1}\|_{\mathcal{V},\infty} \tau^2,
\end{align*}
hence
\[
(e^{-\tau \Delta} \chi_S - \chi_S)_1 \geq \frac12 \Leftrightarrow \tau \in [\tau_1, \tau_2],
\]
which proves the result for the case in which $1\not\in S$.

To prove the desired statement if $1\in S$, we note that
\[
(e^{-\tau \Delta} - 1)(\chi_S+\chi_{S^c}) = 0,
\]
so the condition $(e^{-\tau \Delta}\chi_S - \chi_S)_1 < -\frac12$ is equivalent to $(e^{-\tau \Delta}\chi_{S^c} - \chi_{S^c})_1 > \frac12$. Recall that, in this case, $S_1 = \mathcal{N}_1 \cap S^c$, and the same derivation as above holds\footnote{In particular, carefully note that now $-\left(\Delta' \chi_{S_1}\right)_1 = (\kappa_S^{1,r})_1 = |(\kappa_S^{1,r})_1|$ holds.}, since $1\not\in S^c$, with the exception that the admissible range of $\tau$ becomes the open interval $(\tau_1, \tau_2)$. This is because, by our definition of the \ref{alg:MBO} algorithm, the thresholding operator thresholds the $\frac12$-level set to $1$.

\end{proof}

In the remainder of this section we  determine some conditions under which the requirement $|(\kappa_S^{1,r})_1|^2 > \|(\Delta')^2 \chi_{S_1}\|_{\mathcal{V},\infty}$ in Theorem~\ref{thm:localtau} is satisfied. To this end, define the reduced degrees, for $i\in V$, as
\begin{equation}\label{eq:reduceddegrees}
d'_i := \sum_{j\in S_1} \omega_{ij}.
\end{equation}

\begin{lemma}\label{lem:gapcondition}
$|(\kappa_S^{1,r})_1|^2 > \|(\Delta')^2 \chi_{S_1}\|_{\mathcal{V},\infty}$ if and only if
\begin{equation}\label{eq:gapcondition}
d_1^{-2r} (d_1')^2 > \underset{i\in \overline{S_1}}\max\, d_i^{-r} \left|-d_i^{1-r} d_i' - \sum_{j\in S_1} d_j^{1-r} \omega_{ij} + \sum_{k\in \overline{S_1}} d_k^{-r} d_k' \omega_{ik}\right|.
\end{equation}
\end{lemma}
\begin{proof}
Consider the $|V| \times |V|$ matrix corresponding to $\Delta'$. After possibly relabeling the nodes, it can be written as $\left(\begin{array}{cc} L' & 0\\ 0 & 0\end{array}\right)$, where $L'$ is the $|\overline{S_1}| \times |\overline{S_1}|$ matrix with entries
\[
L'_{ij} = d_i^{-r} \begin{cases} -d_i & \text{if } i=j,\\ \omega_{ij} & \text{if } i\neq j.\end{cases}
\]
Then
\begin{align*}
(L'^2)_{ij} &= \sum_{k\in \overline{S_1}} (L')_{ik} (L')_{kj}\\
&= (L')_{ii} (L')_{ij} + (L')_{ij} (L')_{jj} + \sum_{k\in \overline{S_1}\setminus\{i, j\}}(L')_{ik} (L')_{kj}\\
&= d_i^{-r} \left[-(d_i^{1-r} + d_j^{1-r}) \omega_{ij} + \sum_{k\in \overline{S_1}\setminus\{i, j\}} d_k^{-r} \omega_{ik} \omega_{jk}\right].
\end{align*}
Thus $\left((\Delta')^2 \chi_{S_1}\right)_i = 0$ if $i\not\in \overline{S_1}$, and, for $i \in \overline{S_1}$, we have
\begin{align*}
\left((\Delta')^2 \chi_{S_1}\right)_i &= d_i^{-r} \sum_{j\in S_1} \left[-(d_i^{1-r} + d_j^{1-r}) \omega_{ij} + \sum_{k\in \overline{S_1}\setminus\{i, j\}} d_k^{-r} \omega_{ik} \omega_{jk}\right]\\
&= d_i^{-r} \sum_{j\in S_1} \left[-(d_i^{1-r} + d_j^{1-r}) \omega_{ij} + \sum_{k\in \overline{S_1}} d_k^{-r} \omega_{ik} \omega_{jk}\right]\\
&= d_i^{-r} \left[ -d_i^{1-r} d'_i - \sum_{j\in S_1} d_j^{1-r} \omega_{ij} + \sum_{k\in\overline{S_1}} d_k^{-r} d'_k \omega_{ik}\right],
\end{align*}
where, for the second equality, we used that $\omega_{ii}=\omega_{jj}=0$.

From \eqref{eq:dvgnu} and the definition of $S_1$, we find
\[
|(\kappa_S^{1,r})_1|^2 = d_1^{-2r} \sum_{j,k \in S_1} \omega_{1j} \omega_{1k} = d_1^{-2r} d_1'^2.
\]
\end{proof}

\begin{corol}\label{cor:r=1}
Let $r=1$ and $d'_1 >0$. If there exists an $0\leq \e<1$, such that, for all $i \in S_1$,
\begin{equation}\label{eq:condr=1}
\frac{d'_i}{d_i} \leq \e \frac{d'_1}{d_1} \quad \text{and} \quad \frac{\omega_{i1}}{d_i} < (1-\e^2) \frac{d'_1}{d_1},
\end{equation}
then condition \eqref{eq:gapcondition} is satisfied, and there exist $0<\tau_1< \tau_2$ as in Theorem~\ref{thm:localtau}.

If the first condition in \eqref{eq:condr=1} is satisfied with
\[
0\leq \e < \frac12 \left(\sqrt{5}-1\right) \approx 0.618,
\]
then the second condition in \eqref{eq:condr=1} can be replaced by the condition that, for all $i\in S_1$, $\omega_{i1} \leq d'_i$.
\end{corol}
Note, by \eqref{eq:dvgnu} and \eqref{eq:reduceddegrees}, that the conditions in \eqref{eq:condr=1} can be rewritten as
\[
(\kappa_{S_1}^{1,1})_i - 1 \geq \e (\kappa_{S_1}^{1,1})_1 \quad \text{and} \quad (\kappa_{\{1\}}^{1,1})_i > (1-\e^2) (\kappa_{S_1}^{1,1})_1,
\]
for all $i\in S_1$.

\begin{proof}[Proof of Corollory~\ref{cor:r=1}]
For $r=1$, we compute, for $i\in \overline{S_1}$,
\begin{align*}
&\hspace{0.5cm} d_i^{-r} \left|-d_i^{1-r} d_i' - \sum_{j\in S_1} d_j^{1-r} \omega_{ij} + \sum_{k\in \overline{S_1}} d_k^{-r} d_k' \omega_{ik}\right| \\
&  = \left| \frac{d'_i}{d_i} - \frac{d'_i}{d_i} + \sum_{k\in \overline{S_1}} \frac{d_k'}{d_k} \frac{\omega_{ik}}{d_i}\right| \\
&  = \sum_{k\in \overline{S_1}} \frac{d_k'}{d_k} \frac{\omega_{ik}}{d_i},
\end{align*}
hence condition \eqref{eq:gapcondition} becomes $\displaystyle
\left(\frac{d'_1}{d_1}\right)^2 > \underset{i\in\overline{S_1}}\max\, \sum_{k\in \overline{S_1}} \frac{d_k'}{d_k} \frac{\omega_{ik}}{d_i}$.
If the first condition in \eqref{eq:condr=1} is satisfied, we have
\begin{align*}
\sum_{k\in \overline{S_1}} \frac{d_k'}{d_k} \frac{\omega_{ik}}{d_i} &= \frac{d'_1}{d_1} \frac{\omega_{i1}}{d_i} + \sum_{k\in S_1} \frac{d_k'}{d_k} \frac{\omega_{ik}}{d_i} \leq  \frac{d'_1}{d_1} \frac{\omega_{i1}}{d_i} + \e \frac{d'_1}{d_1} \frac{d'_i}{d_i}\\
&= \frac{d'_1}{d_1} \left[\frac{\omega_{i1}}{d_i}  + \e \frac{d'_i}{d_i}\right].
\end{align*}
Since $d'_1>0$, condition \eqref{eq:gapcondition} reduces to
$\displaystyle
\underset{i\in\overline{S_1}}\max\, \left[\frac{\omega_{i1}}{d_i}  + \e \frac{d'_i}{d_i}\right] < \frac{d'_1}{d_1}.
$
If the maximum is achieved at $i=1$, then $\displaystyle \frac{\omega_{i1}}{d_i}  + \e \frac{d'_i}{d_i} = \e \frac{d'_1}{d_1} < \frac{d'_1}{d_1}$. If, on the other hand, the maximum is achieved at some $i\in S_1$, then $\displaystyle \frac{\omega_{i1}}{d_i}  + \e \frac{d'_i}{d_i} \leq \frac{\omega_{i1}}{d_i}  + \e^2 \frac{d'_1}{d_1} <  \frac{d'_1}{d_1}$. Here, we first used the first condition from \eqref{eq:condr=1}, and then the second. Combined with Theorem~\ref{thm:localtau} and Lemma~\ref{lem:gapcondition}, this proves the first claim.

\smallskip

Now assume, instead of the second condition in \eqref{eq:condr=1}, that $\omega_{i1} \leq d'_i$ for all $i\in S_1$. Then $\frac{\omega_{i1}}{d_i} \leq \frac{d'_i}{d_i} \leq \e \frac{d'_1}{d_1}$. Hence, if $0<\e < 1-\e^2$, the second condition in \eqref{eq:condr=1} is satisfied. This requirement is met, if $0<\e < \frac12 \left(\sqrt5-1\right)$.
\end{proof}


It is worthwhile to understand the conditions in the corollary. Adding the two conditions in \eqref{eq:condr=1} gives $\frac{d'_i + \omega_{i1}}{d_i} \leq (1 + \e - \e^2) \frac{d'_1}{d_1}$. The ratio $\frac{d'_i+\omega_{i1}}{d_i}$ is a measure of the relative strength of connection of node $i$ within the set $\overline S_1$ compared to all its connections in $G$. Similarly $\frac{d'_1}{d_1}$ is the relative strength of connection of node $1$ within $\overline S_1$ (or, equivalently, within $S_1$), compared to all its connections in $G$. The conditions in \eqref{eq:condr=1} thus require node 1 to be a node with comparatively large relative connection strength within $\overline S_1$, compared to the other nodes in $\overline{S_1}$. This will allow enough mass to diffuse to or away from node $1$ (depending on whether or not $1\in S$), for it to pass the threshold value $\frac12$, without too much of the locally available mass diffusing to other nodes.

Some examples in Section~\ref{sec:examples} further examines the conditions in \eqref{eq:gapcondition} and \eqref{eq:condr=1}.

\begin{remark}
One can interpret the condition
\[
|(\kappa_S^{1,r})_1|^2 > \|(\Delta')^2 \chi_{S_1}\|_{\mathcal{V},\infty}
\]
from Theorem~\ref{thm:localtau}, in terms of the spectral radius of $\Delta'$ (similar to \cite[Equation (8.7)]{Chung:1997}). We compute
\begin{align*}
\rho(\Delta') &= \sup_{u \in \mathcal{V}\setminus \{0\} } \frac{ \langle u, \Delta' u\rangle_\mathcal V }{ \| u \|_\mathcal{V}} = \frac12 \sup_{u \in \mathcal{V}\setminus \{0\} } \frac{\sum_{i,j\in \overline{S_1}} \omega_{ij} (v_i-v_j)^2}{\sum_{i\in \overline{S_1}} d_i^r v_i^2}\\
&\leq \sup_{u \in \mathcal{V}\setminus \{0\} } \frac{\sum_{i,j\in \overline{S_1}} \omega_{ij} (v_i^2+v_j^2)}{\sum_{i\in \overline{S_1}} d_i^r v_i^2} = 2 \sup_{u \in \mathcal{V}\setminus \{0\} } \frac{\sum_{i,j\in \overline{S_1}} \omega_{ij} v_i^2}{\sum_{i\in \overline{S_1}} d_i^r v_i^2}\\
&= 2 \sup_{u \in \mathcal{V}\setminus \{0\} } \frac{\sum_{i\in \overline{S_1}} \bar d_i v_i^2}{\sum_{i\in \overline{S_1}} d_i^r v_i^2} \leq 2 (\bar d_-)^{-r} \bar d_+,
\end{align*}
where
\[
\bar d_i := \sum_{j\in \overline{S_1}} \omega_{ij}, \quad \bar d_+ := \underset{i\in\overline{S_1}}\max\, \bar d_i, \quad \bar d_- := \underset{i\in\overline{S_1}}\min\, \bar d_i.
\]
Because $\left((\Delta')^2 \chi_{S_1}\right)_i = 0$ if $i\not\in \overline{S_1}$, it is straightforward to adapt the proof of Lemma~\ref{lem:normEquiv}, to find $(\bar d_-)^\frac{r}2 \|(\Delta')^2 \chi_{S_1} \|_{\mathcal V,\infty} \leq \|(\Delta')^2 \chi_{S_1} \|_{\mathcal V}$.

Combining these results, the condition
\[
|(\kappa_S^{1,r})_1|^2 > \|(\Delta')^2 \chi_{S_1}\|_{\mathcal{V},\infty}
\]
is satisfied if
\[
d_1^{-2r} (d'_1)^2 > 4 (\bar d_-)^{-\frac{r}2} (\bar d_-)^{-2r} (\bar d_+)^2 \sqrt{\mathrm{vol} \ S_1},
\]
or, equivalently, if
\[
\frac14 \left(\frac{d'_1}{\bar d_+}\right)^2 \left(\frac{\bar d_-}{d_1}\right)^{2r} > (\bar d_-)^{-\frac{r}2} \sqrt{\mathrm{vol}\ S_1}.
\]
Using $\mathrm{vol}\ S_1= \sum_{i\in S_1} d_i^r \geq |S_1| d_-^r$, we can deduce the stronger sufficient condition
\[
\frac14 \left(\frac{d'_1}{\bar d_+}\right)^2 \left(\frac{\bar d_-}{d_1}\right)^{2r} > \left(\frac{d_-^r}{\bar d_-}\right)^{-\frac{r}2} |S_1|.
\]
\end{remark}

\section{Allen-Cahn equation on graphs}\label{sec:AC}

In this section, we investigate the Allen-Cahn equation on graphs. A short overview of the continuum Allen-Cahn equation can be found in Section~\ref{sec:continuumAC} in Appendix~\ref{sec:continuum}.

    We propose the following Allen-Cahn equation (ACE) on graphs, for all $i\in V$:
    \begin{equation}\label{eq:ACEe}\tag{ACE$_\e$}
         \left\{\begin{array}{clc} \dot{u}_i &= -(\Delta u)_i - \frac1\e d_i^{-r} W'(u_i) & \text{for } t>0,\\
         u_i & =(u_0)_i &\text{at } t=0,
         \end{array}\right.
    \end{equation}
    for a given initial condition $u_0 \in \mathcal{V}$ and $\e>0$. Here $W\in C^2(\R)$ is a double well potential. For definiteness we set $W$ to be the standard double well potential $W(u) = (u+1)^2 (u-1)^2$, hence $W'(u)=4u(u^2-1)$ and $W$ has two stable minima at the wells at $u=\pm 1$ and an unstable local maximum at $u=0$.
Recall that the sign convention for the Laplacian is opposite to the one used in the continuum literature. Note that for $\e $ sufficiently small, this system has $3^{n}$ equilibria, of which $2^{n}$ are stable.

   The \ref{alg:MBO} algorithm is closely related to time-splitting methods applied to the Allen-Cahn evolution \eqref{eq:ACEe}. The diffusion step is precisely the time evolution with respect to the first term of \eqref{eq:ACEe} and the thresholding step is the asymptotic behavior of evolution with respect to the second term of \eqref{eq:ACEe}.

    The case $V=\mathbb{Z}^d$ with weights $\omega_{ij}=\omega(\|i-j\|)$ was considered in \cite{BatesChmaj99}, where it is seen as an approximation to the Ising model, and stationary solutions and traveling waves are constructed for $\e$ small enough. The authors note that, when $\omega_{ij}$ corresponds to nearest-neighbors, this equation is known as the discrete Nagumo equation, which is a simplified model of neural networks. In this context, \cite{HupkesPelinovskySandstede11} considered the Nagumo equation in $\mathbb{Z}^1$ and derived the existence of traveling waves. In general, we are not aware of any previous works where \eqref{eq:ACEe} is considered for an arbitrary weighted graph $(V,E,\omega_{ij})$. It would be interesting to see whether the analysis in \cite{BatesChmaj99} can can be extended to use \eqref{eq:ACEe} to study phase transitions in general graphs, a topic of interest in other areas of mathematics \cite{Lyons00}.

    Just as in the continuum case, we arrive at \eqref{eq:ACEe} as the gradient flow given by the graph Ginzburg-Landau functional,\\ $GL_\e(u) \colon \mathcal{V} \to \mathbb R$,
    \begin{equation}\label{eq:GL functional}
         \tag{GL$_\e$}
         GL_\e(u) := \frac12 \|\nabla u\|_{\mathcal{E}}^2 + \frac1\e \langle D^{-r} W\circ u, 1\rangle_{\mathcal{V}},
    \end{equation}
    where $(D^{-r} W\circ u)_i = d_i^{-r} W(u_i)$, and whose first variation is given by
    \[
         \left.\frac{d}{dt} GL_\e(u+tv) \right|_{t=0} = \langle \Delta u, v\rangle_{\mathcal{V}} + \frac1\e \langle D^{-r} W'(u), v\rangle_{\mathcal{V}}.
    \]
    The factor $d_i^{-r}$ in the potential term is needed to cancel the factor $d_i^r$ in the $\mathcal{V}$-inner product. Equation \eqref{eq:ACEe} is then the $\mathcal{V}$-gradient flow associated with \eqref{eq:GL functional}.

    Recall that the Laplacian $\Delta$ also depends on $r$. In fact, the equation in \eqref{eq:ACEe} can be rewritten as
    \[
         d_i^r \dot u_i = -\sum_{j\in V} \omega_{ij} (u_i-u_j) - \frac1\e W'(u_i),
    \]
    showing that the factor $d_i^r$ can be interpreted as a node-dependent time rescaling.

    By standard ODE arguments and the smoothness of the right hand side of \eqref{eq:ACEe}, for each $\e>0$ a unique $C^1$ solution to  \eqref{eq:ACEe} exists for all $t>0$.

    This continuum case (see Appendix~\ref{sec:continuumAC}) suggests an approach for finding a valid notion of mean curvature (and its flow) for graphs: Take initial data 
    $u(0) = \chi_S-\chi_{S^c}$, for some node set $S\subset V$, and consider the corresponding solution $u^\e(t)$ to \eqref{eq:ACEe}, for all times $t>0$. The question is whether the limit
         \begin{equation*}
	          \bar u_i(t):= \lim \limits_{\e \to 0_+} u^\e_i(t)	
         \end{equation*}
         exists. Even if it does, it is unlikely that $\bar u$ is of the form $\chi_{S(t)}-\chi_{S(t)^c}$, which can be interpreted as a binary indicator function for some evolving set $S(t)$, for all times $t>0$, but there may be an approximate phase separation: $\bar u_i(t) \in [-1-\delta, -1+\delta] \cup [1-\delta, 1+\delta]$, for some small $\delta>0$. Is there a way to characterize the evolution of the ``interface'' between the two level sets of $\bar u_i(t)$?

    However, a little analysis shows the above approach is rather na\"ive. Indeed, unlike in the continuum case, the graph Laplacian of the indicator function of a set $S \subset V$ is always a well-defined bounded function (in any norm). Thus, for small $\e$ the potential term in the equation will dominate the dynamics, and {\it pinning} or {\it freezing} will occur, as proven in Theorem~\ref{thm:pinningAC}. This is the dynamics in which the sign of the value of $u$ on each node is fixed by the sign of the initial value, and $u$ at each node just settles into the corresponding well of $W$.

	As discussed at the start of Section~\ref{sec:MCF}, the question how to connect the sequence of sets evolving by graph mean curvature to the super (or sub) level sets $\{ i\in V\colon u^\e_i(t)>0\}$ for solutions of \eqref{eq:ACEe}, is still open. See also Question~\ref{ques:ACMCF}.
	
	\begin{remark} Note that in the \ref{alg:MBO} algorithm, the values of $u$ are reinitialized in every iteration to 0 or 1. Our choice of the double well potential $W$ in \eqref{eq:ACEe} has two equilibria corresponding to the level sets for $\pm 1$. Correspondingly, the unstable equilibrium for  \ref{alg:MBO} corresponds to the 1/2 level set, while for \eqref{eq:ACEe} it corresponds to the 0 level set. This agrees with the now standard notations for Allen-Cahn and MBO.
\end{remark}
	
 	Below, we show that for all $\e$ below a finite $\e_0>0$ the functions $u^\e_i(t)$ do not change sign as $t$ varies, so that pinning occurs. Recall that a set which contains the forward orbit of each of its elements is called \emph{positively invariant}, and that the number of nodes in the graph $G$ is $|V| = n$.
	\begin{lemma} \label{lem:positiveinvariant}
		 Consider the set $S := \{ u\in \mathcal{V} \colon \|u\|_{\mathcal{V}}^2 \leq \frac{17}{4} n \, d_+^r \}$ and let $u(t)$ be the solution to \eqref{eq:ACEe} for a given $\e>0$. Then $t \mapsto \|u(t)\|_{\mathcal{V}}^2$ is decreasing at each $t$ such that $u(t) \in S^c$. As a consequence, the set $S$ is positively invariant and every trajectory of \eqref{eq:ACEe} enters $S$ in finite time.
	\end{lemma}

	\begin{proof}
		 Define the set $A(t):= \{ i\in V \colon u_{i}^{2}(t) \leq 2 \}$.
		 We compute
	     \begin{align*}
	          \tfrac{d}{dt} \|u(t)\|_{\mathcal{V}}^2 &= 2 \left < u(t), \dot u(t) \right >_{\mathcal{V}} \\
	                       & = - 2 \| \nabla u(t) \|_{\mathcal{E}}^2 - \tfrac{8}{\e} \sum_{i\in V}u_i(t)^2(u_i(t)^2-1)\\
	                       & = -2\|\nabla u(t)\|_{\mathcal{E}}^2- \tfrac{8}{\e} \sum_{i\in A^c(t) } u_i(t)^{2} (u_i(t)^{2} -1 )\\
  &\hspace{0.6cm} + \tfrac{8}{\e}\sum_{i\in A(t) } u_i(t)^{2} (1 - u_i(t)^{2}  )  \\
	                       & < -\tfrac{8}{\e}\left ( \sum_{i\in A^c(t) } u_i(t)^{2} - \tfrac{ |A(t)|}{4}  \right ).
	     \end{align*}
	     The last inequality follows, since $u_i(t)^2 - 1 > 1$ for $i\in A^c(t)$, and $\max\{ x^2 (1-x^2)\colon x^2 \leq 2\} = \frac14$.	
	     Note that $\|u(t)\|_{\mathcal{V}}^2 \leq d_+^r \sum_{i\in V} u_i(t)^2$. Thus, if $u(t) \in S^c$, then $\sum_{i\in V} u_i(t)^2 > \frac{17}4 n$, and hence
         \begin{equation*}
              \sum_{i\in A^c(t) } u_i(t)^{2} =  \sum_{i\in V} u_i(t)^2-\sum_{i\in A(t) }u_i(t)^2 > \tfrac{17}{4}n - 4 |A(t)|.
         \end{equation*}
         Therefore,
         \begin{align*}
              \tfrac{d}{dt} \|u(t)\|_{\mathcal{V}}^2 & < -\tfrac{8}{\e}	\left ( \tfrac{17}{4}n - 4 |A(t)| - \tfrac{1}{4} |A(t)| \right ) < 0,
         \end{align*}
         where we have used that $|A(t)| \leq n$.
	     This shows  $\|u(t)\|_{\mathcal{V}}^2$ is decreasing in the region $S^{c}$, as desired. The other statements in the lemma now follow.
	\end{proof}

    \begin{theorem}\label{thm:pinningAC}
Assume $|u_i(0)| > 0$ for all $i\in V$. There exist an $\e_\rho$ and an $\e_\kappa$ (depending on the spectral radius of $\Delta$ via \eqref{eq:epsrho}, and on $\sup_{t\geq 0} \|\Delta u(t)\|_{\mathcal{V},\infty} < \infty$ via \eqref{eq:epskappa}, respectively), such that, if either $\e\leq\e_\rho$ or $\e \leq \e_\kappa$, then the solution $u(t)$ to \eqref{eq:ACEe} is such that $\text{sign}(u_i(t))$ is constant in time, for all $i \in V$.                 	
    \end{theorem}

    \begin{proof}
         	By Lemma~\ref{lem:positiveinvariant},  $\|u(t)\|_{\mathcal V}^2 \leq \frac{17}4 n \, d_+^{-r}$ for $t$ large enough. Hence, by continuity of $u(t)$, there is a $C$ (depending on the initial condition) such that, for all $t\geq 0$,
         \begin{equation*}
              \|u(t)\|_{\mathcal{V}} \leq C.
         \end{equation*}
         Thus, if $\rho>0$ denotes the spectral radius of $\Delta$, we get, for all $i\in V$,
         \begin{equation*}
              d_i^r |\Delta u_i(t)|^2 \leq \|\Delta u(t)\|_{\mathcal{V}}^2 \leq \rho^2 C^2.
         \end{equation*}
         In particular
         \begin{equation}\label{eq:Deltauiineq}
         |\Delta u_i(t)|\leq \rho C d_i^{-\frac{r}2},
         \end{equation}
         for all $i \in V$, thus, we have the inequalities
         \begin{align*}
               -\rho C d_i^{-\frac{r}2} -\tfrac{1}{\e} d_i^{-r} W'(u_i) \leq \dot u_i \leq \rho C d_i^{-\frac{r}2} - \tfrac{1}{\e} d_i^{-r} W'(u_i).
         \end{align*}
	Without loss of generality, we can assume that there is a number $\alpha \in (0,1)$ such that $|u_i(0)|\geq \alpha$ for all $i \in V$. If there is an $i\in V$ such that $|u_i(t)|=\alpha$ for a given $t$, then we have that $|W'(u_i(t))| = 4\alpha(1-\alpha^2) $, with a sign opposite to that of $u_i$. Thus, if
         \begin{equation}\label{eq:epsrho}
              \e \leq \e_0 := C^{-1}\rho^{-1}4\alpha(1-\alpha^2) d_+^{-\frac{r}2} \leq C^{-1}\rho^{-1}4\alpha(1-\alpha^2) d_i^{-\frac{r}2},
         \end{equation}
         then $\dot u_i \leq 0$ if $u_i(0) <  0$, and $\dot u_i \geq 0$ if $u_i(0) > 0$. Hence $u_i(t)$ can never reach zero, and by continuity in $t$ it does not change sign.

         Alternatively, instead of \eqref{eq:Deltauiineq}, we can estimate
         \[
         |\Delta u_i(t)| \leq \sup_{t\geq 0} \|\Delta u(t)\|_{\mathcal{V},\infty} < \infty.
         \]
         The finitude of the right hand side follows from \eqref{eq:Deltauiineq}. Following the same reasoning as above, we then conclude
         \begin{equation}\label{eq:epskappa}
         \e_\kappa := \left(\sup_{t\geq 0} \|\Delta u(t)\|_{\mathcal{V},\infty}\right)^{-1} d_+^{-r} 4\alpha (1-\alpha)^2.
         \end{equation}
    \end{proof}

The constant $\e_\kappa$ in Theorem~\ref{thm:pinningAC} involves $\| \Delta u \|_{\mathcal V,\infty}$, which is ``curvature-like''. Compare this to the constant $\tau_\kappa$ in Theorem~\ref{thm:pinningMBO}, which depends on the maximum curvature of the indicator set in the graph, $\| \Delta \chi_S \|_{\mathcal V,\infty}$, as also discussed in Section~\ref {sec:connections}. This tentative similarity makes us suspect, that a condition on the local curvature, similar to those for the \ref{alg:MBO} algorithm given in Theorem~\ref{thm:localtau}, guarantees a phase change in the Allen-Cahn flow. We discuss this further in Question~\ref{conj:ACcurvature}.

    We see that the discrete nature of the graph, manifest in the finite spectral radius of the Laplacian,  makes the limit behavior of \eqref{eq:ACEe} as $\e \to 0$ much different than that for the continuum case. In particular, this means that we ought to look for a notion of mean curvature flow on graphs more carefully.

\begin{remark}
   For $\e$ small enough, but not smaller than the $\e_0$ from Theorem~\ref{thm:pinningAC} above, we expect interesting asymptotic behavior for the motion of the phases in \eqref{eq:ACEe} on intermediate time scales. Such asymptotics might be connected to the graph curvature of the phases,
  which would match the situation in the continuum setting, where the solution has phases that for large times behave as if they were evolving by mean curvature flow, while the solution itself becomes stationary in the limit $t\to+\infty$. This phenomenon is known as dynamic metastability (see for instance \cite{BronsardKohn91} and the references therein). See also Question~\ref{ques:ACMCF}.
\end{remark}

\section{Explicit and computational examples}\label{sec:examples}
In this section we give several examples of graphs where the mean curvature, MBO, and Allen Cahn evolutions can be compared either explicitly or computationally.

\subsection{Complete graph}\label{sec:complete}
Consider the complete graph, $K_n$, on n nodes with $\omega_{ij}=\omega$ for all $i,j \in V$. See Figure~\ref{fig:stara}. In this case, the matrix representation of the graph Laplacian is given by the circulant matrix,
$$
L= \omega [(n-1)\omega]^{-r} \left( n \ \text{Id}_n - 1_n1_n^t \right),
$$
where $1_n$ denotes the vector in $\mathbb R^n$ with all entries equal to 1, and $\text{Id}_n$ the identity matrix in $\R^{n\times n}$.
The eigenvalues of $L$ are  given by 0 and $\omega n [(n-1)\omega]^{-r}$ (with multiplicity $n-1$).
In particular, $\lambda_2 = \lambda_n = \rho$. Note that the normalized eigenvector corresponding to eigenvalue 0 is given by $(\mathrm{vol}\ V)^{-\frac12} \chi_V$.

Let $S\subset V$ be a set with volume ratio $R_S = \frac{\mathrm{vol} \ S}{\mathrm{vol}\ V}$ (see also Theorem~\ref{cor:trivialDynamics}). Using the spectral decomposition from \eqref{eq:specDecomp}, the evolution of $\chi_S$  by the heat equation can be explicitly written as
$$
e^{-t \Delta} \chi_S = R_S \chi_V + e^{- \rho t} \left( \chi_S - R_S \chi_V \right).
$$
Assume $R_S \neq \frac{1}{2}$. Then there exists a critical time step $\tau_c$ depending only on $\mathrm{vol} \ S$, $\mathrm{vol} \ V$, and $\rho$ such that $\tau<\tau_c$ implies the solution to the \ref{alg:MBO} evolution is pinned and $\tau\geq\tau_c$ implies exactly one iteration of the \ref{alg:MBO} evolution gives a stationary solution, either $0$ or $\chi_V$ depending on the initial mass, $M(\chi_S)=\mathrm{vol} \ S$ (see \eqref{eq:mass}). From the solution, the critical time step $\tau_c$ can be directly computed,
$$
\tau_c = \frac{1}{\rho} \log \frac{ \max\{Rs, 1-R_S\} }{| \frac{1}{2} - R_S |}
$$
If $R_S=\frac{1}{2}$, symmetry pins the \ref{alg:MBO} evolution for all $\tau > 0$.

The bound from Theorem~\ref{thm:pinningMBO} states that pinning occurs if $\tau< \tau_\rho= \rho^{-1} \log \left( 1 + \frac{1}{2} (n R_S)^{-\frac{1}{2}}\right)  $, where we have used the fact that, for all $i\in V$, $d_i^r = \frac{\mathrm{vol} V}{n}$.
The bound in Theorem~\ref{cor:trivialDynamics} states that trivial dynamics occur if
$\tau > \tau_t = \rho^{-1}  \log \left(  \frac{n^{\frac{1}{2}} R_S^{\frac{1}{2}} R_{S^c}^{\frac{1}{2}} }{ | \frac{1}{2} - R_S |} \right)$.
Note that for $n>2$,  $R_S > \frac{1}{n}$ and $\tau_t > \tau_c > \tau_\rho$.

By symmetry, both the Allen-Cahn equations and mean curvature flows reduce to two-dimensional systems, with one variable governing the value of the nodes in $S$ and the other the nodes in $S^c$. Critical parameters $\epsilon$ and $\eth t$ exist for which below the phase remains the same for all nodes and above the phase simultaneously changes.

\subsection{Star graph} \label{sec:star}

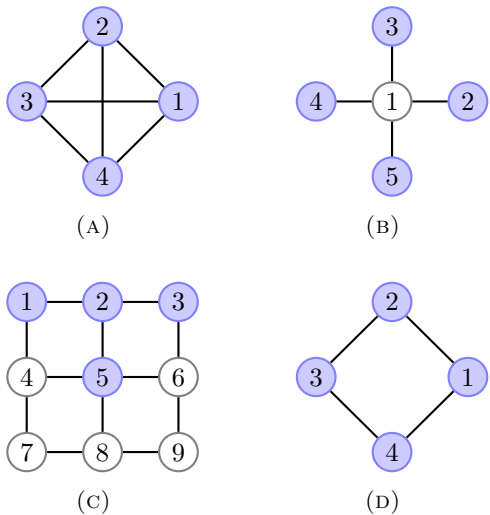
\begin{figure}
\begin{center}
\begin{subfigure}[b]{0.22\textwidth}
\centering
\begin{tikzpicture} 
[inner sep=.8mm,
dot/.style={circle,draw=blue!50,fill=blue!20,thick}] 
\node (1) at (1,0)[dot]{1};
\node (2) at (0,1)[dot] {2}
	edge[-,thick] node[auto,swap] {} (1);
\node (3) at (-1,0)[dot]{3}
	edge[-,thick] node[above left,swap] {} (1)
	edge[-,thick] node[above left,swap] {} (2);	
\node (4) at (0,-1)[dot]{4}
	edge[-,thick] node[above left,swap] {} (1)
	edge[-,thick] node[above left,swap] {} (2)
	edge[-,thick] node[above left,swap] {} (3);
\end{tikzpicture}
\caption{} \label{fig:stara}
\end{subfigure}
\qquad \qquad
\begin{subfigure}[b]{0.22\textwidth}
\centering
\begin{tikzpicture} 
[inner sep=.8mm,
dot/.style={circle,draw=blue!50,fill=blue!20,thick},
redDot/.style={circle,draw=black!50,thick}] 
\node (1) at (0,0)[redDot] {1};
\node (2) at (1,0)[dot]{2}
	edge[-,thick] node[auto,swap] {} (1);
\node (3) at (0,1)[dot] {3}
	edge[-,thick] node[auto,swap] {} (1);
\node (4) at (-1,0)[dot]{4}
	edge[-,thick] node[above left,swap] {} (1);	
\node (5) at (0,-1)[dot]{5}
	edge[-,thick] node[above left,swap] {} (1);
\end{tikzpicture}
\caption{} \label{fig:starb}
\end{subfigure} \\ \vspace{.5cm}
\begin{subfigure}[b]{0.22\textwidth}
\centering
\begin{tikzpicture} 
[inner sep=.8mm,
dot/.style={circle,draw=blue!50,fill=blue!20,thick},
redDot/.style={circle,draw=black!50,thick}] 
\node (1) at (-1,2)[dot]{1};
\node (2) at (0,2)[dot] {2}
	edge[-,thick] node[auto,swap] {} (1);
\node (3) at (1,2)[dot]{3}
	edge[-,thick] node[above left,swap] {} (2);	
\node (4) at (-1,1)[redDot]{4}
	edge[-,thick] node[above left,swap] {} (1);
\node (5) at (0,1)[dot]{5}
	edge[-,thick] node[above left,swap] {} (2)
	edge[-,thick] node[above left,swap] {} (4);
\node (6) at (1,1)[redDot]{6}
	edge[-,thick] node[above left,swap] {} (3)
	edge[-,thick] node[above left,swap] {} (5);
\node (7) at (-1,0)[redDot]{7}
	edge[-,thick] node[above left,swap] {} (4);
\node (8) at (0,0)[redDot]{8}
	edge[-,thick] node[above left,swap] {} (5)
	edge[-,thick] node[above left,swap] {} (7);
\node (9) at (1,0)[redDot]{9}
	edge[-,thick] node[above left,swap] {} (6)
	edge[-,thick] node[above left,swap] {} (8);
\end{tikzpicture}
\caption{} \label{fig:starc}
\end{subfigure}
\qquad \qquad
\begin{subfigure}[b]{0.22\textwidth}
\centering
\begin{tikzpicture} 
[inner sep=.8mm,
dot/.style={circle,draw=blue!50,fill=blue!20,thick}] 
\node (1) at (1,0)[dot]{1};
\node (2) at (0,1)[dot] {2}
	edge[-,thick] node[auto,swap] {} (1);
\node (3) at (-1,0)[dot]{3}
	edge[-,thick] node[above left,swap] {} (2);	
\node (4) at (0,-1)[dot]{4}
	edge[-,thick] node[above left,swap] {} (1)
	edge[-,thick] node[above left,swap] {} (3);
\end{tikzpicture}
\caption{} \label{fig:stard}
\end{subfigure}
\end{center}
\caption{ Some small graphs, discussed in the examples of Section~\ref{sec:examples}.
{\bf (a)} The complete graph $K_4$,
{\bf (b)} the star graph $SG_5$,
{\bf (c)} a small grid,
and
{\bf (d)} a cycle graph $C_4$; see Sections~\ref{sec:complete}, \ref{sec:star}, \ref{sec:squaregrid}, and~\ref{sec:tori}. }
\label{fig:star}
\end{figure}

Consider a star graph $SG_n$ as in Figure~\ref{fig:starb} with $n\geq3$ nodes. Here the central node (say node 1) is connected to all other nodes and the other $n-1$ nodes are only connected to the central node. Hence, for all $i \in \{2, \ldots, n\}$, $\omega_{1i} = \omega_{i1}> 0$, and all the other $\omega_{jk}$ are zero.

We consider the unnormalized graph Laplacian $L=D-A$ ($r=0$ in \eqref{eq:eigDef4}). Since $d_1 = \sum_{j=2}^n \omega_{1j}$ and $d_i = \omega_{1i}$, for $i\in \{2, \ldots, n\}$, we can explicitly compute the characteristic polynomial of $L$:
\[
p(\lambda) = \Big( -\lambda + \sum_{k=2}^n \omega_{1k} \Big) \prod_{j=2}^n (\omega_{1j} - \lambda) - \sum_{k=2}^n \omega_{1k}^2 \prod_{j\geq 2,\,\,j\neq k} (\omega_{1j}-\lambda).
\]
If all non-zero edge weights have the same value $\omega$, this simplifies considerably to
\[
p(\lambda) = \big((n-1) \omega - \lambda \big) (\omega-\lambda)^{n-1} - \omega^2 (n-1) (\omega-\lambda)^{n-2}.
\]
Hence, in this case, the eigenvalues are $\lambda_1 = 0$, $\lambda_i = \omega$ for $i \in \{2, \ldots, n-1\}$, and $\lambda_n = n\omega$. A choice of corresponding (normalized) eigenvectors $\{v^i\}_{i=1}^n$ is given by\footnote{Here, subscripts $j$ denote the components of the vectors.}
\begin{align*}
&v^1= n^{-\frac{1}{2}} \chi_V, \qquad
v^i_j  = 2^{-\frac{1}{2}}\begin{cases} 1 & \text{if } j=i,\\ -1 &\text{if } j=i+1,\\ 0 & \text{else},\end{cases} \text{ for } i \in \{2, \ldots, n-1\}\\
&v^n_j = \frac{1}{\sqrt{n(n-1)}}\begin{cases} n-1 & \text{if } j=1,\\ -1 & \text{if } j\neq 1. \end{cases}
\end{align*}

We now let $S = \{ 1\}$ and note that $\chi_S$ has the explicit expansion in terms of these eigenvectors,
$$
\chi_S = n^{-\frac{1}{2}} v^1 + (n-1)^{\frac{1}{2}} n^{-\frac{1}{2}} v^n.
$$
We now consider the \ref{alg:MBO} iterates of $\chi_S$. We compute
$$
e^{-\Delta \tau}  \chi_S = n^{-1} \chi_V  + (n-1)^{\frac{1}{2}} n^{-\frac{1}{2}} (e^{-n\omega \tau} ) v^n.
$$
Thus pinning occurs if $\tau< \tau_c := \frac{1}{n\omega} \log \left( 2 \frac{ n-1}{n-2} \right)$. If $\tau>\tau_c$, the solution to the \ref{alg:MBO} evolution gives the stationary solution, $0$, after exactly one iteration.
The bound from Theorem~\ref{thm:pinningMBO} states that pinning occurs if $\tau< \frac{1}{n\omega} \log \frac{3}{2} $.
The bound in Theorem~\ref{cor:trivialDynamics} states that trivial dynamics occur if
$\tau >  \frac{1}{\omega} \log \left( 2 \frac{ n-1}{n-2} \right) $.
Qualitatively, this example shows that it is easier for a solution to be pinned on nodes with smaller degree.

\medskip

We now consider an implication of Theorem~\ref{thm:localtau} in the case where the graph induced by $\overline{S_1}$ is a star graph with node 1 as center, {\it i.e.}, $d'_i = 0$ for all $i\in S_1$. This is certainly true for the case when the graph is a star graph and $S = \{2, \ldots, n\}$. The following lemma states, in the case where $r=1$, a simple criterion on the degrees for which there exists  a $\tau$ interval in which node 1 switches phase in a single iteration of \ref{alg:MBO}. We will see an application of the lemma in Section~\ref{sec:tree}.

\begin{lemma} \label{lemma:star}
Let $r=1$ and consider the case where the subgraph induced by $\overline{S_1}$ is a star graph with node 1 as center, {\it i.e.}, $d'_i = 0$ for all $i\in S_1$. If either
\begin{itemize}
\item $|S_1| = 1$ and $d_i < d_1$, for $i\in S_1$, or
\item $|S_1| \geq 2$ and $d_i \leq d_1$, for all $i\in S_1$,
\end{itemize}
then there exists a nontrivial $\tau$ interval (as in Theorem~\ref{thm:localtau}) such that node $1$ will change phase in the next \ref{alg:MBO} iterate.
\end{lemma}
\begin{proof}
We see from condition \eqref{eq:condr=1} (or via direct computation from \eqref{eq:gapcondition}) that a sufficient condition for $\tau_1 < \tau_2$, is to have, for all $i\in S_1$, $\displaystyle \frac{\omega_{i1}}{d_i} < \frac{d'_1}{d_1}$ or equivalently $\displaystyle \frac{\omega_{i1}}{d'_1} < \frac{d_1}{d_i}$. If $i\in S_1$ is the only node in $S_1$, then $\omega_{i1} = d'_1$ and we find the condition $d_i < d_1$. If however $|S_1| \geq 2$, we have $\displaystyle \frac{\omega_{i1}}{d'_1} \leq 1$, because $\displaystyle d'_1 = \omega_{i1} + \sum_{j\in S_1\setminus\{i\}} \omega_{ij}$, and thus the condition on $\frac{d_1}{d_i}$ can be replaced by the simpler (but stronger) condition $d_i < d_1$, for all $i\in S_1$.
\end{proof}

\subsection{A regular tree } \label{sec:tree}
We consider the \ref{alg:MBO} iterations on a regular tree as in Figure~\ref{fig:tree}. Let $\omega_{ij}=\omega$, for all $(i,j)\in E$,  and $r=1$. As in Figure~\ref{fig:treea}, we consider the case where the initial set $S$ consists of the leaves of a branch. We first observe that the subgraph induced by $\overline{S_j}$, for any $j\in V$, is a star graph with node $j$ as center,  {\it i.e.}, $d'_i = 0$ for all $i\in S_j$ (for an example of a star graph with five nodes, see Figure~\ref{fig:starb}), so that the hypothesis of Lemma~\ref{lemma:star} are satisfied with nodes 9 and 10 each playing the role of ``node 1'' in the lemma. 

Applying Lemma~\ref{lemma:star} to node 9 in Figure~\ref{fig:treea} where $S=\{ 1,2,3,4\}$, we see that there exists a $\tau$ such that node 9 will change in the next iteration. By symmetry, node 10 will change in the same iteration. If node 13 doesn't change in the first MBO iteration, Lemma~\ref{lemma:star} can be applied again (because node 13 has two children, the hypotheses of the lemma are again satisfied with $9, 10 \in S_{13}$) to show that there exists a $\tau$ such that node 13 will be added to the set. 
After node 13 has been added to the set, $S$, as in Figure~\ref{fig:treeb} the MBO iterates are stationary. To see that node 15 cannot be added to $S$, assume that it were.  Then the value of the Lyapunov functional, \eqref{eq:Lyapunov}, must have decreased. But by symmetry, in the next MBO iteration, node 15 will be removed from $S$, again decreasing the value of the Lyapunov functional, a contradiction. The final configuration in Figure~\ref{fig:treeb} minimizes the normalized cut, as defined in Section~\ref{sec:BalancedGraphCut}.

This argument is easily generalized to trees where each node, excluding leaves, has the same number of children $c\geq 2$.

\begin{figure}
\begin{center}
\begin{subfigure}[b]{.85\textwidth}
\begin{tikzpicture} 
[inner sep=.8mm,
dot/.style={circle,draw=blue!50,fill=blue!20,thick},
redDot/.style={circle,draw=black!50,thick},
scale=0.6] 
\node (1) at (-7,1)[redDot]{1};
\node (2) at (-5,1)[redDot] {2};
\node (3) at (-3,1)[redDot] {3};
\node (4) at (-1,1)[redDot] {4};
\node (5) at (1,1)[dot] {5};
\node (6) at (3,1)[dot] {6};
\node (7) at (5,1)[dot] {7};
\node (8) at (7,1)[dot] {8};
\node (9) at (-6,2)[dot] {9}
	edge[-,thick] node[above left,swap] {} (1)
	edge[-,thick] node[above left,swap] {} (2);	
\node (10) at (-2,2)[dot] {10}
	edge[-,thick] node[above left,swap] {} (3)
	edge[-,thick] node[above left,swap] {} (4);	
\node (11) at (2,2)[dot] {11}
	edge[-,thick] node[above left,swap] {} (5)
	edge[-,thick] node[above left,swap] {} (6);	
\node (12) at (6,2)[dot] {12}
	edge[-,thick] node[above left,swap] {} (7)
	edge[-,thick] node[above left,swap] {} (8);	
\node (13) at (-4,3)[dot] {13}
	edge[-,thick] node[above left,swap] {} (9)
	edge[-,thick] node[above left,swap] {} (10);	
\node (14) at (4,3)[dot] {14}
	edge[-,thick] node[above left,swap] {} (11)
	edge[-,thick] node[above left,swap] {} (12);	
\node (15) at (0,4)[dot] {15}
	edge[-,thick] node[above left,swap] {} (13)
	edge[-,thick] node[above left,swap] {} (14);		
\end{tikzpicture}
\caption{Initial configuration} \label{fig:treea}
\end{subfigure} \\
\vspace{1cm}
\begin{subfigure}[b]{.85\textwidth}
\begin{tikzpicture} 
[inner sep=.8mm,
dot/.style={circle,draw=blue!50,fill=blue!20,thick},
redDot/.style={circle,draw=black!50,thick},
scale=0.6] 
\node (1) at (-7,1)[redDot]{1};
\node (2) at (-5,1)[redDot] {2};
\node (3) at (-3,1)[redDot] {3};
\node (4) at (-1,1)[redDot] {4};
\node (5) at (1,1)[dot] {5};
\node (6) at (3,1)[dot] {6};
\node (7) at (5,1)[dot] {7};
\node (8) at (7,1)[dot] {8};
\node (9) at (-6,2)[redDot] {9}
	edge[-,thick] node[above left,swap] {} (1)
	edge[-,thick] node[above left,swap] {} (2);	
\node (10) at (-2,2)[redDot] {10}
	edge[-,thick] node[above left,swap] {} (3)
	edge[-,thick] node[above left,swap] {} (4);	
\node (11) at (2,2)[dot] {11}
	edge[-,thick] node[above left,swap] {} (5)
	edge[-,thick] node[above left,swap] {} (6);	
\node (12) at (6,2)[dot] {12}
	edge[-,thick] node[above left,swap] {} (7)
	edge[-,thick] node[above left,swap] {} (8);	
\node (13) at (-4,3)[redDot] {13}
	edge[-,thick] node[above left,swap] {} (9)
	edge[-,thick] node[above left,swap] {} (10);	
\node (14) at (4,3)[dot] {14}
	edge[-,thick] node[above left,swap] {} (11)
	edge[-,thick] node[above left,swap] {} (12);	
\node (15) at (0,4)[dot] {15}
	edge[-,thick] node[above left,swap] {} (13)
	edge[-,thick] node[above left,swap] {} (14);		
\end{tikzpicture}
\caption{Final configuration} \label{fig:treeb}
\end{subfigure}
\end{center}
\caption{The initial and final configurations for an evolution by the \ref{alg:MBO} scheme on a tree graph; see Section~\ref{sec:tree}.}
\label{fig:tree}
\end{figure}
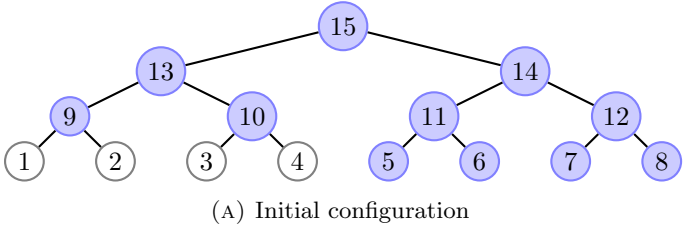
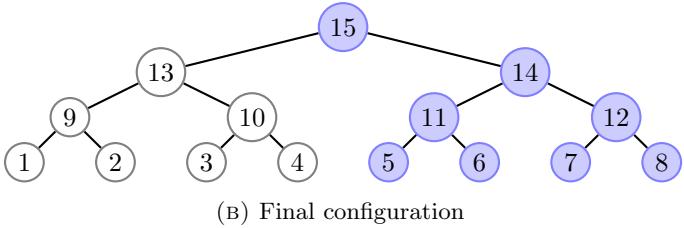

\subsection{A small square grid} \label{sec:squaregrid}
Here we construct an explicit example where Theorem~\ref{thm:localtau} can be applied to show that there exists a time  interval $(\tau_1, \tau_2)$ such that a node is guaranteed to change in one iteration of the \ref{alg:MBO} algorithm.

Consider a three by three (nonperiodic) square grid as in  Figure~\ref{fig:starc} with unit edge weights, with nodes numbered 1 through 9 from left to right, top to bottom. Let $S=\{4, 6, 7, 8, 9\}$.  We focus on node 5. We have $\mathcal{N}_5 = \{2, 4, 6, 8\}$ and $S_5 = \mathcal{N}_5 \cap S = \{4, 6, 8\}$. We then compute $d'_5=3$, $d_5=4$, $d'_4=d'_6=d'_8=0$ and $d_4=d_6=d_8=3$. It is easily checked that, for $i\in S_5$,
\[
\frac{d'_i}{d_i} = 0 \leq \frac34 = \frac{d'_5}{d_5} \quad \text{and} \quad \frac{\omega_{i5}}{d_i} = \frac13 < \frac34 = \frac{d'_5}{d_5},
\]
such that conditions \eqref{eq:condr=1} are satisfied. Furthermore, with $\overline{S_5} := S_5 \cup \{5\}$,
\[
\sum_{k\in \overline{S_5}} \frac{d'_k}{d_k} \frac{\omega_{ik}}{d_i} = \begin{cases} 0 & \text{if } i=5,\\ \frac14 & \text{if } i\in S_5,\end{cases}
\]
thus $\left(\frac{d'_5}{d_5}\right)^2 - \max_{i\in \overline{S_5}} \sum_{k\in \overline{S_5}} \frac{d'_k}{d_k} \frac{\omega_{ik}}{d_i}  = \left(\frac34\right)^2 - \frac14 = \frac5{16} > 0$, and even the full condition \eqref{eq:gapcondition}, for $r=1$, is satisfied. From \eqref{eq:tau12} we can then compute
\[
\tau_{1,2} = \frac{3/4}{1/4} \pm 4 \sqrt{\frac5{16}} = 3 \pm \sqrt5,
\]
for the time interval $(\tau_1, \tau_2)$ of Theorem~\ref{thm:localtau}.

\subsection{Torus graph} \label{sec:tori}

Consider the $n$-cycle, $C_n$ with $n$ nodes. The nodes are arranged in a circle and each node is connected to its 2 neighbors.  We take $\omega_{ij} = \omega$ for $i\sim j$ and zero otherwise.
See Figure~\ref{fig:stard}.
We consider the unnormalized graph Laplacian $L=D-A$ ($r=0$ in \eqref{eq:eigDef4}).
In this case, $L$ is a circulant matrix $\text{diag}(\{-1,2,-1\}, \{-1,0,1\})$. The eigenpairs $\{ (\lambda_j, v^j) \}_{j=1}^n$ are given by
\begin{align*}
\lambda_j &= 2 \omega - 2 \omega \cos \frac{2 \pi (j-1)}{n}   \\
v^j_i &= \exp\left( 2\pi i (j-1)/n \right).
\end{align*}

We then consider the $2$-torus graph, $T^2_{n_1,n_2}$ which is the Kronecker (tensor) product of the $n_1$- and $n_2$-cycles. See Figure~\ref{fig:tori}. In particular, if $u$ and $v$ are eigenfunctions of the graph Laplacian on $C_{n_1}$ and $C_{n_2}$ with corresponding eigenvalues $\alpha$ and $\beta$ respectively, then $w=u\otimes v$ (with $w_{i,j} = u_i v_j$) is an eigenvector of $T^2_{n_1,n_2}$ with corresponding eigenvalue $\alpha + \beta$.
In particular,  the spectral radius of the Laplacian is  $\rho = 8 \omega$.

Consider for a moment $T^2_{n_1,n_2}$ as a discretization of the torus, $\mathbb T^2$.
The nontrivial minimal-perimeter subsets of $\mathbb T^2$ are given by ``strips''.
Thus we might expect that for some initial condition, $\chi_S$, $S\subset V$ the evolution by MBO, Allen-Cahn, or MC would converge to a strip.

We consider the \ref{alg:MBO} evolution on a $32 \times 12$ torus with $\omega=1$ and initial condition, as in Figure~\ref{fig:toria}.
For $\tau=1.12$, the solution is stationary after 4 iterations once the ``high curvature corners'' have been removed, as in Figure~\ref{fig:torib}. For $\tau = 4$, the solution evolves into a minimal-perimeter ``strip'' in 5 iterations, as in Figure~\ref{fig:toric}.

For the parameters in Figure~\ref{fig:torib}, we compute the guaranteed stationarity bounds in
 \eqref{eq:taup} and  \eqref{eq:tauk} to be  $\tau_\rho\approx 0.0057$ and  $\tau_\kappa = \frac{1}{4}$, respectively, showing these bounds are not sharp.

Consider \eqref{eq:MCFh}, with $S_n$ equal to the minimal-perimeter strip in Figure \ref{fig:toric}. Then $S_{n+1}=S_n$ is a minimizer of $\mathcal{F}(\cdot, S_n)$, but so are $S_{n+1} = S_n \cup \partial (S_n^c)$ and $S_{n+1} = S_n \setminus \partial S_n$ (or variations in which only one `vertical line' of the boundary is added or removed). This illustrates a possible type of non-uniqueness for \eqref{eq:MCFh}, which occurs when $S_n$ is totally geodesic ({\it i.e.}, its boundary is a geodesic).
To reiterate, the stationary solution in Figure~\ref{fig:torib} is frozen (due to the smallness of $\tau$), while the solution in Figure \ref{fig:toric} is totally geodesic.

\begin{figure}[t!]
\begin{center}
\begin{subfigure}[b]{0.32\textwidth}
\includegraphics[width=1.1\textwidth]{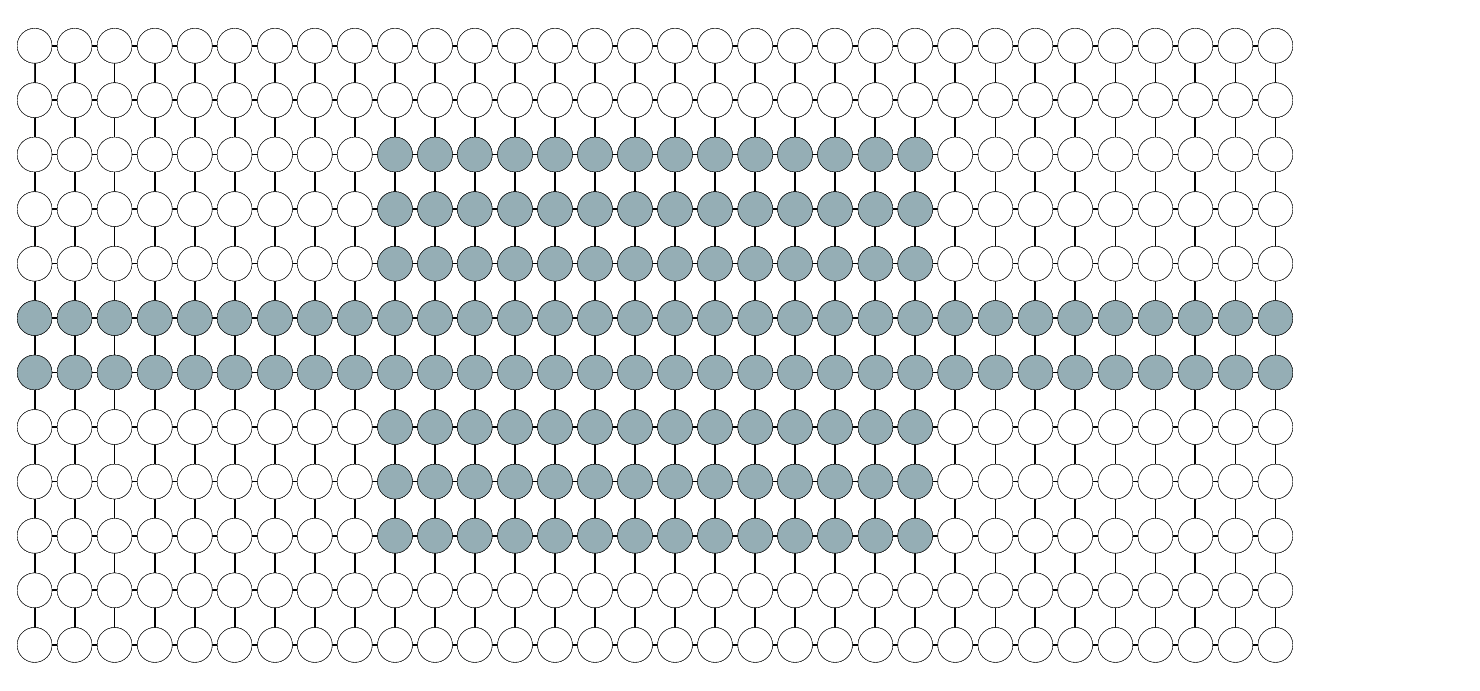}
\caption{} \label{fig:toria}
\end{subfigure}
\begin{subfigure}[b]{0.32\textwidth}
\hspace{.15cm}
\includegraphics[width=1.1\textwidth]{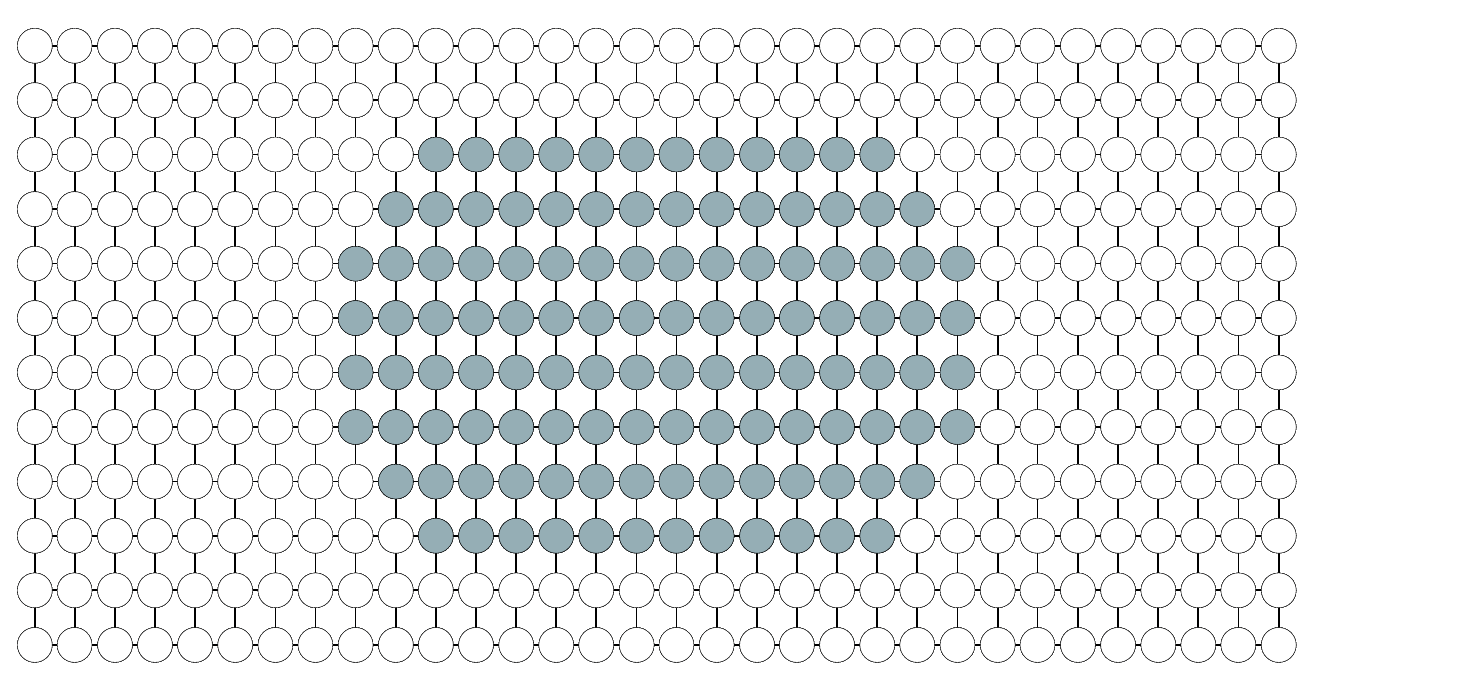}
\caption{} \label{fig:torib}
\end{subfigure}
\begin{subfigure}[b]{0.32\textwidth}
\hspace{.15cm}
\includegraphics[width=1.1\textwidth]{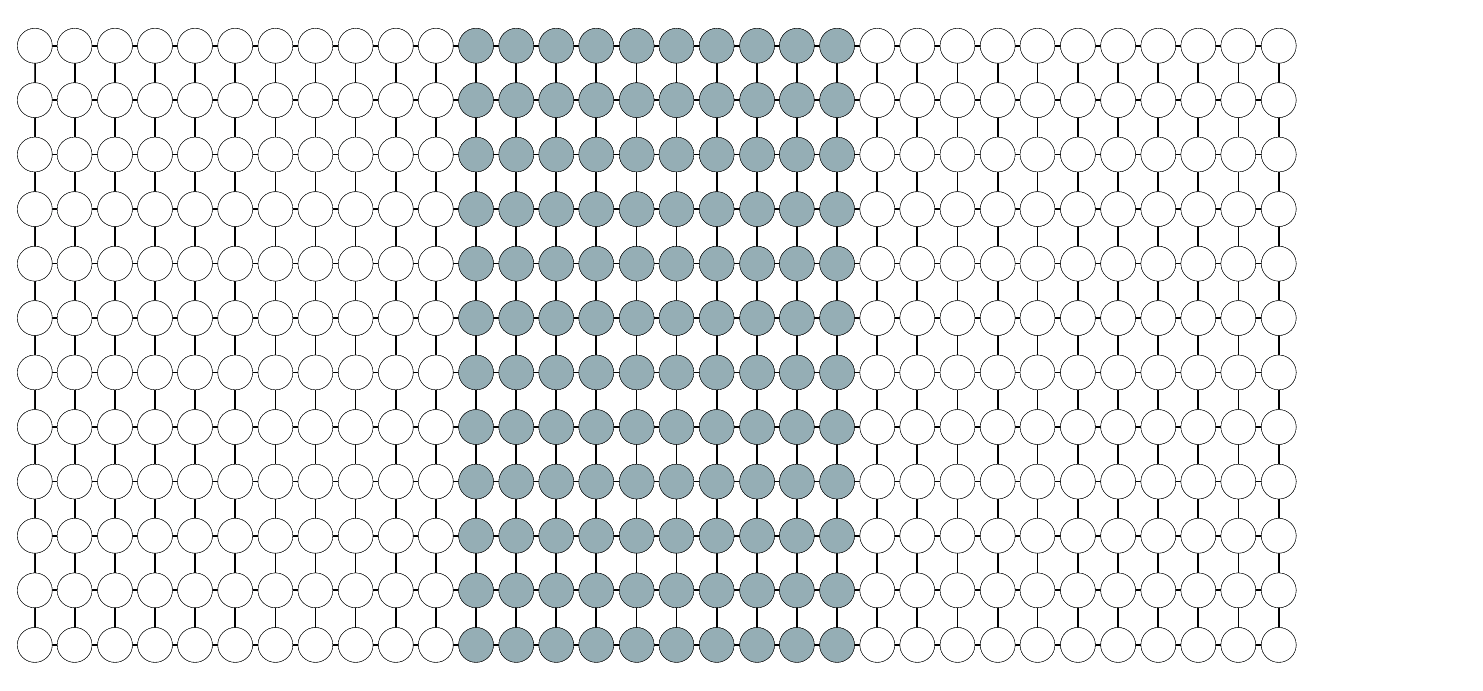}
\caption{} \label{fig:toric}
\end{subfigure}
\caption{Two \ref{alg:MBO} evolutions on the 2-torus graph, $T^2_{32,12}$. The top and bottom `border' nodes are connected (not shown) as are the left and right `border' nodes.  {\bf (a)} Initial condition. {\bf (b)} For $\tau=1.12$, the stationary state shown is reached in 4 iterations. {\bf (c)} For $\tau=4$, the stationary state shown is reached in 5 iterations.  See Section~\ref{sec:tori}. }
\label{fig:tori}
\end{center}
\end{figure}

\subsection{Buckyball graph}\label{sec:bucky}

\begin{figure}[t!]
\begin{center}
\includegraphics[width=.45\textwidth]{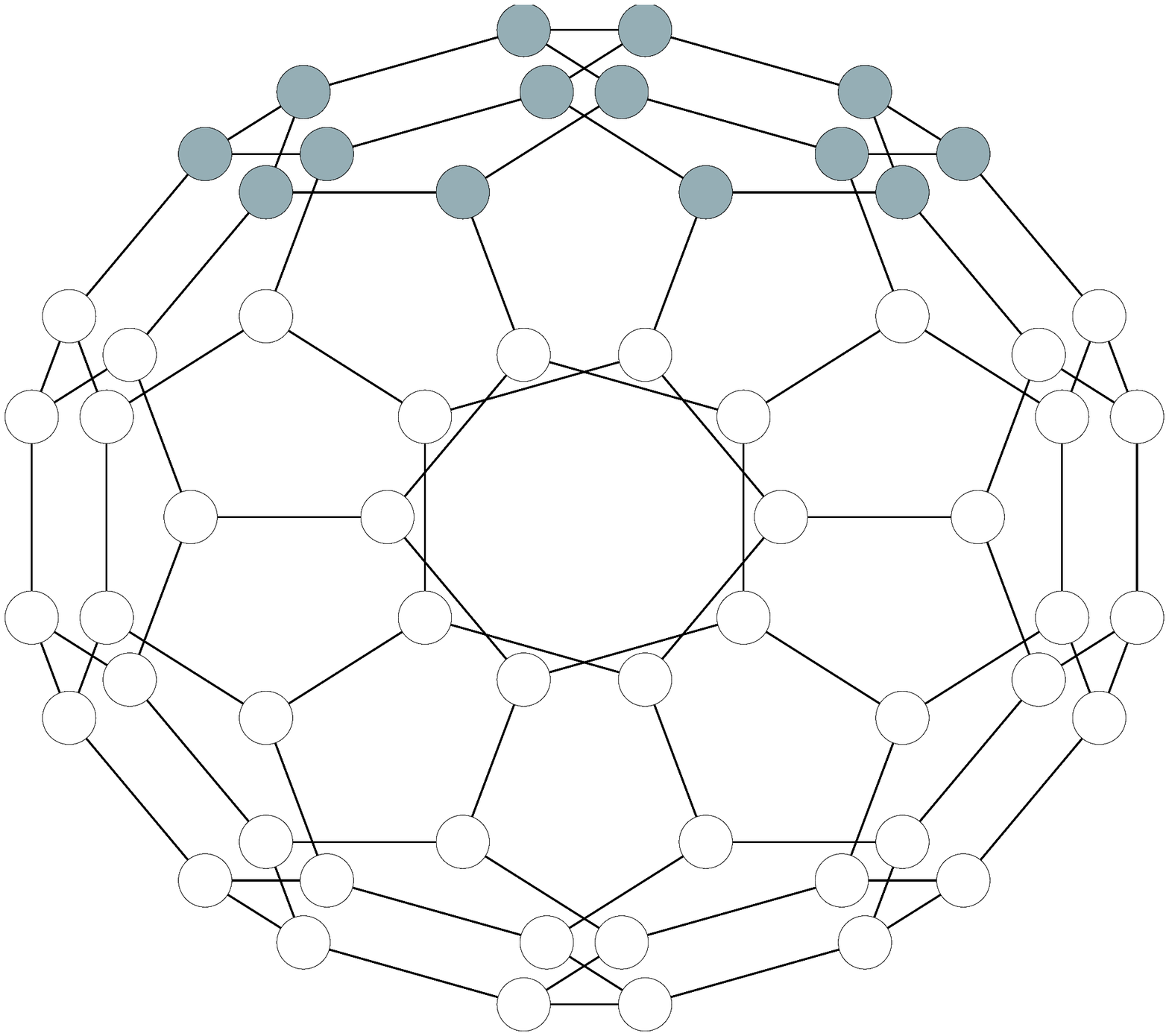}
\includegraphics[width=.45\textwidth]{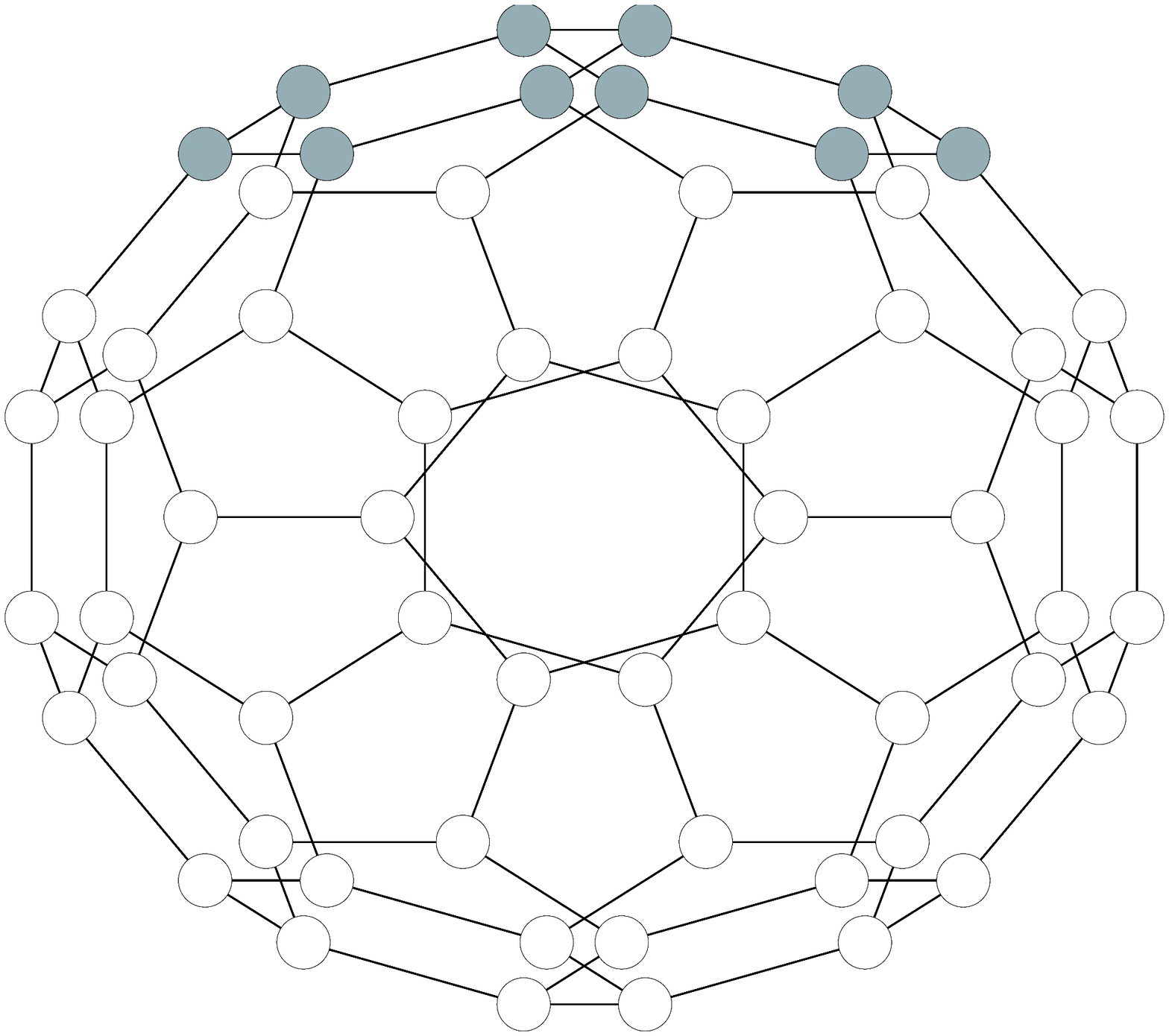}
\includegraphics[width=.45\textwidth]{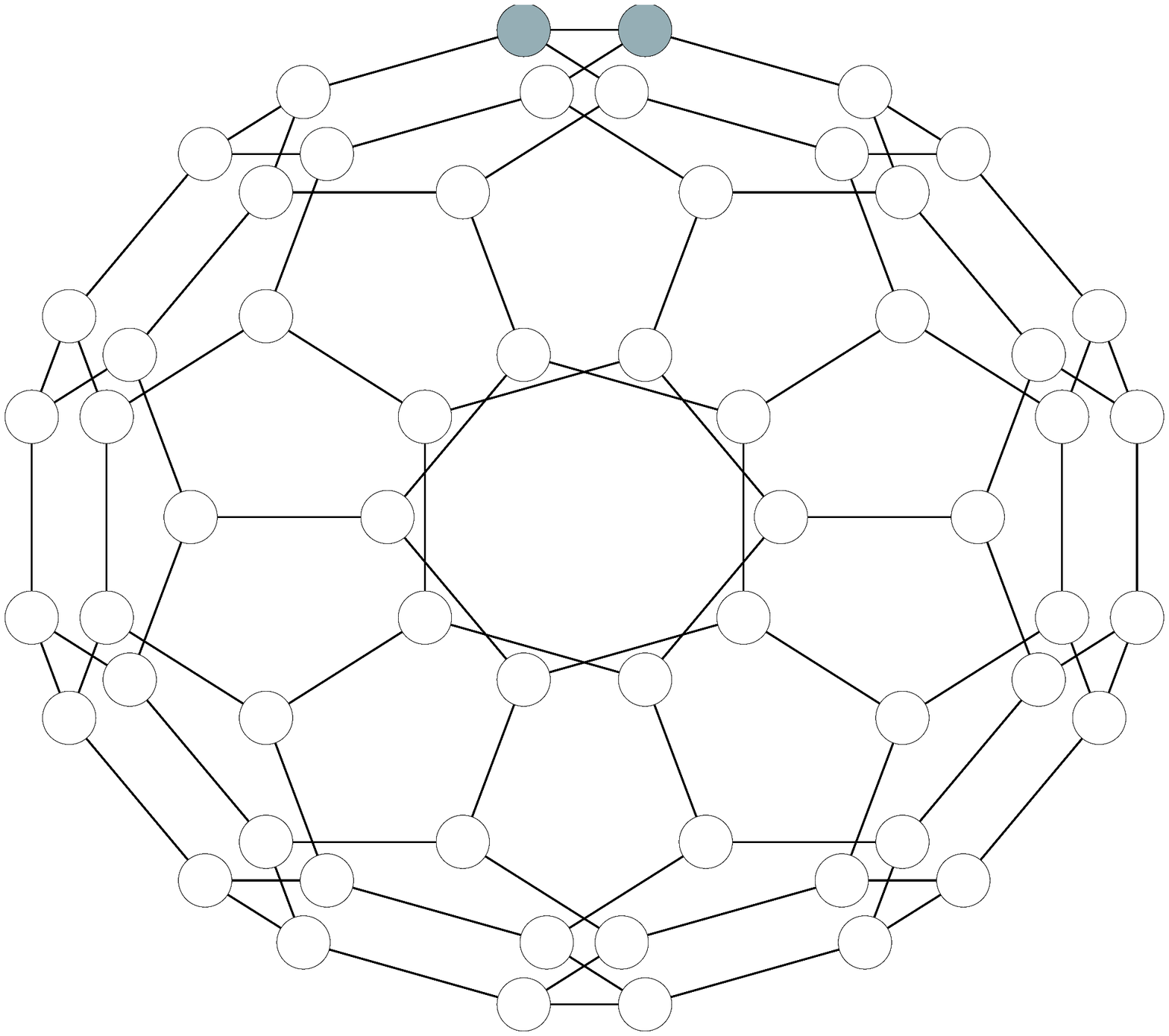}
\includegraphics[width=.45\textwidth]{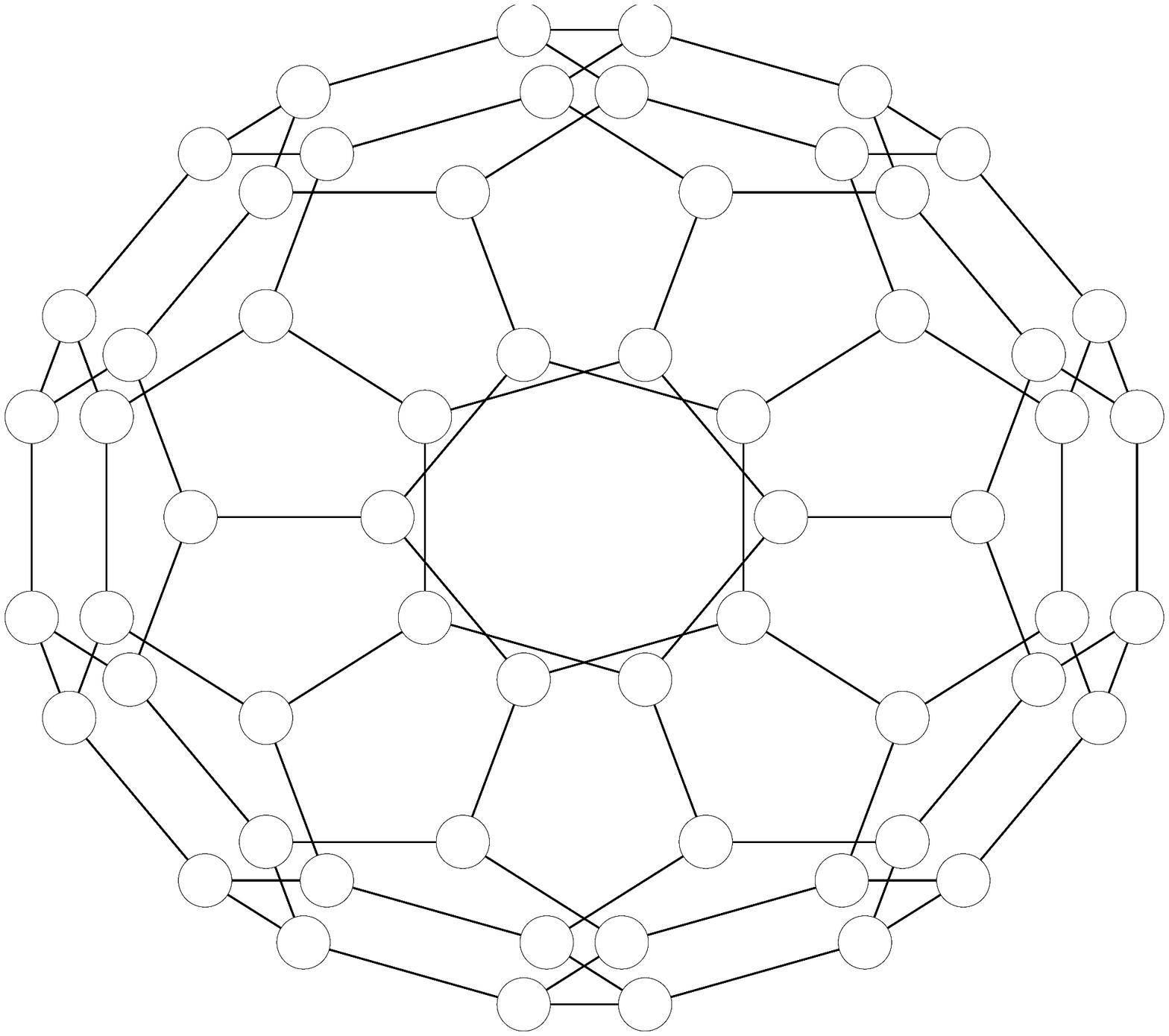}
\caption{An \ref{alg:MBO} evolution with $\tau=2$ on the buckyball graph. The solution at each iteration is the characteristic function of the gray nodeset. See Section~\ref{sec:bucky}. }
\label{fig:bucky}
\end{center}
\end{figure}

Consider the buckyball graph with 60 nodes and 90 edges with $\omega_{ij}=\omega$ for all edges $(i,j)$ as in Figure~\ref{fig:bucky}. The graph is regular; each node has degree $3\omega$.

Consider for a moment the buckyball graph as a (coarse) discretization of the sphere, $\mathbb S^2$. There are no nontrivial minimal-perimeter subsets of $\mathbb S^2$. Great circles are the only nontrivial stationary submanifolds of $\mathbb S^2$ (and have constant curvature). In fact, great circles are totally geodesic. Thus we might expect that for any initial condition, $\chi_S$, $S\subset V$ such that $|S|\neq |\mathbb S^2|/2$, the evolution by MBO, Allen-Cahn, or MC would converge to a stationary solution, either $0$ or $\chi_V$ depending on the initial mass, $M(\chi_S)=\mathrm{vol} \ S$. If $S$ is chosen to be a symmetric partitioning of the nodes for  the buckyball graph, we expect that the \ref{alg:MBO} evolution will be stationary for all values of $\tau$.

The bound from Theorem~\ref{thm:pinningMBO} states that pinning in \ref{alg:MBO} occurs if
$$\tau< \rho^{-1} \log \left( 1 + \frac{1}{2} |S|^{- \frac{1}{2}} \right).
$$
The bound in Theorem~\ref{cor:trivialDynamics} states that trivial dynamics occur if
$$
\tau >  \lambda_2^{-1} \log \left(  \frac{ (3 \omega)^{\frac{r}{2}} |S| (n-|S|) } { \left| |S| - \frac{n}{2}\right| } \right).
$$
We find numerically that $\lambda_2 \approx \omega^{1-r} 3^{-r} \cdot 0.2434$ and $\lambda_n \approx \omega^{1-r} 3^{-r} \cdot 5.6180$.

For initial condition $\chi_S$, with $|S| = 14$, as given in Figure~\ref{fig:bucky} (top left),  and $r=0$ and $\omega = 1$, Theorems~\ref{thm:pinningMBO} and~\ref{cor:trivialDynamics} predict that pinning occurs if $\tau < 0.0223$ and trivial dynamics occur if $\tau > 15.1811$. We find numerically that this initial condition is  pinned if $\tau< 1.89$ and trivial dynamics occur if $\tau > 3.54$.  For intermediate values of $\tau$, the iterates shrink to the empty node set. For $\tau=2$, the iterates take 3 iterations to reach steady state, as illustrated in Figure~\ref{fig:bucky}.

For the initial condition $\chi_S$ where $S$ is taken to be a symmetric partitioning of the nodes, \ref{alg:MBO} evolution is pinned  for all values of $\tau$.

\subsection{Adjoining regular lattices}\label{sec:adj}

\begin{figure}[t!]
\begin{center}
\includegraphics[width=.45\textwidth]{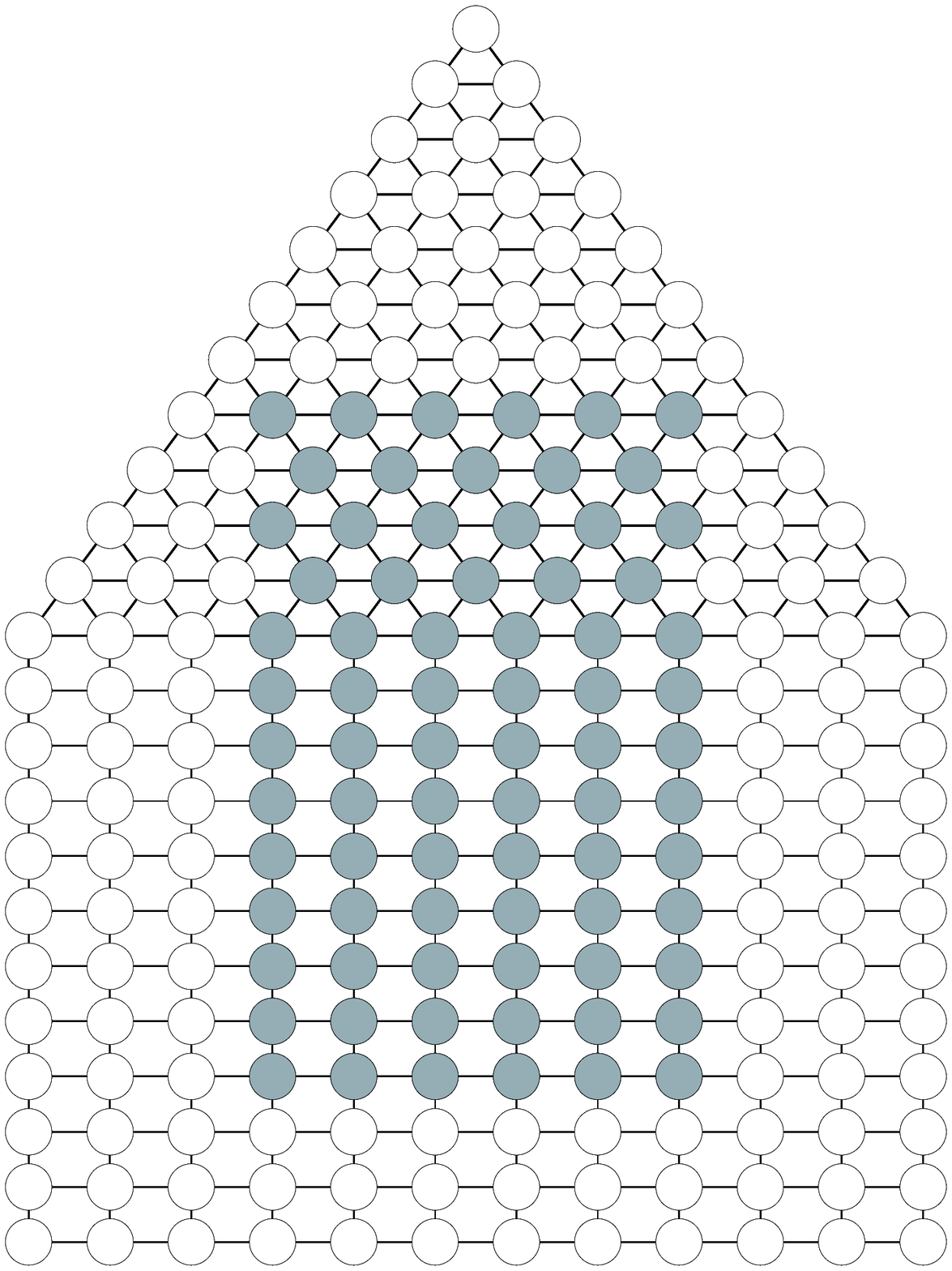} \qquad
\includegraphics[width=.45\textwidth]{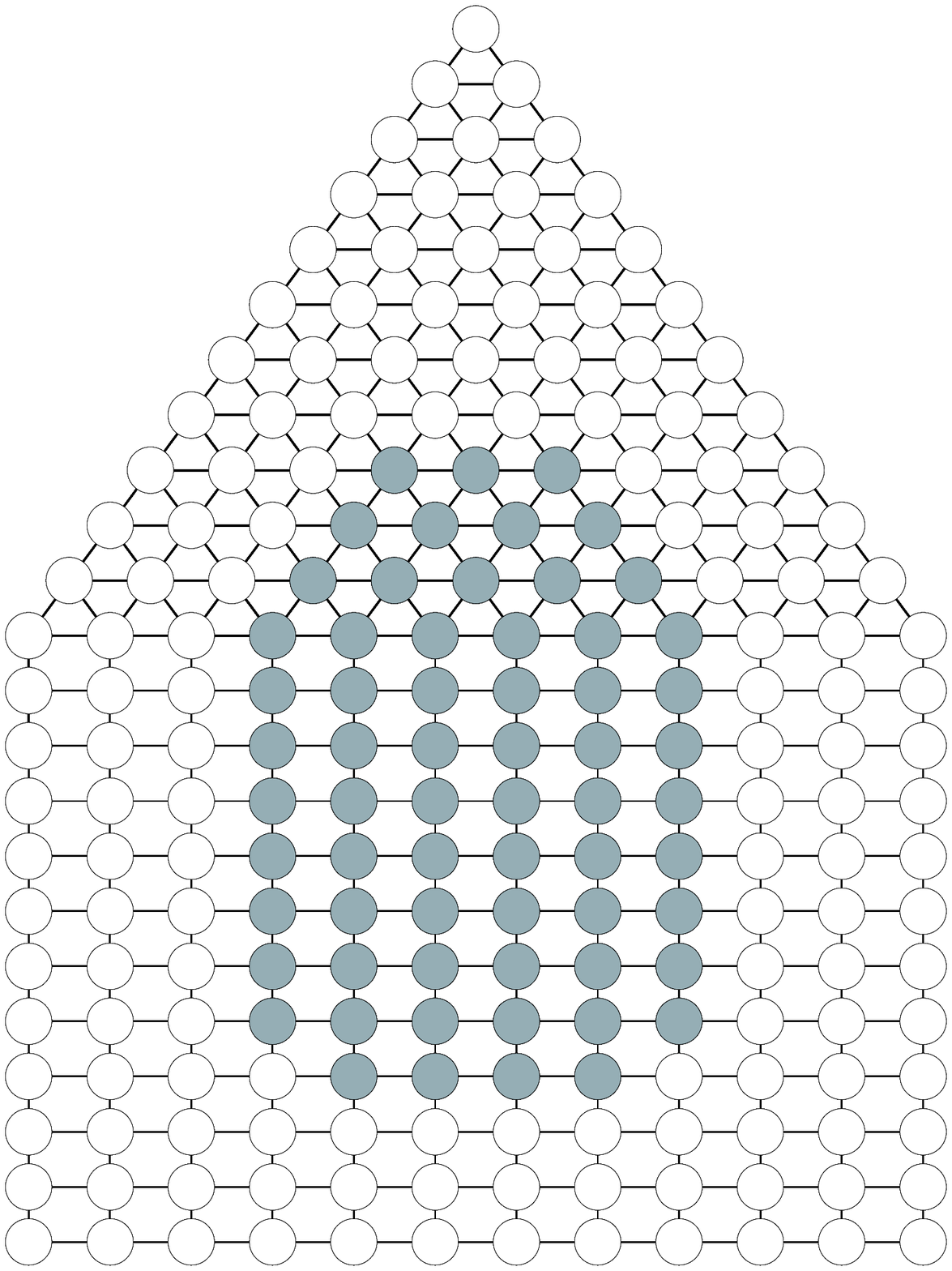}
\includegraphics[width=.45\textwidth]{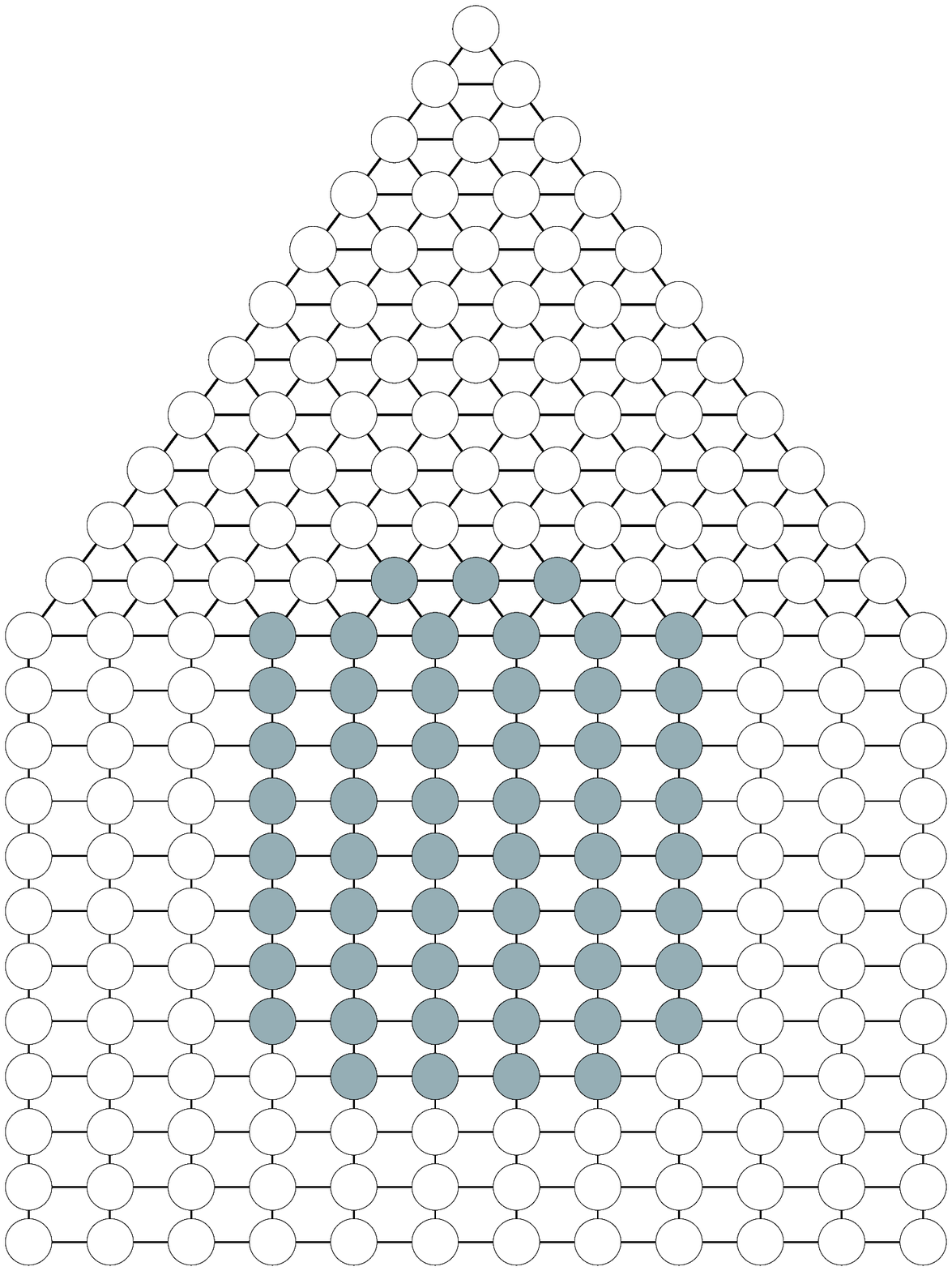} \qquad
\includegraphics[width=.45\textwidth]{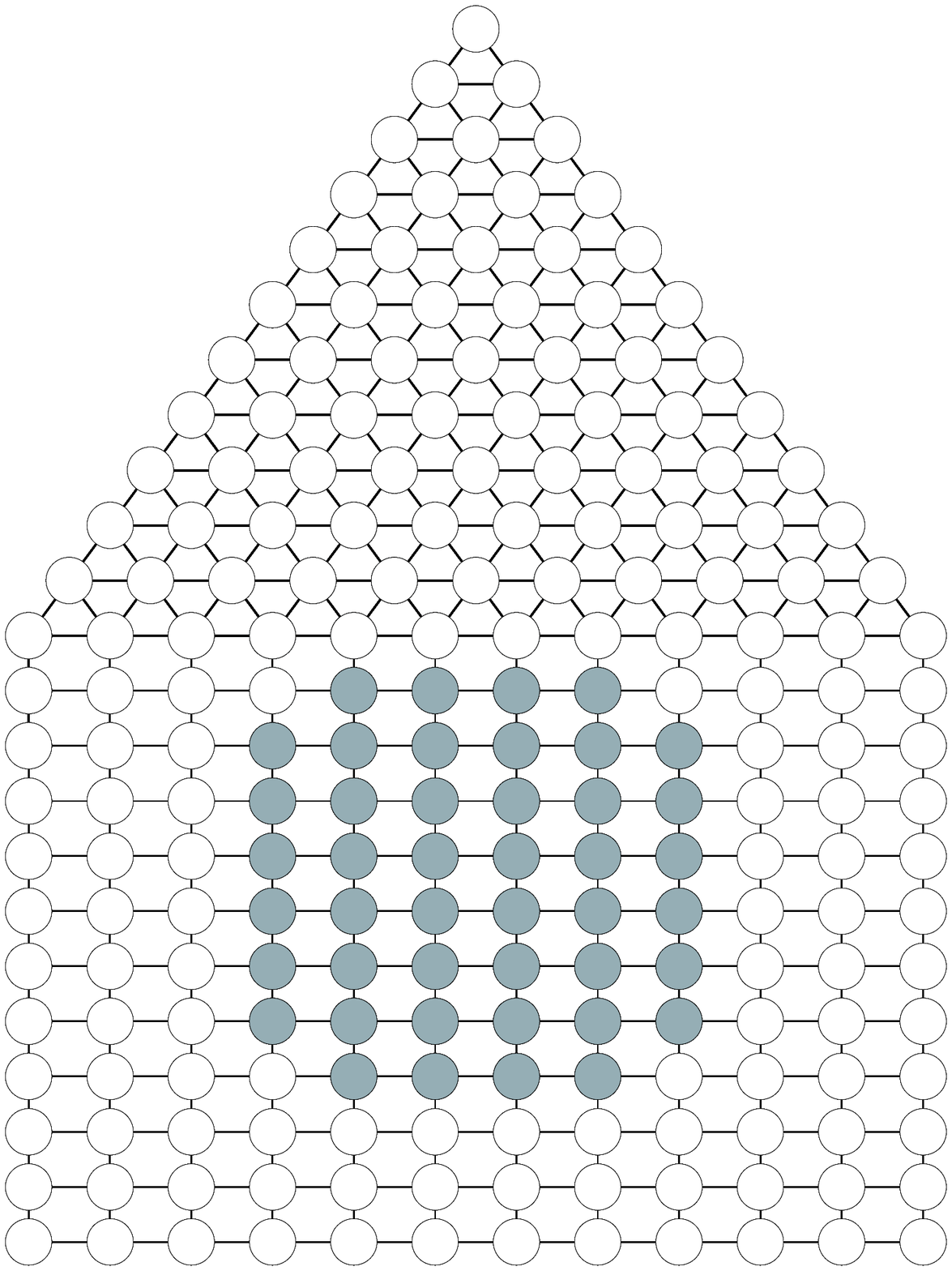}
\caption{An \ref{alg:MBO} evolution with $\tau=0.8$ on a graph consisting of adjoining regular lattices. The solution at chosen iterations is the characteristic function of the gray nodeset.
For the initial condition, given by the top left panel, the evolution reaches a steady state in 9 iterations. Iterations 3 (top right), 6 (bottom left), and 9 (bottom right) are shown. This example strengthens the `rule of thumb' that it is easier for a solution to become pinned on nodes with smaller degree.  See Section~\ref{sec:adj}. }
\label{fig:adj}
\end{center}
\end{figure}

\begin{figure}[t!]
\begin{center}
\includegraphics[width=.45\textwidth]{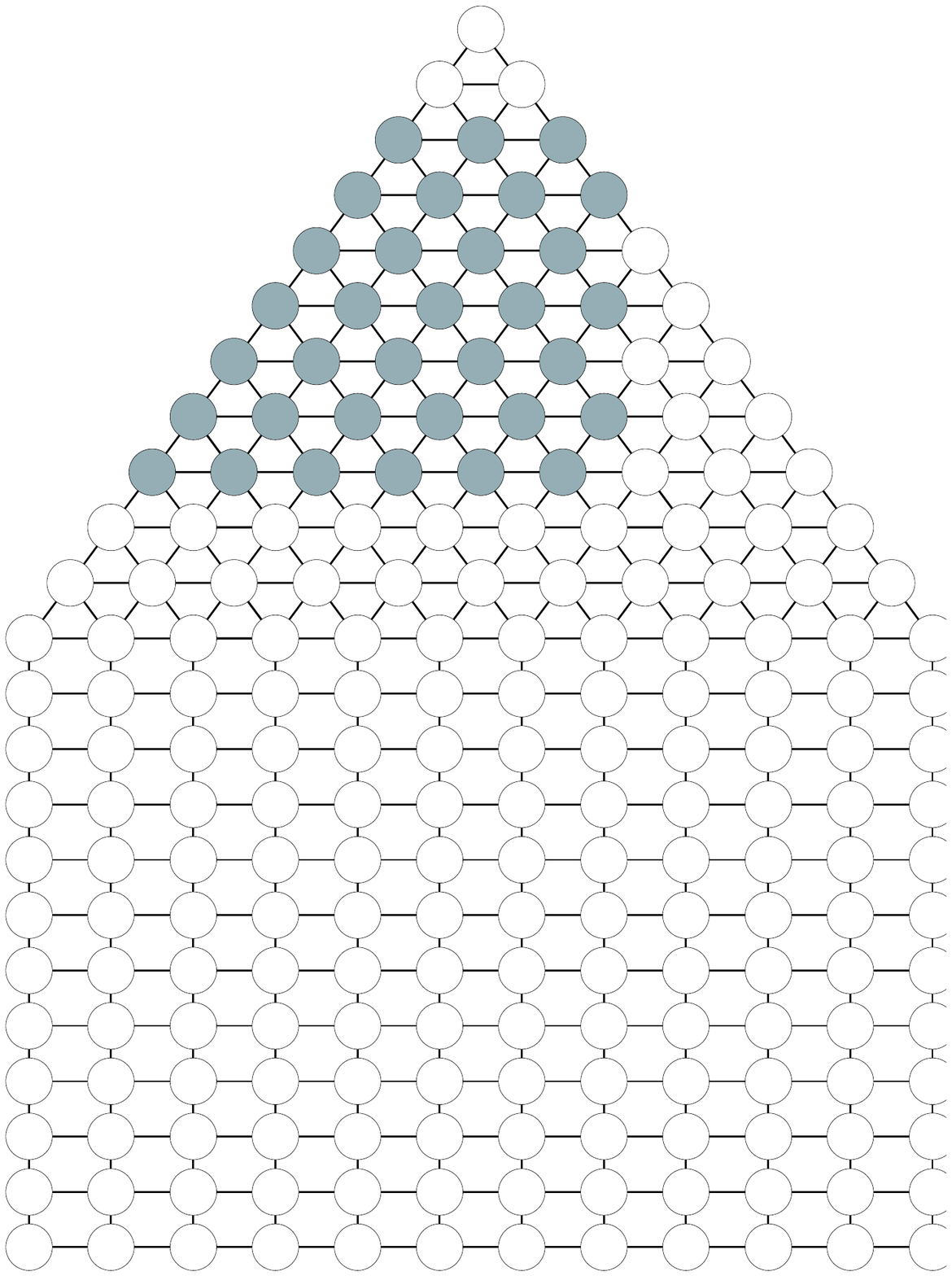} \qquad
\includegraphics[width=.45\textwidth]{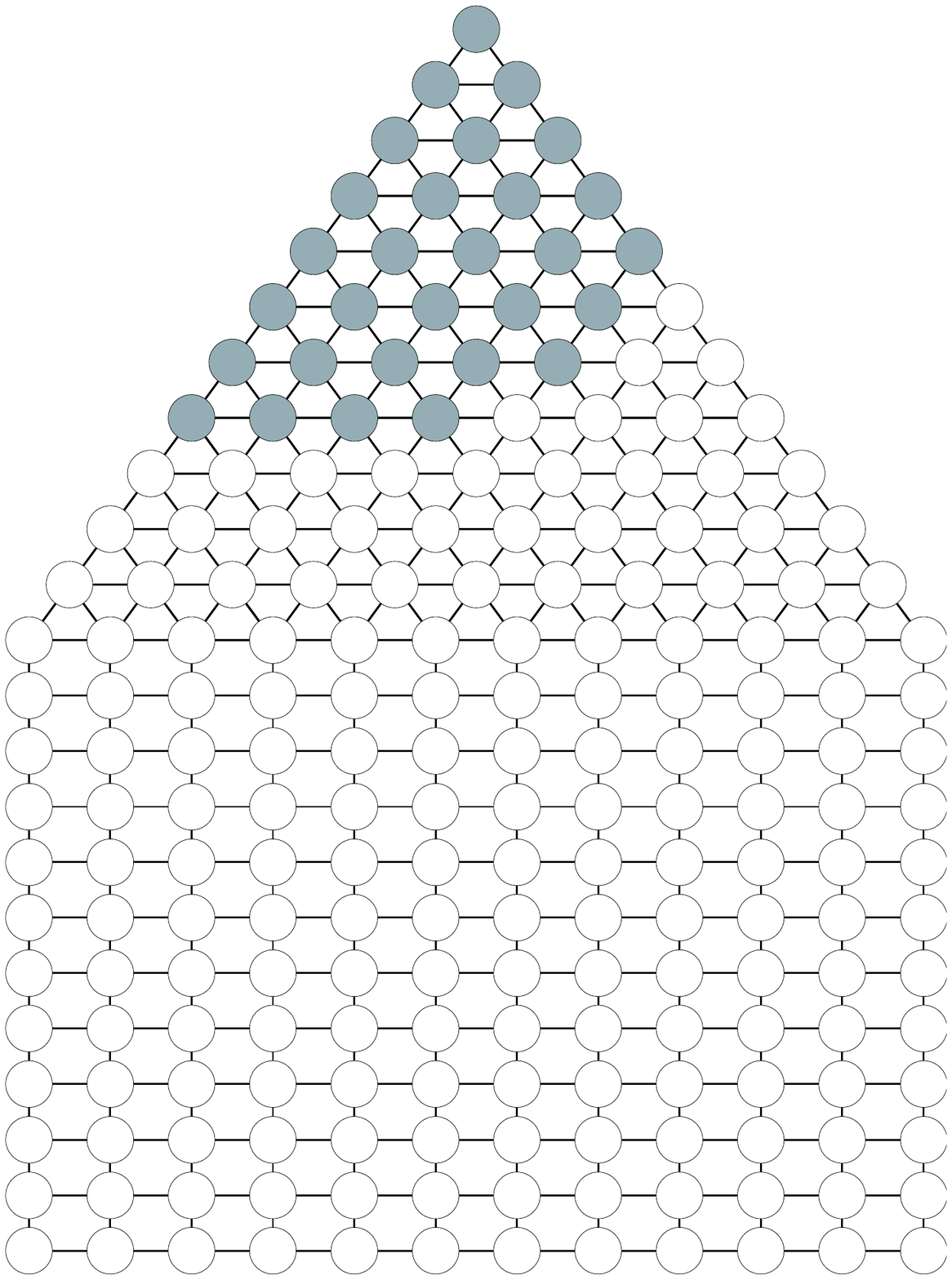}
\includegraphics[width=.45\textwidth]{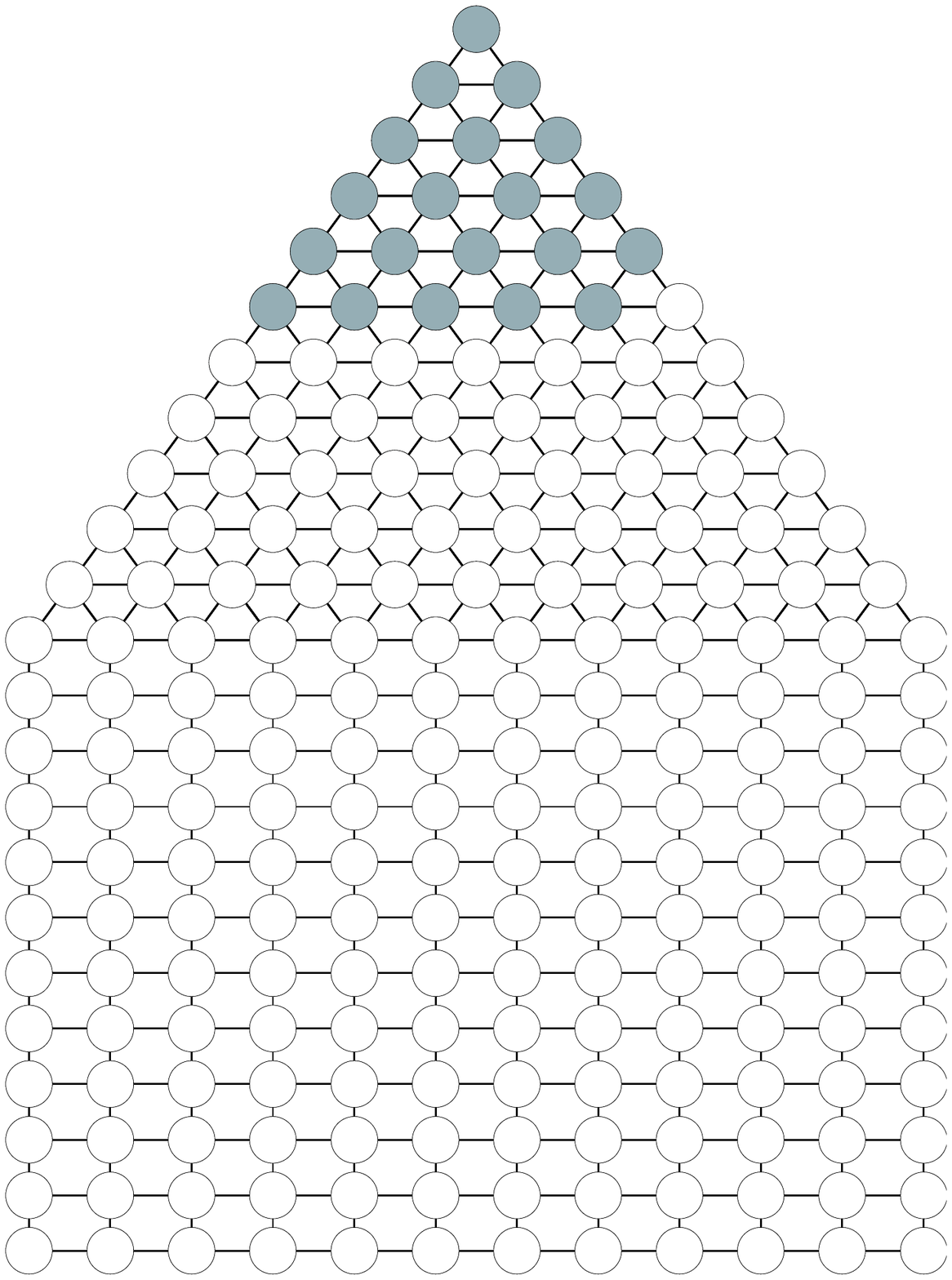} \qquad
\includegraphics[width=.45\textwidth]{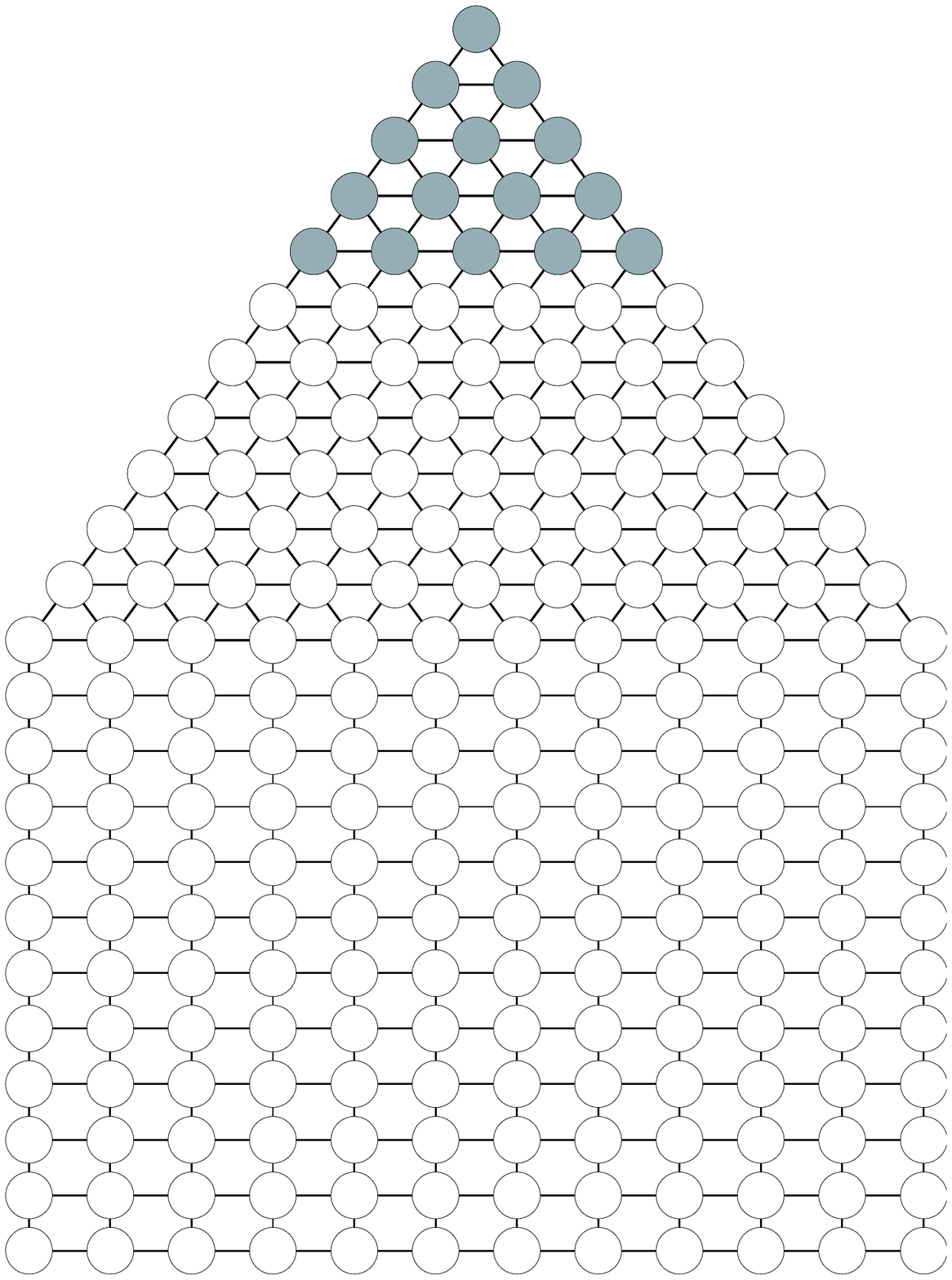}
\caption{An \ref{alg:MBO} evolution with $\tau=0.9$ on a graph consisting of adjoining regular lattices. The solution at chosen iterations is the characteristic function of the gray nodeset.
For the initial condition, given by the top left panel, the evolution reaches a steady state in 13 iterations. Iterations 4 (top right), 9 (bottom left), and 13 (bottom right) are shown. This example strengthens the `rule of thumb' that it is easier for a solution to become pinned on nodes with smaller degree.  See Section~\ref{sec:adj}. }
\label{fig:adjb}
\end{center}
\end{figure}

We consider the graph which is composed by adjoining a square and triangular lattice. See  Figure~\ref{fig:adj}. We take $r=0$ and $\omega_{ij} = 1$ for $i\sim j$ and zero otherwise. Note that the degree of a node in the triangular lattice is 6 and the degree of a node in the square lattice is 4.

To test the intuition from the star graph (see Section~\ref{sec:star}) that it is easier for the solution to pin on nodes with smaller degree, we consider the initial condition given in the top left panel of  Figure~\ref{fig:adj}. The mass is initially distributed over both the square and triangular lattice sites. We consider the \ref{alg:MBO} evolution with $\tau=0.8$. The solution moves freely on the lattice sites with degree $>4$, {\it i.e.}, on the triangular lattice. However, on the square lattice, the solution only `rounds corners'.

The nodes on the `border' of the graph (where the regular lattice was cut) have smaller degree. In Figure~\ref{fig:adjb}, we demonstrate that the solution can also be pinned on the border. Again, the initial condition is given in the top left panel. In this simulation, we take $\tau=0.9$. Away from the boundary, the solution set can again shrink freely. However, the solution becomes pinned on the border.



\subsection{Two moons graph}\label{sec:moons}

\begin{figure}[t!]
\begin{center}
\includegraphics[width=.45\textwidth]{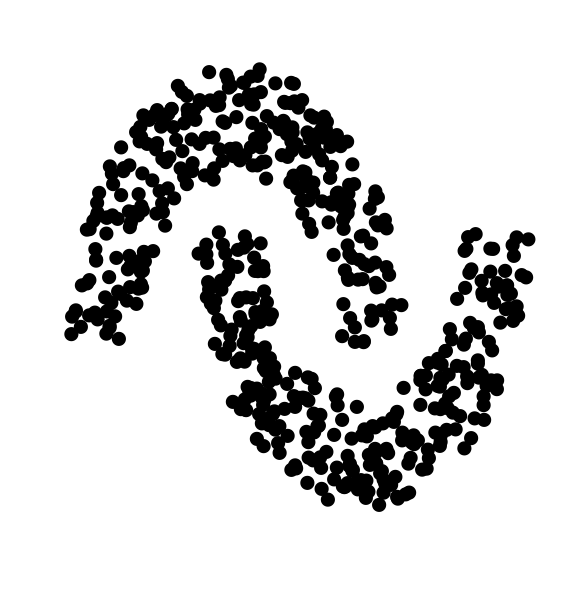}
\includegraphics[width=.45\textwidth]{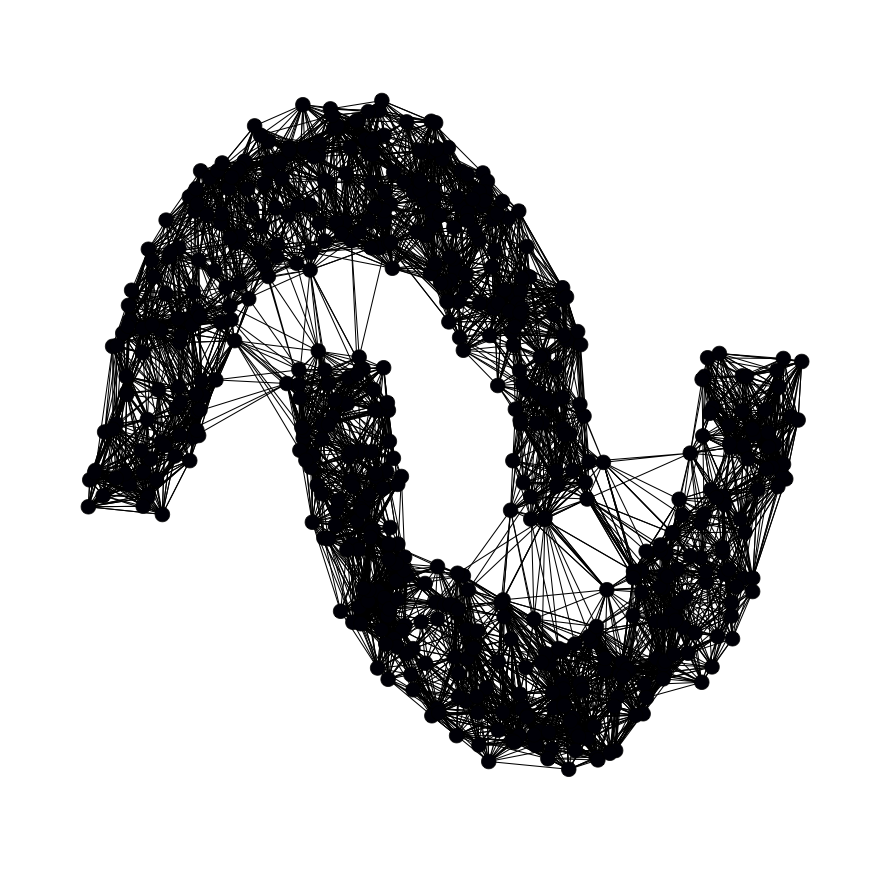} \\
\includegraphics[width=.45\textwidth]{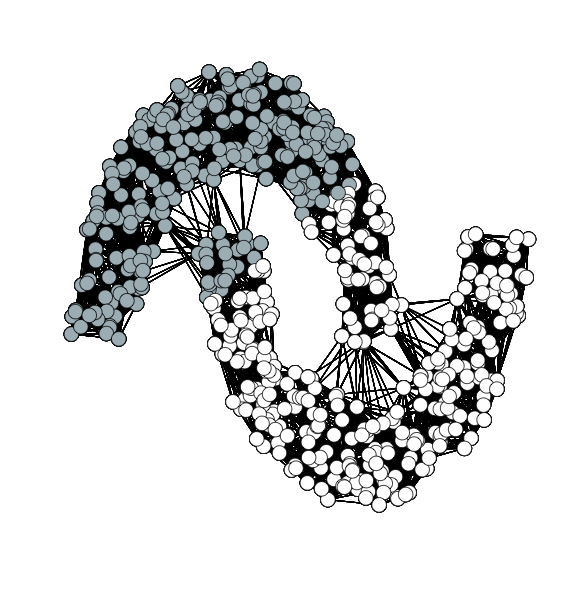}
\includegraphics[width=.45\textwidth]{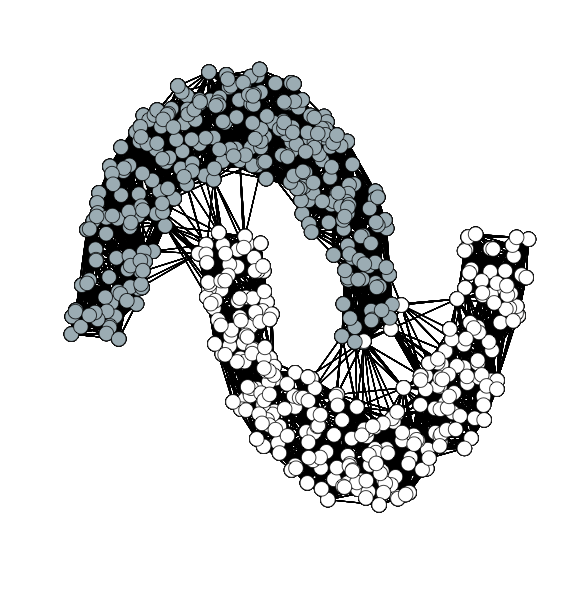}
\caption{(top) Construction of the two moons graph. (top left) Some random sampling of two moons. (top right) The connectivity of the graph resulting from connecting nearest neighbors after adding high dimensional noise. (bottom) An \ref{alg:MBO} evolution for $\tau=5$, starting with initial condition on the left and terminating at the stationary solution on the right in 9 iterations.  See Section~\ref{sec:moons}. }
\label{fig:moons}
\end{center}
\end{figure}

In this last example, we consider a graph which is widely used as a benchmark problem for partitioning algorithms. Our construction of the graph follows \cite{Buhler2009}. 
 The graph is generated by first randomly distributing 600 points in a region described by two half arcs in $\mathbb R^2$---referred to as ``two moons''. See Figure~\ref{fig:moons} (top left).  The points are then embedded in $\mathbb R^{100}$ and randomly perturbed by i.i.d. Gaussian noise with mean zero and standard deviation $\sigma = 0.1$. Let $k=10$. The edge weights are chosen to be
$$
w_{ij} = \max \{ s_i(j) , s_j(i) \}, \quad \text{where } s_i(j) = e^{- \frac{4}{d_i^2} \| x_i - x_j \|^2},
$$
and $d_i$ is the Euclidean distance between $x_i$ and its $k$-th nearest neighbor. We then take the symmetrized $k$-nearest neighbors graph. This is given in Figure~\ref{fig:moons} (top right).

We consider the \ref{alg:MBO} evolution with $\tau=5$ and initial condition as shown in Figure~\ref{fig:moons} (bottom left). After 9 iterations, the \ref{alg:MBO} evolution converges to the state in Figure~\ref{fig:moons} (bottom right).

We want to stress that the two moons example is meant as an illustration of the \ref{alg:MBO} algorithm on a more complex toy graph. In this paper we do not aim to compete in terms of accuracy or efficiency with existing clustering methods, hence we will not focus on those aspects of the two moons example.

\section{Discussion and open questions}\label{sec:conjectures}

    Motivated by curvature flows in continuum mechanics, we described several analogous processes on graphs. In particular we used the graph total variation, or graph cut, to define curvature on graphs, which we then related to the graph Allen-Cahn equation, graph MBO scheme, and graph mean curvature flow. The continuum intuition for these problems suggests many results, some of which we proved in this paper, some which we have shown cannot hold on a graph because of the lack of infinitesimal length scales, and some which we state below as, still unproven, open questions.

        In a sense to be made precise, for a suitable choice of $\tau$ (not too small, not too large, depending on $\omega$, most likely depending on the graph's spectrum), the dynamics of \ref{alg:MBO} are expected to approximate those of graph MCF. 

    \begin{question}[MBO and graph Mean Curvature Flow]\label{con:MBOMCF} Is there an interval of $\tau$ (depending on $\eth t$), for which a single \ref{alg:MBO} iteration minimizes the \eqref{eq:MCFh} functional $\mathcal{F}$ from \eqref{eq:MCFfunctional}? For such a $\tau$, the graph mean curvature flow \eqref{eq:MCFh} would coincide with the \ref{alg:MBO} scheme (up to a time rescaling).
    \end{question}
    An approach to Question~\ref{con:MBOMCF}, uses the Taylor series expansion for the solution of the graph heat equation:
    \[
    e^{-t\Delta} \chi_S = \sum_{k=0}^\infty \frac1{k!} \left(-t \Delta \right)^k  \chi_S,
    \]
    for $S\subset V$. Hence, we can rewrite the Lyapunov functional $J$ from \eqref{eq:Lyapunov} as
\begin{align*}
&J(\chi_S) = \langle 1-\chi_S, \chi_S - \tau \Delta \chi_S\rangle_{\mathcal V} + R_S(\tau),\\
&\text{where } R_S(\tau) := \sum_{k=2}^\infty \frac{(-\tau)^k}{k!} \langle \chi_{S^c}, \Delta^k \chi_S\rangle_{\mathcal{V}}.
\end{align*}
Using $\langle 1-\chi_S, \chi_S\rangle_{\mathcal V} = 0$, $\langle 1, \Delta \chi_S\rangle_{\mathcal V} = 0$, and \eqref{eq:TVusingDelta}, we find
\[
J(\chi_S) = \tau \TVa^1(\chi_S) + R_S(\tau).
\]
    This connection between the Lyapunov functional $J$ and the total variation, and hence the MCF functional $\mathcal{F}$ from \eqref{eq:MCFfunctional}, strengthens the plausibility of a positive answer to Question~\ref{con:MBOMCF}. A more difficult question, which could be of great use in numerical problems, is how we can estimate the number iterations of \ref{alg:MBO} needed to go from some initial data to a minimizer of the Ginzburg-Landau functional or graph cut functional.

    \begin{question}[Minimizing graph cut] For any, a priori specified, approximation error, is there a local, quantitative bound on the number of iterations of \ref{alg:MBO} needed to approximate a minimizer of the graph cut functional $\TVa^q$ up to the specified accuracy? ``Local'' means here that the bound does not rely on the spectrum of the graph, but instead uses quantities such as graph curvature $\kappa^{q,r}$ or the total variation $\TVa^q(\chi_N)$ for some local graph neighborhood $N\subset V$. The analogous question can be asked for \eqref{eq:MCFh}.
    \end{question}

    In Theorem~\ref{thm:localtau} it was shown that, if the curvature at a given node is sufficiently large and the time step $\tau$ in \ref{alg:MBO} is chosen in the right interval, then the value at the node will change in one \ref{alg:MBO} iteration. The next question is the analogous statement for the Allen-Cahn equation,  \eqref{eq:ACEe}.

    \begin{question}[Non-freezing for Allen-Cahn]\label{conj:ACcurvature}  Let $\e$ be in some positive interval and let $u^\e$ be a solution of the Allen-Cahn equation \eqref{eq:ACEe} for this choice of $\e$. Suppose that the curvature $(\kappa_S^{1,r})_i$ of $S=\{j\in V: u^\e_j(t_0)\leq 0\}$ at  a node $i\in S$, is sufficiently large (possibly depending on $\e$). Is there is some interval of positive times such that $u^\e_i(t_0+h)>0$ for $h$ in this interval?
    \end{question}

    Because \eqref{eq:ACEe} is derived from the graph functional \eqref{eq:GL functional}, we suspect that the correct curvature in Question~\ref{conj:ACcurvature} is $\kappa_S^{1,r}$, the curvature related to the anisotropic functional $\frac12 \TVa^1$, which was identified as the $\Gamma$-limit of \eqref{eq:GL functional} for $\e\to 0$ in earlier work \cite{vanGennipBertozzi12}, and not the curvature which can be derived from the isotropic total variation functional $\TV$ as the continuum case might suggest at first glance. Since we have seen that pinning occurs for small enough $\e$, full convergence is not expected here, but the numerical examples of Section~\ref{sec:examples} suggest an approximate result for small $\e$ is feasible.

    \begin{question}[Allen-Cahn and graph Mean Curvature Flow]\label{ques:ACMCF} Is there an $\e>0$ such that, given the solution $u^\e$ to \eqref{eq:ACEe} for  some $\epsilon > 0$, there is an increasing sequence of times $t_n$ for which either the sets $S_n:= \{ j\in V: u^\e_j(t_n)\leq 0 \}$ or the sets $S_n:= \{ j\in V: u^\e_j(t_n)\geq 0 \}$ form a solution to the graph MCF?
	
	Furthermore, among sequences with this property is there exactly one sequence $\{t_n\}$ that is maximal in the following sense: there exists no sequence $\{t'_n\}$, of which $\{t_n\}$ is a strict subsequence, such that $\{S_{t'_n}\}_n\not\subset \{S_{t_n}\}_n$ and $\{S_{t'_n}\}$ is still a solution to the graph MCF?
    \end{question}

    A different question is how the graph MCF behaves in the continuum limit, when it is formulated on a sequence of graphs which are ever finer discretizations of some continuum space. We expect that it should give back the usual MCF in the continuum limit, or some anisotropic MCF, as the convergence results in \cite{vanGennipBertozzi12} show the final limit could crucially depend on the scaling in $\eth t$ and the discretization parameter (which will show up in the graph weights). This question is similar (and perhaps equivalent) to the convergence of discretization schemes for the usual MCF. Similar questions can be asked about the graph Allen-Cahn equation and graph MBO scheme.

    \begin{question}[Stability of graph MCF, MBO, and ACE, in the continuum limit] Suppose we are given any sequence of graphs $(V^k,\omega^{k}_{ij})$, $k \in \mathbb{N}$, converging in the Gromov-Hausdorff sense to a Riemannian manifold $(M,g)$. Is there a fixed time interval such that, as $k \to \infty$, any sequence generated by \ref{alg:MBO} with $\tau$ in this interval converges to a sequence generated by the (possibly anisotropic) continuum MBO algorithm in $M$ (with the Laplacian induced by $g$)? Accordingly, do solutions of \eqref{eq:ACEe} converge to solutions of the (possibly anisotropic) continuum Allen-Cahn equation in $M$, and do solutions to \eqref{eq:MCFh} converge to viscosity solutions (via the level set formulation) of (possibly anisotropic) continuum MCF in $M$, with initial data given by the limit of the initial data in each $V^n$? \end{question}

    As explained in Appendix \ref{sec:continuum}, MCF is closely related to certain models of continuum phase transitions, particularly Allen-Cahn and Ginzburg-Landau dynamics. However, another important connection with statistical mechanics involves  Ising models and other interacting particle systems, which are known to converge in the mesoscopic limit to flow by mean curvature. In work of Katsoulakis and Souganidis
    \cite{KatsSou1994,KatsSou1997} convergence to a viscosity solution of MCF is first proved. See also the related work of Funaki and Spohn
    \cite{FunSpo1997} where MCF is derived as a deterministic limit of stochastic Ginzburg-Landau dynamics. On the other hand, there is vast literature concerned with (for instance) the Ising model (and its generalizations) on graphs \cite{Lyons1989,Lyons00}, see also Durrett's book
    \cite{Dur2007}. This suggests the following question.

    \begin{question}[Possible probabilistic interpretations of graph MCF, MBO, and AC] Is \eqref{eq:MCFh} related to an interacting particle system on the underlying graph? Also, are there interacting particle systems or stochastic processes in $V$ that are related to \ref{alg:MBO} or  \eqref{eq:ACEe}?
    \end{question}

    Finding such a system would partly resolve the issue that a front moving on a graph in continuum time necessarily does so in a way that, from a continuum point of view, looks discontinuous (as discussed previously in this paper, in particular in Sections~\ref{sec:MCF} and~\ref{sec:AC}.), as the particle dynamics would be continuous in time and stochastic. 
    The convergence results in \cite{KatsSou1994,KatsSou1997,FunSpo1997} show that the above question has an a priori higher chance of having a positive answer for a large graph, as it already holds in the continuum limit. An interesting direction would be to investigate the relation between such probabilistic interpretations and 'spread of information' dynamics such as \emph{bootstrap percolation} \cite{ChalupaLeathReich79}, \emph{gossip algorithms} \cite{Sha08}, and \emph{replicator dynamics} \cite{GhoshLermanSurachawalaVoevodskiTeng11,GhoshLerman12,SmithLermanGarcia-CardonaPercusGhosh13}.

	\subsection*{Acknowledgements}
This work was supported by the W. M. Keck Foundation, UC Lab Fees Research Grant 12-LR-236660, ONR grants N000141210838 and N000141210040, AFOSR MURI grant FA9550-10-1-0569, and NSF grants DMS-1118971, DMS-0968309, and DMS-0914856.	
Yves van Gennip did the bulk of the work for this paper while at UCLA. Braxton Osting is supported in part by a National Science Foundation (NSF) Postdoctoral Fellowship DMS-1103959. Nestor Guillen is supported in part by National Science Foundation (NSF) grant NSF-DMS-1201413.
We would like to acknowledge helpful  conversations with Milan Bradonji\'c, J\'er\^ome Darbon, Selim Esedo\=glu, Inwon Kim, Tijana Kosti\'c,  Stanley J. Osher, and Richard Tsai.

\begin{appendices}
\section{The continuum case}\label{sec:continuum}
In this appendix, we briefly review and provide references for the Allen-Cahn equation, the MBO algorithm, and mean curvature flow in the continuum setting.

\subsection{The continuum Allen-Cahn equation}\label{sec:continuumAC}

    The Allen-Cahn equation is a reaction-diffusion equation, given by
    \begin{equation}\label{eq:continuousAC}
    	  u_t = \Delta u + f(u),
    \end{equation}
    where $u\colon \mathbb{R}^n \times \mathbb{R}_+ \to \mathbb{R}$ and $\Delta$ is the standard Laplacian (although other linear elliptic operators can be considered as well), and $f$ is a non-linear function of the form $f=-W'$ where $W\colon \mathbb{R} \to \mathbb{R}$ is a double well potential with two global minima.  For simplicity, take $W(u) = (u+1)^2(u-1)^2$, where the minima are at $\pm 1$.

    A question which is always of interest is understanding the way that solutions to \eqref{eq:continuousAC} converge to equilibrium. For each fixed $x$, one expects that $u(x,t)$ approaches either $1$ or $-1$, as $t \to +\infty$, as these values correspond to the minima of $W$. This indicates that for very large $t$ the function $u$ defines two regions of $\mathbb{R}^n$, where it is very close to either $1$ or to $-1$, separated with a smooth transition layer in between.

    This asymptotic behavior is well understood nowadays.  Rescaling $(x,t)$ as  $(\frac{x}\e,\frac{t}{\e^2})$, we obtain the equation
    \begin{equation}\label{eq:continuousACrescaled}
    	 u^\e_t = \Delta u^\e +\e^{-2}f(u^\e).
    \end{equation}

    Note that for very small $\e$ the function $u^\e$ describes the long time behavior of the original $u$. Then, it is well known  (see  \cite{BarlesSonerSouganidis93,BronsardKohn91} for background and discussion) that, as $\e \to 0^+$, the solutions $u^\e(x,t)$ converge to a function which takes the value $-1$ in some set $S_t$ (depending on time) and takes the value $1$ in $S_t^c$. Here $S_t$ is a set whose boundary is evolving by mean curvature flow (see Section~\ref{subsec: continuum mcf}).

    Although the original motivation for studying \eqref{eq:continuousAC} was phase transitions, it is also the gradient flow of the Ginzburg-Landau functional. Precisely, equation \eqref{eq:continuousACrescaled} is the $L^2$ gradient flow of the functional
    \begin{equation}\label{eq:ContinuousGLRescaled}
    	 GL_\e(u) := \int \frac{\e}{2}|\nabla u|^2+\frac{1}{\e}W(u)\;dx.
    \end{equation}
    It is expected that solutions to \eqref{eq:continuousACrescaled}  converge to a local minimum of this functional, as $t \to +\infty$, thus schemes for \eqref{eq:continuousACrescaled}  could be used for approximating minima of \eqref{eq:ContinuousGLRescaled}. This is the application that serves as the biggest motivation in the graph setting.

For more information about reaction-diffusion equations with a polynomial nonlinearity we refer to \cite[Section 1.1]{Temam97}.

\subsection{Continuum mean curvature flow}\label{subsec: continuum mcf}

    Mean curvature flow (MCF) consists of the evolution of a closed, oriented hypersurface $\Sigma_t \subset \mathbb{R}^d$ over time, such that the inner normal velocity at a given point of $\Sigma_t$ is equal to the mean curvature of $\Sigma_t$ at that point. The study of such a flow has been greatly motivated by phase transition models in crystal growth and materials science, in particular since the important work of Allen and Cahn \cite{AllCah1979}. Starting with the seminal work of Brakke \cite{Bra1978}, the mathematical study of this flow has been vast, and has involved areas of mathematics ranging from differential geometry to stochastic control. The use of MCF is now widespread in the modeling of moving fronts \cite{CagSoc1994}. The reason why MCF is so ubiquitous in the phase transitions literature, is that many singular limits of reaction diffusion equations ({\it i.e.,} singular limits of Ginzburg-Landau dynamics) converge to motion by mean curvature. See \cite{BronsardKohn91,Peg1989,BarlesGeorgelin95} for precise convergence theorems and further discussion.

    A well known feature of MCF is both the formation of singularities and the occurrence of topological changes, regardless of the smoothness of the initial data. A significant portion of the literature on MCF deals with notions of weak solutions, the first of which goes back to Brakke \cite{Bra1978}. Partial regularity for weak solutions as well as regularity up to the first singular time have been widely studied \cite{Eck2004}.

    An equivalent formulation of the flow looks not only at the hypersurface $\Sigma_t$, but at the entire domain $\Omega_t$  bounded by it, so that $\partial \Omega_t = \Sigma_t$. Accordingly, it is said that $\Omega_t$ itself is evolving by mean curvature flow. This perspective is natural for phase transitions.

    Let $\phi(t,\cdot)\colon \mathbb R^{d} \rightarrow \mathbb R$ be the signed distance function to the set $\Omega$ at time $t$. From the level set method perspective \cite{Osher2003}, the motion by mean curvature\footnote{In the literature two related, but different, concepts of mean curvature appear. One corresponds with the factor $\text{div} \frac{\nabla \phi}{|\nabla \phi|}$ in \eqref{eq:MCF}, the other has a normalization factor $\frac1{d-1}$, where $d$ is the dimension of the space. This normalization by the dimension of the hypersurface justifies the ``mean'' part of ``mean curvature''.} of $\Omega_t$ corresponds to an initial value problem for a fully non-linear degenerate parabolic equation,
    \begin{align}
         &\phi_{t} = F(D^2\phi,\nabla \phi),  \quad  \phi(\cdot,0) = \phi_{0},\notag\\
         &\text{where } F(D^2\phi,\nabla \phi) = -|\nabla \phi|\text{div} \frac{\nabla \phi}{|\nabla \phi|}.\label{eq:MCF}
    \end{align}
    Then, when there is a smooth solution $\phi(x,t)$, the domains given by $\Omega_t := \{ \phi(\cdot,t)<0\}$ will be evolving by mean curvature flow and will start from the original domain $\Omega$. In general, even for an initial domain with a smooth boundary, a smooth solution might not exist for all times, and one must  work with viscosity solutions. In this context, the convergence of the MBO scheme \eqref{eq:MBO} (explained in Section~\ref{sec:continuumMBO}) to such viscosity solutions was proved by Evans \cite{Evans93}.

    It is worth remarking that Soner and Touzi in \cite{SonTou2003} interpret MCF as a stochastic control problem. In this interpretation, one controls a Brownian motion for which one is allowed to turn off diffusion in one given direction. The surface $\Sigma_t$ in this case arises as the set of points that can be reached with probability $1$. This probabilistic interpretation is quite different from those mentioned in the discussion at the end of Section \ref{sec:conjectures}.

    Finally, given the affinity with the graph setting, it is worthwhile to comment briefly on the more recent nonlocal mean curvature flow. Caffarelli and Souganidis \cite{CaffarelliSouganidis10} arrive at this flow by following a nonlocal and continuum analogue of \ref{alg:MBO}, where instead of using the Laplacian one uses a fractional power of the Laplacian $(-\Delta)^s$ with $s\in(0,1/2)$. A level set formulation based on viscosity solutions was developed later by Imbert \cite{Imb2009}.

\subsection{The continuum MBO algorithm}\label{sec:continuumMBO}
The Merriman, Bence, and Osher (MBO) algorithm \cite{MBO1992,MBO1993,MerrimanBenceOsher94}, also known as the threshold dynamics algorithm, approximates the dynamics of mean curvature flow \eqref{eq:MCF} by alternatively applying diffusion and thresholding operators. Let $\chi(t,\cdot)$ be the characteristic function of the set $\Omega_t$ at time $t$. Define the diffusion operator $\chi_{0} \mapsto u(t,\cdot) := e^{t \Delta } \chi_{0}$ to be the solution of the initial value problem
 \begin{align*}
 \dot{u} = \Delta u,  \quad  u(0) = \chi_{0}(\cdot).
 \end{align*}
Define the threshold operator
\begin{align*}
Pu(x) = \begin{cases}
1 & \ u(x) \geq \frac{1}{2} \\
0 & \ u(x) < \frac{1}{2}
\end{cases}.
\end{align*}
The MBO evolution of a set described by $u$  at time $T$  can then be succinctly written
\begin{align}
\label{eq:MBO}
\chi(T,\cdot) = \left(P e^{\tau \Delta} \right)^{k} \chi_{0}, \qquad \text{where } \tau = T/k
\end{align}
is the `time step' and $k$ is a parameter. In \cite{Evans93,BarlesGeorgelin95} convergence of the MBO algorithm to motion by mean curvature, defined in \eqref{eq:MCF}, as $k\uparrow \infty$, is proven.

The MBO scheme and its implementation has evolved considerably since its original proposal. We provide a short, non-exhaustive, overview here. In \cite{Mascarenhas92,ruuth1998}, the MBO scheme was extended to multiple-phase problems. In \cite{Ruuth96,ruuth1998b}, a spectral discretization of the MBO scheme for motion by mean curvature was proposed, which is much more efficient then finite difference approaches. This approach can be applied to both two-phase and multi-phase problems. Convergence for an anisotropic variant of the MBO scheme was proven in \cite{ChambolleNovaga06}, and in \cite{Esedoglu2008} diffusion generated motion was applied to higher order geometric motions. In \cite{EsedogluTsai2006}, the MBO scheme was extended to a thresholding method for approximating the evolution by gradient descent of the Mumford-Shah functional and applied to image segmentation problems. In \cite{ruuth2010diffusion} the authors study MBO-like schemes which use the signed distance function. Recent work \cite{esedoglu2013} presents new algorithms for multiphase mean curvature flow, based on a variational description of the MBO scheme.

It is well-known that, in a finite difference scheme for the MBO algorithm, the time step $\tau$ (equivalently $k$) in \eqref{eq:MBO} must be chosen carefully and in the limit as $k\uparrow \infty$, the discretized MBO evolution is stationary. In fact, when discussing the numerical implementation of the algorithm on a discrete grid, Merriman, Bence, and Osher \cite{MBO1992} observe:
  \begin{quote}
 ``The basic requirement is that [the time step, $\tau$,] be short enough so that the local analysis \ldots is valid, but also long enough so that the boundary curve moves by at least one grid cell on the spatial grid (otherwise the curve would be stuck).''
  \end{quote}
They derive heuristic upper and lower bounds on the time step, $\tau$, for the algorithm to approximate motion by mean curvature.

\section{Calculation of the first variation for graph total variation}\label{sec:formalfirstvar}

In \eqref{eq:variationofTV} we computed $\left.\frac{d}{dt}\right|_{t=0}\TVa^q(u+tv) = \langle \sgn(\nabla u),\nabla v\rangle$ using the convexity of $\TVa^q$.
In this section we generalize this fact to other kinds of graph total variation, which are expressible as
\begin{equation}\label{eqn: def general TV}
\mathcal{TV}(u) := \max\{ \langle \dvg \varphi, u\rangle_{\mathcal{V}} \colon  \varphi\in \mathcal{A}\},
\end{equation}
where $\mathcal{A}\subset \mathcal{E}_e$ is some admissible set of edge functions\footnote{Remember from Remark~\ref{rem:isotropicTV} that functions in $\mathcal{E}_e$ need not be skew-symmetric.
}, such that a maximizer $\varphi^u \in \mathcal{A}$ exists (even if it might not be unique). The key fact is that such a $\mathcal{TV}(u)$ is convex and might be studied via convex analysis. The convexity of $\mathcal{TV}$ is evident from its definition: $u \to \mathcal{TV}(u)$ is a scalar valued function given as the maximum of a family of linear functions
$u \mapsto \langle \dvg \varphi,u\rangle_{\mathcal{V}}$.	
Let us recall some concepts from convex analysis \cite[Chapter 1, Section 5]{EkelandTemam76}, in particular, the subdifferential of $\mathcal{TV}$ at $u$. This set valued function is denoted by $\partial \mathcal{TV}(u)$ and given by
\begin{align*}
\partial \mathcal{TV}(u) &:= \{ v \in\mathcal{V} : \mathcal{TV}(u) < \infty\\
 &\hspace{0.7cm}\text{ and }  \forall\; w\in\mathcal{V}\;\;  \mathcal{TV}(w)\geq \mathcal{TV}(u)+\langle v,w-u\rangle_{\mathcal{V}}\}.
\end{align*}
That is, $v \in \mathcal{TV}(u)$ if and only if it is the slope of an affine function which is tangent to the graph of $\mathcal{TV}$ at $u$. In particular, at the points where $\mathcal{TV}(u)$ is differentiable $\partial \mathcal{TV}(u)$ consists of a single element: the gradient of $\mathcal{TV}(u)$ at $u$.

In the particular case of $\mathcal{TV}$, it follows that
\begin{equation}\label{eqn: subdifferential of TV}
\partial \mathcal{TV}(u) = \{ v \in \mathcal{V} : \langle v,u\rangle_{\mathcal{V}} = \mathcal{TV}(u)\}.
\end{equation}
Indeed, note that by \eqref{eqn: def general TV}, for any $w\in\mathcal{V}$,
$\mathcal{TV}(w) = \langle \dvg \varphi^w,w\rangle_{\mathcal{V}}$,	
where $\varphi^w \in \mathcal{A}$ is a maximizer in \eqref{eqn: def general TV}. It follows that, if  $u \in\mathcal{V}$ is given, then
\begin{equation*}
v \in \partial \mathcal{TV}(u) \Leftrightarrow \langle \dvg \varphi^w,w\rangle_{\mathcal{V}}\geq \langle \dvg \varphi^u,u\rangle_{\mathcal{V}}+ \langle \dvg v,w-u\rangle_{\mathcal{V}}.
\end{equation*}
By choosing $w=0$ and $w=2u$, respectively, we find
$\langle v,u\rangle_{\mathcal{V}} = \langle \dvg \varphi^u,u\rangle_{\mathcal{V}} =  \mathcal{TV}(u)$.
This proves the set identity \eqref{eqn: subdifferential of TV} for any $u \in \mathcal{V}$.

On the other hand, the convexity of $\mathcal{TV}(u)$ implies it is a locally Lipschitz function, making it differentiable for a.e.\footnote{Here a.e. is with respect to Lebesgue measure in $\mathcal{V}$ (which is uniquely determined by its inner product).} $u \in \mathcal{V}$ by Rademacher's theorem \cite[Chapter 6, Section 6.2, Theorem 2]{EvansGariepy92}. Therefore, for a.e. $u \in \mathcal{V}$, $\partial \mathcal{TV}(u)$ contains a single element $v=\dvg \varphi^u$. Then, for a.e. $u$, it follows that
\begin{equation*}
\left.\frac{d}{dt}\right|_{t=0}\mathcal{TV}(u+tv) = \langle \dvg \varphi^u,v\rangle_{\mathcal{V}}	
\end{equation*}



Now we can consider particular choices for $\mathcal{TV}$ and hence for $\mathcal{A}$. For $\mathcal{TV} = \TVa^q$, we have $\varphi^u = \varphi^a$ from \eqref{eq:varphiaTV}, as already explained in Remark \ref{rem: first variation for TVa q}. Similar computations can be done if we take $\mathcal{A}$'s corresponding to $\varphi^u$ given respectively by
\begin{align*}
\varphi^u(u) &= \varphi^{\mathcal{E}}(u) = \left\{\begin{array}{ll}\frac{\nabla u}{\|\nabla u\|_{\mathcal{E}}} & \text{if } \|\nabla u\|_{\mathcal{E}} \neq 0,\\ 0 & \text{if } \|\nabla u\|_{\mathcal{E}} = 0,\end{array}\right.
\quad \text{and}\\
\varphi_{ij}^u(u) &= \varphi_{ij}^{\TV}(u) = \left\{ \begin{array}{ll}\frac{(\nabla u)_{ij}}{|\nabla u|_i} & \text{if } |\nabla u|_i \neq 0,\\ 0 & \text{if } |\nabla u|_i=0,\end{array}\right.
\end{align*}
\textit{i.e.}, optimal $\varphi$'s for $\mathcal{TV}(u)=\|\nabla u\|_{\mathcal{E}}$\footnote{Note that we can write
$
\|\nabla u\|_{\mathcal{E}}=\max\{ \langle \dvg \varphi, u\rangle_{\mathcal{V}} \colon \varphi\in \mathcal{E},\,\, \|\varphi\|_{\mathcal{E}}\leq 1\}.
$}
and $\mathcal{TV}(u)=\TV(u)$, respectively (see Section~\ref{sec:setup}).
The previous analysis shows the first variations in these cases are given by the $\mathcal{V}$-inner product with	
\begin{align*}
\dvg \varphi^{\mathcal{E}}  &= \left\{ \begin{array}{ll} \frac{\Delta u}{\|\nabla u\|_{\mathcal{E}}} & \text{if } \|\nabla u\|_{\mathcal{E}} \neq 0,\\ 0 & \text{if }  \|\nabla u\|_{\mathcal{E}} = 0,\end{array}\right.\\
(\dvg \varphi^{\TV})_i &= \left\{\begin{array}{ll}\frac12 d_i^{-r} \Big[\sum_{j\in V\colon |\nabla u|_j\neq 0} \omega_{ij}^q \frac{(\nabla u)_{ji}}{|\nabla u|_j}&\\ \hspace{2.8cm} - \sum_{j\in V} \omega_{ij}^q \frac{(\nabla u)_{ij}}{|\nabla u|_i}\Big] & \text{if } |\nabla u|_i \neq 0,\\ \frac12 d_i^{-r} \sum_{j\in V\colon |\nabla u|_j\neq 0} \omega_{ij}^q \frac{(\nabla u)_{ji}}{|\nabla u|_j} &\text{if } |\nabla u|_i = 0.\end{array}\right.
\end{align*}
For the latter we can also write
\begin{align*}
(\dvg \varphi^{\TV})_i &= \frac12 d_i^{-r} \sum_{j\in V} \omega_{ij}^q \left(\frac{u_i-u_j}{|\nabla u|_j} + \frac{u_i-u_j}{|\nabla u|_i}\right)\\
 &= \frac12 d_i^{-r} \sum_{j\in V} \omega_{ij}^q \left(\frac1{|\nabla u|_j} + \frac1{|\nabla u|_i}\right) (u_i-u_j),
\end{align*}
where we have to remember that $\omega_{ij}^q \frac{u_i-u_j}{|\nabla u|_j}$ and $\frac1{|\nabla u|_j}$ are to be interpreted as $0$ whenever $|\nabla u|_j=0$ for any $j\in V$ (including $j=i$). Because the node function $\dvg \varphi^{\TV}$ is the first variation of the isotropic graph total variation, in the literature it is sometimes referred to as curvature or 1-Laplacian. In this paper we have argued why the use of the anisotropic total variation $\TVa$ to define curvature, as in \eqref{eq:dvgnu}, is a more natural choice on graphs.

\end{appendices}


\end{document}